\documentclass[11pt,a4paper,reqno]{amsart}
\usepackage[UKenglish]{babel}
\usepackage[T1]{fontenc}
\usepackage{verbatim}
\usepackage{palatino}
\usepackage{amsmath}
\usepackage{mathabx}
\usepackage{amsthm}
\usepackage{graphicx}
\usepackage{xcolor}
\usepackage{mathtools}
\usepackage{overpic}
\usepackage{enumerate}
\DeclarePairedDelimiter\ceil{\lceil}{\rceil}
\DeclarePairedDelimiter\floor{\lfloor}{\rfloor}
\usepackage[colorlinks = true, linkcolor={red!80!black}, citecolor = {blue!60!black}]{hyperref}

\author{Tuomas Orponen and Pablo Shmerkin}
\title[Projections, Furstenberg sets, and the $ABC$ problem]{Projections, Furstenberg sets, \\ and the $ABC$ sum-product problem}
\address{Department of Mathematics and Statistics\\ University of Jyv\"askyl\"a,
P.O. Box 35 (MaD)\\
FI-40014 University of Jyv\"askyl\"a\\
Finland} \email{tuomas.t.orponen@jyu.fi}
\address{Department of Mathematics\\
The University of British Columbia\\
1984 Mathematics Road, Vancouver, BC\\
Canada} \email{pshmerkin@math.ubc.ca}
\date{\today}
\thanks{T.O. is supported by the Academy of Finland via the project \emph{Incidences on Fractals}, grant No. 321896. }
\thanks{P.S. is supported by an NSERC Discovery Grant.}

\newcommand{\R}{\mathbb{R}}

\newcommand{\N}{\mathbb{N}}

\newcommand{\Z}{\mathbb{Z}}

\newcommand{\calH}{\mathcal{H}}

\newcommand{\spt}{\operatorname{spt}}
\newcommand{\Hd}{\dim_{\mathrm{H}}}
\newcommand{\Pd}{\dim_{\mathrm{p}}}

\newcommand{\diam}{\operatorname{diam}}

\newcommand{\dist}{\operatorname{dist}}

\newcommand{\m}{\mathfrak{m}}

\newcommand{\e}{\epsilon}

\def\Barint_#1{\mathchoice
          {\mathop{\vrule width 6pt height 3 pt depth -2.5pt
                  \kern -8pt \intop}\nolimits_{#1}}%
          {\mathop{\vrule width 5pt height 3 pt depth -2.6pt
                  \kern -6pt \intop}\nolimits_{#1}}%
          {\mathop{\vrule width 5pt height 3 pt depth -2.6pt
                  \kern -6pt \intop}\nolimits_{#1}}%
          {\mathop{\vrule width 5pt height 3 pt depth -2.6pt
                  \kern -6pt \intop}\nolimits_{#1}}}

\numberwithin{equation}{section}

\theoremstyle{plain}
\newtheorem{thm}[equation]{Theorem}
\newtheorem*{"thm"}{"Theorem"}

\newtheorem{lemma}[equation]{Lemma}

\newtheorem{cor}[equation]{Corollary}
\newtheorem{proposition}[equation]{Proposition}

\theoremstyle{definition}

\newtheorem{definition}[equation]{Definition}

\newtheorem{notation}[equation]{Notation}

\theoremstyle{remark}
\newtheorem{remark}[equation]{Remark}
\newtheorem{terminology}[equation]{Terminology}

\newcommand{\nref}[1]{(\hyperref[#1]{#1})}

\begin{document}

\begin{abstract}
We make progress on two interrelated problems at the intersection of geometric measure theory, additive combinatorics and harmonic analysis: the discretised sum-product problem, and the dimension of Furstenberg sets. Along the way, we obtain new information on the dimension of exceptional sets of orthogonal projections.

First, we give a new proof of the following asymmetric sum-product theorem: Let $A,B,C \subset \R$ be Borel sets with $0 < {\dim_{\mathrm{H}}} B \leq {\dim_{\mathrm{H}}} A < 1$ and ${\dim_{\mathrm{H}}} B + {\dim_{\mathrm{H}}} C > {\dim_{\mathrm{H}}} A$. Then, there exists $c \in C$ such that
\begin{displaymath} {\dim_{\mathrm{H}}} (A + cB) > {\dim_{\mathrm{H}}} A. \end{displaymath}
We use this to show that every $(s,t)$-Furstenberg set $F \subset \R^{2}$ associated with a line set of equal Hausdorff and packing dimension $t$ satisfies
\begin{displaymath} {\dim_{\mathrm{H}}} F \geq \min\left\{s + t,\tfrac{3s + t}{2},s + 1\right\}. \end{displaymath}
 \end{abstract}

\maketitle

\tableofcontents

\section{Introduction}

The purpose of this paper is to make progress on several interrelated problems at the interface of geometric measure theory, additive combinatorics, and harmonic analysis:
\begin{itemize}
\item The $ABC$ sum-product problem.
\item Exceptional set estimates for orthogonal projections.
\item The $(s,t)$-Furstenberg set problem and related incidence problems.
\end{itemize}
We will introduce these problems in dedicated subsections, and state our main results.

\subsection{Bourgain's projection and sum-product theorems}

All the results of this paper can be seen as generalisations and quantifications of two related, but different highly influential results of J.~Bourgain: the discretised sum-product and projection theorems \cite{Bo1, Bo2}. To put our results into context, we begin by stating these theorems.

The sum-product phenomenon predicts that if $A$ is a subset of ``intermediate size'' of a ring without sub-rings of ``intermediate size'', then either the sumset $A+A$ or the product set $A\cdot A$ has size substantially larger than that of $A$. The cases in which the ring is $\mathbb{R},\mathbb{Z}$ or a prime field $\mathbb{F}_p$, and size is measured by cardinality, have attracted considerable attention, see \cite{HRSZ25,MR4565644,Roche-Newton24,MR3474329,MR4069186,MR4469270} for a small selection of recent papers. Bourgain's sum-product theorem concerns the case in which $A\subset\R$ but ``size'' is measured by $\delta$-covering number: we denote the smallest number of $\delta$-balls needed to cover a set $X$ (in any metric space) by $|X|_{\delta}$. We remark that, after the introduction, $|X|_{\delta}$ will take on a slightly different, but for all purposes equivalent, meaning: see Definition \ref{def:coveringNumber}.

\begin{thm}[{Bourgain's discretised sum-product theorem \cite{Bo1}}] \label{thm:Bourgain-SP}
Given $s,t\in (0,1)$, there exists $\e=\e(s,t)>0$ such that the following holds for small enough $\delta > 0$.

Let $A\subset [1,2]$ be a set with $|A|_{\delta}\le\delta^{-t}$, satisfying the non-concentration assumption
\begin{equation} \label{eq:delta-s-set}
|A\cap B(x,r)|_{\delta} \le  \delta^{-\e}\cdot r^s |A|_{\delta} \quad \text{for all } x\in [1,2], \, r\in[\delta,1].
\end{equation}
Then
\[
\max\{|A+A|_{\delta},|A\cdot A|_{\delta}\} \ge |A|_{\delta}^{1+\e}.
\]
\end{thm}

Theorem \ref{thm:Bourgain-SP} is connected with the problem of relating the size of a planar set to the size of its orthogonal projections in a ``sparse'' set of directions. Given $\theta\in S^1$, we denote by $\pi_{\theta}:\R^2\to \mathrm{span}(\theta)$ the orthogonal projection onto the span of $\theta$.
\begin{thm}[{Bourgain's discretised projection theorem \cite{Bo2}}] \label{thm:Bourgain-Proj}
Given $s\in (0,1]$, $t\in (0,2)$ there exists $\e=\e(s,t)>0$ and $\delta_{0} = \delta_{0}(s,t) > 0$ such that the following holds for all $\delta \in (0,\delta_{0}]$.

Let $E\subset S^1$ and $K\subset B(0,1) \subset \R^{2}$ be sets with the following properties:
\begin{itemize}
\item[\textup{(i)}] $|E\cap B(\theta,r)|_{\delta}\le \delta^{-\e} \cdot r^{s}|E|_{\delta}$ for all $\theta\in S^1$ and $r\in [\delta,1]$,
\item[\textup{(ii)}] $|K|_{\delta}\le\delta^{-t}$ and $|K\cap B(x,r)|_{\delta} \le \delta^{-\e}r^{s}|K|_{\delta}$ for all $x\in \R^{2}$ and  $r\in[\delta,1]$.
\end{itemize}
Then there exists $\theta\in E$ such that
\[
|\pi_{\theta}(K')|_{\delta} \ge \delta^{-\e} |K|^{1/2}_{\delta},\qquad K'\subset K, \, |K'|_{\delta} \ge\delta^{\e}|K|_{\delta}.
\]
\end{thm}

Theorem \ref{thm:Bourgain-Proj} has, as a corollary, the following Hausdorff dimension version:
\begin{cor}[\cite{Bo2}] \label{cor:Bourgain-Proj}
Given $s\in (0,1]$, $t\in (0,2)$, there exists $\e=\e(s,t) > 0$ such that the following holds. Let $K\subset \R^2$ be a Borel set with $\Hd(K)=t$. Then,
\[
\Hd\big\{ \theta \in S^{1} : \Hd \pi_\theta (K) < \tfrac{1}{2}\Hd(K)+\e \big\} \le s.
\]
\end{cor}

\subsection{\texorpdfstring{The $ABC$ sum-product problem}{The ABC sum-product problem}} It follows from Corollary \ref{cor:Bourgain-Proj} applied to a Cartesian product that if $A,B,C \subset \R$ are Borel sets with $\Hd A = \Hd B \in (0,1)$ and $\Hd C > 0$, then there exists $c \in C$ with
\begin{equation}\label{bourgain} \Hd (A + cB) > \Hd A. \end{equation}
The requirement $\Hd A = \Hd B$ could be slightly relaxed. However, the analogue of \eqref{bourgain} for sets $A,B$ without any relation between $\Hd A,\Hd B$ does not follow in any simple way from Bourgain's Theorem or its proof. In \cite{Orponen24}, the first author proposed the following: \eqref{bourgain} should be valid whenever $0 \leq \Hd B \leq \Hd A < 1$, and $\Hd C > \Hd A - \Hd B$ (as shown in \cite{Orponen24}, this would be sharp). This was established in \cite{Orponen24} under the stronger assumption $\Hd C > (\Hd A - \Hd B)/(1 - \Hd B)$. S. Eberhard and P. Varj\'{u} \cite{EV23} obtained a full resolution of this problem using entropy and additive combinatorial tools (we are grateful to P. Varj\'{u} for informing us about their work in progress). The first main result of this paper is a new proof of the $ABC$ sum-product problem, obtained via a completely different approach:
\begin{thm}\label{thm:ABCIntro} Let $0 < \beta \leq \alpha < 1$ and $\kappa > 0$. Then, there exists $\eta = \eta(\alpha,\beta,\kappa) > 0$ such that whenever $A,B \subset \R$ are Borel sets with $\Hd A = \alpha$ and $\Hd B = \beta$, then
\begin{displaymath} \Hd \{c \in \R : \Hd(A + cB) \leq \alpha + \eta\} \leq (\alpha - \beta) + \kappa. \end{displaymath}
\end{thm}
Slightly informally, Theorem \ref{thm:ABCIntro} says that if $\Hd C \geq \Hd A - \Hd B + \kappa$, then there exists $c \in C$ with $\Hd (A + cB) \geq \Hd A + \eta$. We point out that a finite field version of the $ABC$ theorem was established much earlier in \cite{2018arXiv180109591O}, and our proof of Theorem \ref{thm:ABCIntro} borrows several ingredients from \cite{2018arXiv180109591O}.

Theorem \ref{thm:ABCIntro} is derived from a more technical, but often more useful, $\delta$-discretised version, which generalizes the (comparable size) Cartesian product case of Theorem \ref{thm:Bourgain-Proj}:
\begin{thm}[$\delta$-discretised $ABC$]\label{thm:ABCConjecture}
Let $0 < \beta \leq \alpha < 1$. Then, for every
\begin{equation*}
    \gamma \in (\alpha - \beta,1], 
\end{equation*}
there exist $\chi,\delta_{0} \in (0,\tfrac{1}{2}]$ such that the following holds. Let $\delta \in (0,\delta_{0}]$, and let $A,B \subset \delta \Z \cap [0,1]$ be  sets satisfying the following hypotheses:
\begin{enumerate}
\item[\textup{(A)}] $|A| \leq \delta^{-\alpha}$.
\item[\textup{(B)}] $B \neq \emptyset$, and $|B \cap B(x,r)| \leq \delta^{-\chi} r^{\beta}|B|$ for all $x \in \R$ and $r\in [\delta,1]$.
\end{enumerate}
Further, let $C\subset \delta\cdot \Z\cap [1,2]$ be a non-empty set satisfying $|C\cap B(x,r)|\le \delta^{-\chi} r^{\gamma}|C|$ for $x \in \R$ and $r\in [\delta,1]$.

Then there exists $c \in C$ such that
\begin{equation}\label{conclusion2} |\{a + cb : (a,b) \in G\}|_{\delta} \geq \delta^{-\chi}|A|, \quad \text{for all } G \subset A \times B, \, |G| \geq \delta^{\chi}|A||B|. \end{equation}
\end{thm}
Theorem \ref{thm:ABCIntro} is a straightforward corollary of Theorem \ref{thm:ABCConjecture}, and the details were already recorded in \cite{Orponen24}.

\begin{remark}\label{rem5} In Theorem \ref{thm:ABCConjecture}, the set $C \subset \delta \Z \cap [1,2]$ may also be replaced by a probability measure $\nu$ on $[1,2]$ satisfying $\nu(B(x,r)) \leq \delta^{-\chi} \cdot r^{\gamma}$ for all $x \in \R$ and $r \in [\delta,1]$. Then, the conclusion \eqref{conclusion2} holds for some $c \in \spt(\nu)$. \end{remark}

\subsection{Exceptional set estimates for orthogonal projections}  \label{s:intro-projections}
Recall that $\pi_{\theta}$ stands for the orthogonal projection onto the line spanned by $\theta$. The following theorem from 1968 is due to Kaufman \cite{Ka}, sharpening a result of Marstrand \cite{Mar}:
\begin{thm}[Kaufman] Let $K \subset \R^{2}$ be a Borel set. Then,
\begin{equation}\label{kaufman} \Hd \{\theta \in S^{1} : \Hd \pi_{\theta}(K) < u\} \leq u, \quad \text{for all } 0 \leq u \leq \min\{\Hd K,1\}. \end{equation}
\end{thm}
It is conjectured (see for example \cite[(1.8)]{MR2994685}) that Kaufman's estimate is not sharp for any $0 < u < \min\{\Hd K,1\}$, and the sharp bound is
\begin{equation}\label{form93} \Hd \{\theta \in S^{1} : \Hd \pi_{\theta}(K) < u\} \leq \max\{2u - \Hd K,0\} \end{equation}
for $0 \leq u \leq \min\{\Hd K,1\}$. It follows from Corollary \ref{cor:Bourgain-Proj}  that the left-hand side of \eqref{kaufman} tends to $0$ as $u \searrow \tfrac{1}{2}\Hd K$. This behaviour is predicted by the conjectured inequality \eqref{form93}. It can be tracked from the proof of Corollary \ref{cor:Bourgain-Proj} that the dependence of  the dimension of the exceptional set on $2u-\Hd K$ is worse than exponential, and no improvement over \eqref{kaufman} is achieved if $u$ is not very close to $\tfrac{1}{2}\Hd K$.

Recently, the authors \cite{OS23} showed that Kaufman's bound \eqref{kaufman} can be improved by a small $\epsilon$ for all $0 < u < \min\{\Hd K,1\}$, but "$\epsilon$" (which depends on $u$ and $\Hd K$) ultimately comes from an application of Theorem \ref{thm:Bourgain-Proj}, so again it is tiny. For $\Hd K > 1$, a more quantitative improvement to \eqref{kaufman} is due to Peres and Schlag \cite{MR1749437}. Namely, the upper bound in \eqref{kaufman} can be replaced by $\max\{u + 1 - \Hd K,0\}$. However, this bound is weaker than \eqref{kaufman} for $\Hd K < 1$.

Our next result yields the conjectured numerology \eqref{form93} for sets with equal Hausdorff and packing dimension (which include Ahlfors-regular sets).
\begin{thm}\label{thm:projection} Let $K \subset \R^{2}$ be an arbitrary set of equal Hausdorff and packing dimension. Then \eqref{form93} holds for all $0 \leq u \leq \min\{\Hd K,1\}$. \end{thm}

\begin{remark}
It is not a typo that $K$ is not assumed to be Borel (or analytic). This is in line with the fact that Marstrand's theorem holds for arbitrary sets with equal Hausdorff and packing dimension, see \cite{MR3854026,MR4324956}. The main trick is to replace an appeal to Frostman's lemma by an appeal to Lemma \ref{lemma10}.
\end{remark}

\begin{remark} The proof of Theorem \ref{thm:projection} depends on \cite{OS23}. In fact, Theorem \ref{thm:projection} depends on Theorem \ref{thm:ABCConjecture}, which depends on the radial projection theorem in \cite{OSW24}, which finally depends on \cite{OS23} (which, as remarked, itself uses Theorem \ref{thm:Bourgain-Proj}).
\end{remark}

Theorem \ref{thm:projection} is derived from a natural $\delta$-discretised version, see Corollary \ref{cor2}.

\subsection{Furstenberg sets} The results in Section \ref{s:intro-projections} are fairly straightforward corollaries of estimates for $\delta$-discretised Furstenberg sets. An \emph{$(s,t)$-Furstenberg set} is a set $F \subset \R^{2}$ with the following property. There exists a non-empty line family $\mathcal{L}$ with $\Hd \mathcal{L} \geq t$ such that $\Hd (F \cap \ell) \geq s$ for all $\ell \in \mathcal{L}$. The Hausdorff dimension $\Hd \mathcal{L}$ of a line family $\mathcal{L}$ can be defined in several equivalent ways: for example, one may parametrise all non-vertical lines as $\ell_{a,b} = \{(x,y) \in \R^{2} : y = ax + b\}$, and then define
\begin{displaymath} \Hd \mathcal{L} := \Hd \{(a,b) : \ell_{a,b} \in \mathcal{L}\}. \end{displaymath}
This approach is concrete and explicit, but only works for families consisting of non-vertical lines. A "parametrisation free" approach is to first define a natural metric on the space $\mathcal{A}(2,1)$ of all affine lines in $\R^{2}$, as in \cite[Section 3.16]{Mattila95}, and then define $\Hd \mathcal{L}$ using that metric.

We briefly describe the well-known heuristic connection between Furstenberg sets and projections. By projective duality between point and lines, the Furstenberg problem can be alternatively rephrased as follows: let $P \subset \R^{2}$ with $\Hd(P)=t$, and for each $p\in P$ let $\mathcal{L}_p$ be an $s$-dimensional family of lines containing $p$. The $(s,t)$-Furstenberg set problem is equivalent to finding lower bounds on the Hausdorff dimension of
\begin{displaymath} \mathcal{L} = \bigcup_{p \in P} \mathcal{L}_{p}. \end{displaymath}
If each $\mathcal{L}_p$ consists of lines with slopes in a fixed set $E$, then morally the projection of $P$ through a typical direction in $E$ has dimension $\Hd(\mathcal{L})-s$ (this assumes a certain ``Fubini property'' that is not true for Hausdorff dimension, but can be justified at the $\delta$-discretised level). Since the exceptional projection problem corresponds to a very special type of Furstenberg set with "product structure", it is expected that the Furstenberg set problem is strictly more difficult than the projection problem. Regardless, all the currently known bounds match.

Finding lower bounds for the Hausdorff dimension of $(s,t)$-Furstenberg sets is an active and rapidly developing topic that dates back to unpublished work of Furstenberg, and to Wolff's influential exposition \cite{Wolff99}.  The main conjecture is that every $(s,t)$-Furstenberg set $F \subset \R^{2}$ with $s \in (0,1]$ and $t \in [0,2]$ has Hausdorff dimension
\begin{equation}\label{FurstConjecture}
\Hd F \geq \min \left\{s + t,\tfrac{3s + t}{2},s + 1\right\}.
\end{equation}
In addition to matching natural examples, \eqref{FurstConjecture} is the natural continuum counterpart of the foundational Szemer\'{e}di-Trotter bound \cite{MR729791} in incidence geometry; see \cite{Wolff99} for further discussion (although limited to the special case $t=1$).

The bound $s+t$ for $0\le t\le s\le 1$, which morally corresponds to Kaufman's bound for projections, is known and due to \cite{HSY22,LutzStull17}. The bound $s+1$ for $s+t\ge 2$, which is the Furstenberg analogue of Falconer's exceptional bound for projections \cite{MR673510}, was recently obtained by Y.~Fu and K.~Ren \cite{FuRen24} using Fourier analysis. For other pairs $(s,t)$, the conjecture remains open. Two elementary bounds recorded in \cite{MR3973547,MR12,Wolff99} are
\begin{equation}\label{form110}
\Hd F \geq \min\left\{\tfrac{t}{2} + s,2s\right\}.
\end{equation}
The case $t=2s$ is special, and this was exploited in \cite{HSY22,KT01} to obtain an $\epsilon$-improvement in that case. More recently, an $\epsilon$-improvement was achieved \cite{OS23,ShmerkinWang25b} in all cases. All these $\epsilon$'s here are tiny and non-explicit. Based on \cite{MR4283564} and the special structure in the case $t=2s$, explicit bounds were derived in \cite{BZ22}. For example, they show that the dimension of a $(1/2,1)$-Furstenberg set is at least $1+1/4536$. Fu and Ren \cite{FuRen24} also proved the lower bound $2s+t-1$, which improves over \eqref{form110} for $t>1$.

Our main result is the following:
\begin{thm}\label{thm:FurstIntro}
Let $F \subset \R^{2}$ be an $(s,t)$-Furstenberg set with $s \in (0,1]$ and $t \in [0,2]$. If the associated $t$-dimensional line family also has packing dimension $t$, then \eqref{FurstConjecture} holds. \end{thm}

Theorem \ref{thm:FurstIntro} is the ``Furstenberg version'' of Theorem \ref{thm:projection}. In analogy with \ref{s:intro-projections}, we prove a $\delta$-discretised theorem for Furstenberg sets (Theorem \ref{thm:Furstenberg-Ahlfors}) which yields Theorem \ref{thm:FurstIntro} as a corollary.

\subsection{Structure of the paper}
The logic of the proofs can be illustrated as follows:
\begin{align*}
\text{Radial projections} \, & \Longrightarrow ABC \text{ Theorem (Theorem \ref{thm:ABCConjecture})}\\
& \Longrightarrow \, \text{Projections of regular sets (Corollary \ref{cor2})}\\
& \Longrightarrow \, \text{Furstenberg sets associated with regular line sets (Theorem \ref{thm:Furstenberg-Ahlfors})}. \end{align*}
The arrows refer to implications between the $\delta$-discretised statements.

After a lengthy section on preliminaries, we first prove (the $ABC$) Theorem \ref{thm:ABCConjecture} in Section \ref{s:ABC}. Then, in Section \ref{s:ProjRegular}, we apply Theorem \ref{thm:ABCConjecture} to establish a $\delta$-discretised version of (the projection) Theorem \ref{thm:projection} for "almost Ahlfors-regular sets" (Definition \ref{def:deltaTRegularSet}). This is the longest proof in the paper.

In Section \ref{s:Furstenberg}, we derive our bounds for ($\delta$-discretised) Furstenberg sets. Finally, we derive the "continuous" results stated in the introduction from the $\delta$-discretised counterparts in Section \ref{s:continuousResults}; this is largely an exercise in pigeonholing, but we give full details.

\subsection{Some proof ideas}

\subsubsection{The $ABC$ Theorem} The proof of the $ABC$ theorem is easy to explain at a very heuristic level. Let us make a counter assumption that
\begin{equation}\label{form115} \sup_{c \in C} \Hd (A + cB) = \Hd A. \end{equation}
From the Pl\"unnecke-Ruzsa inequalities, it follows (heuristically) that
\begin{displaymath} \sup_{c,c'} \Hd (A + (c - c')B) \leq \Hd A. \end{displaymath}
Starting from \eqref{form115}, one can also show that
\begin{displaymath} \Hd(A + (b - b')C) = \Hd A \end{displaymath}
for "most" pairs $(b,b') \in B \times B$. This idea was already present in the finite field proof in \cite{2018arXiv180109591O}, but we provide a simplified argument starting at \eqref{form108}. Now, another application of the Pl\"unnecke-Ruzsa inequalities shows that
\begin{displaymath} \Hd ((c - c')B + (b - b')C) \leq \Hd A \quad \Longleftrightarrow \quad \Hd \left(B + \tfrac{b - b'}{c - c'}C \right) \leq \Hd A \end{displaymath}
for "most" quadruples $(b,b',c,c') \in B^{2} \times C^{2}$. However, it follows from the radial projection theorems in \cite{OSW24} that
\begin{displaymath} \Hd \tfrac{B - B}{C - C} \geq \min\{\Hd B + \Hd C,1\}. \end{displaymath}
Therefore, by \eqref{kaufman}, we may find $(b - b')/(c - c') \in (B - B)/(C - C)$ such that
\begin{displaymath} \Hd \left(B + \tfrac{b - b'}{c - c'}C\right) \geq \min\{\Hd B + \Hd C,1\} > \Hd A. \end{displaymath}
In fact, this is true for "most" quadruples $(b,b',c,c') \in B^{2} \times C^{2}$, and this yields the desired contradiction.

\subsubsection{Projections of regular sets} Next, we explain the connection between Theorem \ref{thm:projection} in the Ahlfors-regular case and the $ABC$ theorem, naturally brushing all technicalities under the carpet. Assume that $K \subset \R^{2}$ is a compact set with $\Hd K = t$, fix $u < \min\{t,1\}$, and let $E \subset S^{1}$ be a set of directions such that $\dim E = s \in [0,1]$, and $\Hd \pi_{\theta}(K) \leq u$ for all $\theta \in E$. For convenience, let us assume instead that $E \subset [0,1]$, and $\pi_{\theta}(x,y) = x + \theta y$.

Next, let $\theta_{0} \in E$ be the direction such that $\Hd \pi_{\theta_{0}}(K) =: u_{0}$ is maximal. Thus, $u_{0} \leq u$, and $\Hd \pi_{\theta}(K) \leq u_{0}$ for $\theta \in E$. Let us assume that $\theta_{0} = 0$, so $\pi_{\theta_{0}}(x,y) = x$.

If we are very lucky, $K$ now looks like the product of a $u_{0}$-dimensional set and a $(t - u_{0})$-dimensional set, say $K = A \times B$. Write $\alpha := u_{0}$ and $\beta := t - u_{0}$. If it happened that $s > \alpha - \beta = 2u_{0} - t$, then the $ABC$ theorem (roughly speaking) would tell us that there exist $\theta \in E$ and $\wp > 0$ such that
\begin{displaymath} \Hd \pi_{\theta}(K) = \Hd \pi_{\theta}(A \times B) = \Hd (A + \theta B) \geq \Hd A + \wp = u + \wp. \end{displaymath}
However, this would violate the maximality of "$u_{0}$", so we may deduce that
\begin{displaymath} s \leq 2u_{0} - t \leq 2u - \Hd K, \end{displaymath}
as in conjecture \eqref{form93}. Needless to say, the "lucky coincidence" that $K = A \times B$ is difficult to arrange. To make this happen, the Ahlfors-regularity of $K$ is very useful. Even under the Ahlfors-regularity assumption, we will not be able to show that $K = A \times B$, but instead that there exists a scale $\delta > 0$ and a $\delta^{1/2}$-tube $\mathbf{T} \subset \R^{2}$ such that $K \cap \mathbf{T}$ resembles a product set at scale $\delta$. Similar arguments have earlier appeared in \cite{MR4055989,MR4388762,OS23}.

In fact, this difficulty causes the proof to proceed rather differently from the idea above. At the core of the actual argument is Proposition \ref{prop2} which roughly speaking says the following. Assume (inductively) that we have already managed to prove a projection theorem of the following kind. If $K \subset \R^{2}$ is closed and $t$-regular, and $E \subset S^{1}$ is $s$-dimensional (with $s < 2 - t$), then for a "typical" direction $\theta \in E$ we have $\Hd \pi_{\theta}(K) \geq u$ (the real statement contains a $\delta$-discretised version of such a hypothesis). Then, as long as $u < (s + t)/2$, we can find a constant $\zeta = \zeta(u,s,t) > 0$ such that a similar conclusion holds with $u + \zeta$ in place of $u$. Iterating this proposition, we can gradually "lift" $u$ as close to the value $(s + t)/2$ as we like, see Section \ref{s:induction} for the details. This will prove Theorem \ref{thm:projection} for Ahlfors-regular sets (and sets of equal Hausdorff and packing dimension).

\subsubsection{Furstenberg sets for regular line sets} Furstenberg estimates are expected to be strictly harder to establish than the corresponding projection estimates, due to the special product structure of the subclass of Furstenberg sets that correspond to projections. In Section \ref{s:Furstenberg} we show, using an argument due to Hong Wang, that if the line set is roughly Ahlfors-regular, then projection estimates  self-improve to Furstenberg estimates. We are grateful to Hong Wang for pointing out an error in an earlier version of the article, and allowing us to include her ideas here.

\subsection{Further connections and related results}

The problems discussed in this article have a myriad of generalisations, connections and applications to areas as varied as ergodic theory, harmonic analysis and geometric group theory. We only discuss a small sample of related results and potential further directions.

Sets of equal Hausdorff and packing dimension also turned out to be easier to handle in other well-known problems. Recently,  H.~Wang and J.~Zahl \cite{WangZahl26} resolved the \emph{sticky Kakeya conjecture} in $\R^{3}$. By definition, a Kakeya set is \emph{sticky} if the associated line family in $\R^{3}$ has equal Hausdorff and packing dimension $2$. In this terminology, one could say that Theorem \ref{thm:FurstIntro} resolves the sticky Furstenberg set conjecture in $\R^{2}$. The Falconer distance set problem has also been resolved (at least in its dimension formulation) for sets of equal Hausdorff and packing dimension \cite{ShmerkinWang25}. To our knowledge, the mechanisms that allow this case to be treated in all these problems share some similarities (they all exploit in the statistical self-similarity of Ahlfors regular sets), but there are also significant differences.

This article is concerned with linear problems in $\R^{2}$. Recently, there has been significant progress on many variants of projection and Furstenberg set problems in higher dimensions, and for various kinds of non-linear generalisations; some of these are also "$\epsilon$-improvements" that ultimately depend on Bourgain's sum-product and projection theorems. See e.g. \cite{GGGHMW22,GGW24,He20, KatzZahl19,RazZahl21, PYZ22,Shmerkin23} for a small selection of recent results. It seems natural to explore whether the results of the methods of this paper can also yield further progress in some of these more general settings.

As a further application of the present resolution of the $ABC$ sum-product problem, we mention the joint work with Nicolas de Saxc\'e \cite{OSS23}, where we answer a question of Bourgain on the Fourier decay of multiplicative convolutions of Frostman measures. 

\begin{remark}\label{rem7}
  After an earlier version of this article appeared on the \emph{arXiv}, K.~Ren and H.~Wang \cite{RenWang23} proved the conjectured bounds for the size of $(s,t)$-Furstenberg sets and exceptional sets of projections in $\R^2$ discussed at \eqref{form93} and \eqref{FurstConjecture}, without \emph{a priori} assumptions on equal Hausdorff and packing dimension. In fact, they proved an important $\delta$-discretised version of these bounds. Their approach relies on Corollary  \ref{cor2} (via its application to Theorem \ref{thm:Furstenberg-Ahlfors}), but combines it with a striking new incidence bound for ``semi well-spaced'' Furstenberg sets, that relies on the \emph{high-low method} in harmonic analysis pioneered in \cite{GSW}. The original \emph{arXiv} version of this paper contained quantitative estimates for the dimension of Furstenberg sets, without the equal Hausdorff and packing dimension assumption, but these have been superseded by those in \cite{RenWang23} and hence removed from the current version. 

  Many other further developments have taken place since the first \emph{arXiv} version of this paper appeared; we mention only a few. Several works refined and extended the discretised Furstenberg estimates of Ren-Wang: see the papers of C. Demeter and H. Wang \cite{MR4869897}, H. Wang and S. Wu \cite{2024arXiv241108871W,2025arXiv250921869W}, and C. Demeter and W. O'Regan \cite{DemeterORegan25}. Among many applications of these refinements, we remark that Wang and Wu \cite{2024arXiv241108871W} established applications of Furstenberg estimates to the Fourier restriction problem. It is worth mentioning that all these works rely on the Ren-Wang incidence bound, and hence ultimately also on Corollary \ref{cor2} of the present paper.
  
  The Kakeya set problem in $\R^{3}$ was fully resolved by H. Wang and J. Zahl \cite{2025arXiv250217655W} in 2025, but the analogues of the Furstenberg set problem remain open in $\R^{d}$ for $d \geq 3$; see K. Ren's paper \cite{2023arXiv230904097R} for partial progress.  
\end{remark}

\begin{remark}
  The earlier \emph{arXiv} versions of this paper contained quantitative Furstenberg estimates under a ``single scale'' non-concentration assumption. This was Theorem 5.61 in v4. Combining the idea of that argument with the Ren and Wang's theorem \cite{RenWang23}, we were able to obtain sharp bounds for this variant of the Furstenberg problem, as well as related ones; these results have been split off into the separate article \cite{OS2026}.
\end{remark}

\subsection*{Acknowledgements} We are grateful to Hong Wang for pointing out a gap in our initial proof of Theorem \ref{thm:Furstenberg-Ahlfors}, and also for suggesting to us how to fix the gap. We thank Guangzeng Yi for a careful reading of the manuscript and pointing out a number of typos. We are also indebted to two anonymous referees for their careful reading and many helpful suggestions that improved the exposition.

\section{Preliminaries}

\subsection{Notation and terminology}

We use asymptotic notation $\lesssim, \gtrsim, \sim$. For example, $A\lesssim B$ means that $A\le C B$ for a universal $C>0$, while $A\lesssim_{r} B$ stands for $A\le C(r) B$ for a positive function $C(r)$. We also use standard $O$ notation; for example $A=O(B)$, $A=O_r(B)$ are synonym with $A\lesssim B$, $A\lesssim_r B$. The notation $A\lessapprox_{\delta} B$, $A\gtrapprox_{\delta} B$, $A\approx_{\delta} B$ or $A\approx B$ will occasionally be used to ``hide'' slowly growing functions of $\delta$ (logarithmic or small negative exponentials); its precise meaning will be made explicit each time this notation is used.

By a \emph{measure}, we always mean a Borel regular finite measure with compact support, unless otherwise specified.  

The open $r$-neighbourhood of a set $A \subset \R^{d}$ is denoted $A^{(r)}$. Logarithms are always base $2$.

\begin{definition}[Dyadic cubes]\label{def:dyadicCubes} If $n \in \Z$, we denote by $\mathcal{D}_{2^{-n}}(\R^{d})$ the family of (standard) half-open dyadic cubes in $\R^{d}$ of side-length $2^{-n}$. (Here, half-open means homothetic to $[0,1)^d$.) If $P \subset \R^{d}$ is a set, we moreover denote
\begin{displaymath} \mathcal{D}_{2^{-n}}(P) := \{Q \in \mathcal{D}_{2^{-n}} : Q \cap P \neq \emptyset\}. \end{displaymath}
Finally, we will abbreviate $\mathcal{D}_{2^{-n}} := \mathcal{D}_{2^{-n}}([0,1)^{d})$ for $n \geq 0$. \end{definition}

\begin{definition}[Covering numbers]\label{def:coveringNumber} For $P \subset \R^{d}$ and $n \in 2^{-\N}$, we write
\begin{displaymath} |P|_{2^{-n}} := |\mathcal{D}_{2^{-n}}(P)|. \end{displaymath}
 \end{definition}

We note that $|P|_{2^{-n}}$ is comparable (up to constants depending on "$d$") to the more common definition of covering number $N(P,2^{-n})$ which encodes the smallest number of open balls of radius "$2^{-n}$" required to cover $P$. The notation $|P|$ (without a subscript) will refer to cardinality in cases where $P$ is a finite set.

\begin{definition}[$(\delta,s,C)$-set]\label{def:deltaSSet} For $\delta \in 2^{-\N}$, $s \in [0,d]$, and $C > 0$, a non-empty bounded set $P \subset \R^{d}$ is called a \emph{$(\delta,s,C)$-set} if
\begin{displaymath} 
    |P \cap B(x,r)|_{\delta} \leq Cr^{s}|P|_{\delta}\quad\text{for all } x \in \R^{d}, \, r \in [\delta,1]. 
\end{displaymath}
If $\mathcal{P}$ is a finite union of dyadic cubes (possibly of different side-lengths), we say that $\mathcal{P}$ is a $(\delta,s,C)$-set if the union $\cup \mathcal{P}$ is a $(\delta,s,C)$-set in the sense above. \end{definition}

It is useful to note that if $P$ is a  $(\delta,s,C)$-set, then $|P|_{\delta} \geq \delta^{-s}/C$. This follows by applying the defining inequality with $r := \delta$ and to any $B(x,r)$ intersecting $P$. Another useful observation is that if $P$ is a $(\delta,s,C)$-set, and $P' \subset P$ with $|P'|_{\delta} \geq \lambda |P|_{\delta}$ for some $\lambda \in (0,1]$, then $P'$ is a $(\delta,s,C/\lambda)$-set.

\subsection{Dyadic tubes and slopes} \label{s:tube-preliminaries}

Recall from Definition \ref{def:dyadicCubes} that if $\delta \in 2^{-\N}$ and $A \subset \R^{d}$, the notation $\mathcal{D}_{\delta}(A)$ stands for dyadic cubes in $\R^{d}$ of side-length $\delta$ which intersect $A$. As before, we abbreviate $\mathcal{D}_{\delta} := \mathcal{D}_{\delta}([0,1)^{d})$, and in this section always $d = 2$. In addition to dyadic cubes, we now need to discuss \emph{dyadic tubes}.

\begin{definition}[Dyadic $\delta$-tubes]\label{def:dyadicTubes} Let $\delta \in 2^{-\N}$. A \emph{dyadic $\delta$-tube} is a set of the form 
  \begin{displaymath} 
     T = \cup \mathbf{D}(p) \cap B^2(0,10), 
  \end{displaymath}
  where $p \in \mathcal{D}_{\delta}([-1,1)\times \R)$, and $\mathbf{D}$ is the \emph{point-line duality map}
\begin{displaymath} \mathbf{D}(a,b) := \{(x,y) \in \R^{2} : y = ax + b\} \subset \R^{2} \end{displaymath}
sending the point $(a,b) \in \R^{2}$ to a corresponding line in $\R^{2}$. Abusing notation, we abbreviate $\mathbf{D}(p) := \cup \mathbf{D}(p)$, and we omit the intersection with $B^2(0,10)$; we simply keep in mind that we focus our attention to a bounded region. The collection of all dyadic $\delta$-tubes is denoted
\[
\mathcal{T}^{\delta} := \{\mathbf{D}(p): p \in \mathcal{D}_{\delta}([-1,1)\times \R)\}.
\]
A finite collection of dyadic $\delta$-tubes $\{\mathbf{D}(p)\}_{p \in \mathcal{P}}$ is called a $(\delta,s,C)$-set if $\mathcal{P}$ is a $(\delta,s,C)$-set in the sense of Definition \ref{def:deltaSSet}. \end{definition}

\begin{remark}
Dyadic $\delta$-tubes are not exactly $\delta$-neighbourhoods of a line, but if $z(p)$ is the centre of $p$ then
\[
\mathbf{D}(z(p))^{(c\delta)} \cap B^2(0,10) \subset \mathbf{D}(p)\cap B^2(0,10) \subset \mathbf{D}(z(p))^{(C\delta)} \cap B^2(0,10)
\]
for some universal constants $c, C > 0$.
\end{remark}

An \emph{ordinary} $\delta$-tube is a $\delta$-neighbourhood of a line in $\R^{2}$ intersected with $B^{2}(0,10)$. The following remark explains the relationship between dyadic and ordinary tubes.
\begin{remark}
Different tubes in $\mathcal{T}^{\delta}$ are essentially distinct, in the sense that the measure of their intersection is at most $(1-c)$ times the measure of each of them, for some constant $c>0$. Conversely, given a family of essentially distinct ordinary tubes $\mathcal{T}_o$ of width $\delta$ and slopes in $[-1,1]$, one can find a family of dyadic tubes $\mathcal{T}\subset\mathcal{T}^{\delta}$ such that each tube in $\mathcal{T}_o$ is covered by $O(1)$ tubes in $\mathcal{T}$, and vice versa. This allows us to convert all of our statements on dyadic tubes to statements on families of essentially distinct ordinary tubes. We deal with dyadic tubes for convenience only.
\end{remark}
For a broader discussion of dyadic tubes, see \cite[Section 2]{OS23}.

\begin{definition}[Slope set]\label{def:slopes} The \emph{slope} of a line $\ell = \mathbf{D}(a,b)$ is defined to be the number $a \in \R$; we will write $\sigma(\ell) := a$.

If $T = \mathbf{D}(p)$ is a dyadic $\delta$-tube, we define the slope $\sigma(T)$ as the minimum of the slopes of the lines contained in $T$, or in other words the left endpoint of the interval $\pi_{1}(p) \in \mathcal{D}_{\delta}([-1,1))$. Thus $\sigma(T) \in (\delta \cdot \Z) \cap [-1,1)$. If $\mathcal{T}$ is a collection of dyadic $\delta$-tubes, we write $\sigma(\mathcal{T}) := \{\sigma(T) : T \in \mathcal{T}\} \subset \delta \cdot \Z$. \end{definition}

\begin{remark}\label{rem:slopes} If $\mathcal{T}(p)$ is a family of dyadic $\delta$-tubes all of which intersect a common $\delta$-cube $p \in \mathcal{D}_{\delta}$, then $|\mathcal{T}(p)|_{\delta} \sim |\sigma(\mathcal{T}(p))|_{\delta}$, and the following are equivalent:
\begin{itemize}
\item $\mathcal{T}(p)$ is a $(\delta,s,C)$-set
\item $\sigma(\mathcal{T}(p)) \subset [-1,1)$ is a $(\delta,s,C')$-set.
\end{itemize} Here $C \sim C'$. For a proof, see \cite[Corollary 2.12]{OS23}. \end{remark}

Since families of dyadic $\delta$-tubes arise as the $\mathbf{D}$-images of families of dyadic $\delta$-cubes, our notations for dyadic cubes carry over to dyadic tubes. Notably, if $\mathcal{P} \subset \mathcal{D}_{\delta}$, and
\begin{displaymath} \mathcal{T} = \{\mathbf{D}(p) : p \in \mathcal{P}\}  \subset \mathcal{T}^{\delta}, \end{displaymath}
and $\delta \leq \Delta \leq 1$, we write
\begin{displaymath} \mathcal{T}^{\Delta}(\mathcal{T}) := \{\mathbf{D}(\mathbf{p}) : \mathbf{p} \in \mathcal{D}_{\Delta}(\mathcal{P})\} \quad \text{and} \quad |\mathcal{T}|_{\Delta} := |\mathcal{T}^{\Delta}(\mathcal{T})|. \end{displaymath}
For $\mathbf{T} = \mathbf{D}(\mathbf{p}) \in \mathcal{T}^{\Delta}(\mathcal{T})$, we will also use the notation $\mathcal{T} \cap \mathbf{T} := \{\mathbf{D}(p) \in \mathcal{T} : p \subset \mathbf{p}\}$.

\subsection{Lemmas on Lipschitz functions} \label{s:lipschitz}

Some proofs in this article depend crucially on carefully choosing scales in multiscale decompositions of different sets so that the ``branching'' of the set between various scales is controlled in some useful manner. This section contains lemmas for this purpose. For the details of the dictionary between Lipschitz functions and branching, see Lemma \ref{l:regular-is-frostman-ahlfors} below.

\begin{definition}[Interpolating slope]
Given a function $f:[a,b]\to\R$, we define $s_f(a,b)$ to be the slope of the secant line of $f$ through $a$ and $b$:
\[
s_f(a,b) = \frac{f(b)-f(a)}{b-a}.
\]
\end{definition}

\begin{definition}[$\e$-linear and superlinear functions] \label{def:eps-linear}
Given a function $f:[a,b]\to\R$ and numbers $\e>0,\sigma$, we say that $(f,a,b)$ is \emph{$(\sigma,\e)$-superlinear} or \emph{$f$ is $(\sigma,\e)$-superlinear on $[a,b]$} if
\[
f(x) \ge f(a)+\sigma(x-a)-\e(b-a)\quad\text{for all } x\in [a,b].
\]
If $\sigma=s_f(a,b)$, then we simply say that $f$ is \emph{$\e$-superlinear}

We say that \emph{$(f,a,b)$ is $\e$-linear} or \emph{$f$ is $\e$-linear on $[a,b]$} if both $(f,a,b)$ and $(-f,a,b)$ are $\e$-superlinear. Equivalently,
\[
\big|f(x)-  L_{f,a,b}(x) \big|\le \e |b-a| \quad\text{for all } x\in [a,b],
\]
where $L_{f,a,b}$ is the affine function that agrees with $f$ on $a$ and $b$.
\end{definition}

The next lemma follows by combining \cite[Lemmas 5.21 and 5.22]{ShmerkinWang25}.
\begin{lemma} \label{l:combinatorial-weak}
For every $d\in\N$ and $\epsilon>0$ there is $\tau=\tau(d,\e)>0$ such that the following holds: for any non-decreasing $d$-Lipschitz function $f:[0,m]\to\R$ with $f(0)=0$ there exist  sequences
\begin{align*}
0&=a_0 < a_1 <\cdots < a_{n}=: a \le m,\\
0&\le t_0 <t_1<\cdots <t_{n-1} \le d,
\end{align*}
such that:
\begin{enumerate}[(\rm i)]
   \item $a_{j+1}-a_j \ge \tau m$.
  \item $(f,a_{j},a_{j+1})$ is $(t_j,0)$-superlinear.
  \item $\sum_{j=0}^{n - 1} (a_{j+1}-a_j)t_j \ge f(m)-\e m$.
\end{enumerate}
\end{lemma}

\begin{remark} \label{r:combinatorial-weak}
Note that (ii) and (iii) together imply that if $F$ is the piecewise affine function with $F(0)=0$ and slope $t_j$ on $[a_{j},a_{j+1}]$, then
\[
F(a_j) \le f(a_j) \le F(a_j)+\e m,\qquad j\in \{0,\ldots,n-1\}.
\]
Therefore, also
\[
|f(a_i)-f(a_j)- (F(a_i)-F(a_j))|\le \e m\qquad i,j\in\{0,\ldots,n-1\}.
\]
\end{remark}

\begin{cor}\label{cor:lemma2} Let $0 < \sigma \leq d$, $d \geq 1$, $\zeta \in (0,\sigma]$, and $\epsilon \in (0,\zeta/6]$. Let $f \colon [0,1] \to [0,\infty)$ be a non-decreasing piecewise affine $d$-Lipschitz function satisfying $f(0) = 0$ and $(f,0,1)$ is $(\sigma,\e)$-superlinear.  Then, there exists a point $a \in [\tfrac{\zeta}{12d},\tfrac{1}{3}]$ such that $(f,a,1)$ is $(\sigma-\zeta,0)$-superlinear.
\end{cor}
\begin{proof}
Write $c=\zeta/(12d)$. Let $g(x): [0,1-c] \to [0,(1-c)d]$ be the shifted function $g(x)=f(x+c)-f(c)$. Then $g$ satisfies the assumptions of Lemma \ref{l:combinatorial-weak}, which we apply with $\e/100$ in place of $\e$. Let $(a_j), (t_j)$ be the resulting numbers. Let $a=\inf\{ a_j: t_j \ge \sigma-\zeta\}+c$. Since the $t_j$ are increasing, we get that $(g,a-c,1-c)$ is $(\sigma-\zeta,0)$-superlinear, and then so is $(f,a,1)$. Since $(f,0,1)$ is $(\sigma,\e)$-superlinear, we have
\[
\sigma a -\e \le f(a).
\]
On the other hand, if $F$ is the piecewise affine function with $F(0)=0$ and slope $t_j$ on $[a_{j},a_{j+1}]$, then $F(a)\le dc+(\sigma-\zeta)(a-c)$.  By Remark \ref{r:combinatorial-weak}. 
\[
f(a) \le dc+(\sigma-\zeta)(a-c)+\e/100.
\]
Combining these two inequalities, some algebra yields
\[
a \le \frac{dc+1.01\e}{\zeta}+ c\le \frac{1}{3}.
\]
\end{proof}

\subsection{Uniform sets and branching numbers} \label{s:uniform}

\begin{definition}\label{def:uniformity}
Let $n \geq 1$, and let
\begin{displaymath} \delta = \Delta_{n} < \Delta_{n - 1} < \ldots < \Delta_{1} \leq \Delta_{0} = 1 \end{displaymath}
be a sequence of dyadic scales.  We say that a set $P\subset [0,1)^2$ is \emph{$\{\Delta_j\}_{j=1}^n$-uniform} if there is a sequence $\{N_j\}_{j=1}^n$ such that $N_{j} \in 2^{\N}$ and $|P\cap Q|_{\Delta_{j}} = N_j$ for all $j\in \{1,\ldots,n\}$ and all $Q\in\mathcal{D}_{\Delta_{j - 1}}(P)$.
\end{definition}

\begin{remark} The key feature of $\{\Delta_{j}\}_{j = 1}^{n}$-uniform sets is the following equation which will be used many times below without further remark: if $0 \leq k \leq l < m \leq n$, then
\begin{displaymath} |P \cap Q|_{\Delta_{m}} = |P \cap Q'|_{\Delta_{m}}|P \cap Q|_{\Delta_{l}}\quad\text{for all } Q \in \mathcal{D}_{\Delta_{k}}(P), \, Q' \in \mathcal{D}_{\Delta_{l}}(P). \end{displaymath}
Indeed, both sides equal $N_{k + 1}\cdots N_{m}$. As a corollary, if $P \subset [0,1)^{d}$, then the above simplifies to $|P|_{\Delta_{m}} = |P \cap Q|_{\Delta_{m}}|P|_{\Delta_{l}}$ for $0 \leq l < m \leq n$ and $Q \in \mathcal{D}_{\Delta_{l}}(P)$.

\end{remark}

The following simple but key lemma asserts that one can always find ``dense uniform subsets''. See e.g. \cite[Lemma 3.6]{Sh} for the short proof.
\begin{lemma} \label{l:uniformization}
Let $P\subset [0,1)^{d}$, $m,T \in \N$, and $\delta := 2^{-mT}$. Let also $\Delta_{j} := 2^{-jT}$ for $0 \leq j \leq m$, so in particular $\delta = \Delta_{m}$. Then, there is a $\{\Delta_j\}_{j=1}^{m}$-uniform set $P'\subset P$ such that
\begin{displaymath}
|P'|_\delta \ge  \left(2T \right)^{-m} |P|_\delta.
\end{displaymath}

In particular, if $\epsilon > 0$ and $T^{-1}\log (2T) \leq \epsilon$, then $|P'|_{\delta} \geq \delta^{\epsilon}|P|_{\delta}$.
\end{lemma}

Lemma \ref{l:uniformization} will be mainly used through the following corollary:
\begin{cor}\label{cor:uni} For every $s \in (0,d]$ and $\epsilon \in (0,1]$, there exists $\delta_{0} > 0$ such that the following holds for all $\delta \in (0,\delta_{0}]$. Let $P \subset [0,1)^{d}$ be a $\delta$-separated $(\delta,s,\delta^{-\epsilon})$-set. Then, there exists $T \sim_{\epsilon} 1$ and a $\{2^{-jT}\}_{j = 1}^{m}$-uniform subset $P' \subset P$ so that $|P'| \geq \delta^{\epsilon}|P|$,
\begin{displaymath} 2^{-(m + 1)T} < \delta \leq 2^{-mT}, \end{displaymath}
and $P'$ is also a $(\delta,s,\delta^{-2\epsilon})$-set. \end{cor}

\begin{proof}  Take $T \in \N$ so large that $T^{-1}\log(2T) < \epsilon/2$, and then let $m \in \N$ be the largest number such that $\delta' = 2^{-mT} \geq \delta$. Let $\bar{P} \subset P$ be a $\delta'$-net. Since $\delta'/\delta \leq 2^{T}$, and $P \subset \R^{d}$ is $\delta$-separated, we have
\begin{displaymath} |\bar{P}|_{\delta'} = |\bar{P}| \gtrsim 2^{-dT}|P|. \end{displaymath}
Next, apply Lemma \ref{l:uniformization} to find a $\{2^{-jT}\}_{j = 1}^{m}$-uniform subset $P' \subset \bar{P}$ with
\begin{displaymath} |P'| \geq (\delta')^{\epsilon/2}|\bar{P}| \gtrsim 2^{-dT}\delta^{\epsilon/2}|P|. \end{displaymath}
Now, if $\delta > 0$ is small enough, $|P'| \geq \delta^{\epsilon}|P|$, and $P'$ is the desired subset of $P$.  The fact that it is a $(\delta,s,\delta^{-2\epsilon})$-set follows from the observation made after Definition \ref{def:deltaSSet}.
\end{proof}

A nice feature of every uniform set $P \subset [0,1)^{d}$ is that if $P$ happens to be a $(\delta,s)$-set, and if $\delta \leq \Delta \leq 1$, then $P$ is automatically a $(\Delta,s)$-set:
\begin{lemma}\label{lemma4} Let $\delta \in 2^{-\N}$, and let $P \subset [0,1)^{d}$ be a $(\delta,s,C)$-set, for some $s \in [0,d]$ and $C > 0$. Fix $\Delta \in 2^{-\N} \cap [\delta,1]$, and assume that the map
\begin{displaymath} \mathbf{p} \mapsto |P \cap \mathbf{p}|_{\delta}, \qquad \mathbf{p} \in \mathcal{D}_{\Delta}(P), \end{displaymath}
is constant. Then $P$ is a $(\Delta,s,O_{d}(1)C)$-set.

In particular, assume that $P \subset [0,1)^{d}$ is a $\{2^{-jT}\}_{j = 1}^{m}$-uniform set, with $\delta = 2^{-mT}$. Then $P$ is also a $(\Delta,s,O_{d}(1)C)$-set for every $\Delta = 2^{-jT}$, for $1 \leq j \leq m$. \end{lemma}

\begin{proof}[Proof of Lemma \ref{lemma4}] Fix $\mathbf{p} \in \mathcal{D}_{\Delta}(P)$ arbitrary, and let $r \in 2^{-\N} \cap [\Delta,1]$. The key observation is that if $Q \in \mathcal{D}_{r}(P)$, then
\begin{equation}\label{form30} |P \cap Q|_{\delta} = |P \cap Q|_{\Delta} \cdot |P \cap \mathbf{p}|_{\delta}, \end{equation}
which follows instantly from the constancy of $\mathbf{p} \mapsto |P \cap \mathbf{p}|_{\delta}$. Consequently,
\begin{displaymath} |P \cap Q|_{\Delta} \stackrel{\eqref{form30}}{=} \frac{|P \cap Q|_{\delta}}{|P \cap \mathbf{p}|_{\delta}} \leq C\diam(Q)^{s} \cdot \frac{|P|_{\delta}}{|P \cap \mathbf{p}|_{\delta}} \stackrel{\eqref{form30}}{=} C\diam(Q)^{s}|P|_{\Delta}. \end{displaymath}
This completes the proof (the $O_{d}(1)$-factor in the statement comes from replacing dyadic cubes by balls). \end{proof}

The lemma will mainly be used via the following corollary:

\begin{cor}\label{cor1} Let $\delta \in 2^{-\N}$, and let $P \subset [0,1)^{d}$ be a $(\delta,s,C_{1})$-set, for some $s \in [0,d]$ and $C_{1} > 0$. Fix $\Delta \in 2^{-\N} \cap [\delta,1]$, and assume that the map
\begin{displaymath} \mathbf{p} \mapsto |P \cap \mathbf{p}|_{\delta}, \qquad \mathbf{p} \in \mathcal{D}_{\Delta}(P), \end{displaymath}
is constant. Let $P' \subset P$ be an arbitrary subset satisfying $|P'|_{\delta} \geq |P|_{\delta}/C_{2}$. Then $P'$ is a $(\Delta,s,O_{d}(1)C_{1}C_{2})$-set.

In particular, if $T \in \N$, and $P$ is $\{2^{-jT}\}_{j = 1}^{m}$-uniform, with $\delta = 2^{-mT}$, then $P'$ is a $(\Delta,s,O_{d}(1)C_{1}C_{2})$-set for every $\Delta = 2^{-jT}$, $1 \leq j \leq m$.
\end{cor}

\begin{proof} Let $\mathbf{p} \in \mathcal{D}_{\Delta}(P')$ be the dyadic cube maximising $|P' \cap \mathbf{p}|_{\delta}$ (among all cubes in $\mathcal{D}_{\Delta}(P')$). Then,
\begin{displaymath} \frac{|P|_{\delta}}{C_{2}} \leq |P'|_{\delta} \leq |P' \cap \mathbf{p}|_{\delta} \cdot |P'|_{\Delta} \leq |P \cap \mathbf{p}|_{\delta} \cdot |P'|_{\Delta}. \end{displaymath}
Consequently, applying \eqref{form30} with $Q = [0,1)^{d}$, we find
\begin{equation}\label{form33a} |P'|_{\Delta} \geq \frac{|P|_{\delta}}{C_{2}|P \cap \mathbf{p}|_{\delta}} \stackrel{\eqref{form30}}{=} \frac{|P|_{\Delta}}{C_{2}}, \end{equation}
By Lemma \ref{lemma4}, we already know that $P$ is a $(\Delta,s,O_{d}(1)C_{1})$-set, so the claim follows from \eqref{form33a}.
\end{proof}

Recall that logarithms are always base $2$. We now define the branching function of a uniform set:
\begin{definition}[Branching function]\label{branchingFunction} Let $T \in \N$, and let $P \subset [0,1)^{d}$ be a $\{\Delta_{j}\}_{j = 1}^{m}$-uniform set, with $\Delta_{j} := 2^{-jT}$, and let $\{N_{j}\}_{j = 1}^{m} \subset \{1,\ldots,2^{dT}\}^{m}$ be the associated sequence. We define the \emph{branching function} $\beta \colon [0,m] \to [0,dm]$ by setting $\beta(0) = 0$, and
\begin{displaymath} \beta(j) := \frac{\log |P|_{2^{-jT}}}{T} = \frac{1}{T} \sum_{i = 1}^{j} \log N_{i}, \qquad i \in \{1,\ldots,m\}, \end{displaymath}
and then interpolating linearly. \end{definition}
Note that since $N_{i} \in [1,2^{dT}]$, the branching function $\beta$ is a non-decreasing, $d$-Lipschitz function. 

We now define a stronger version of $(\delta,s,C)$-sets that will play a key role in our main theorems:
\begin{definition}[$(\delta,t,C)$-regular set]\label{def:deltaTRegularSet} Let $t > 0$ and $C \geq 1$. A non-empty set $\mathcal{P} \subset \mathcal{D}_{\delta}$ is called \emph{$(\delta,t,C)$-regular} if
\begin{enumerate}[(i)]
\item $\mathcal{P}$ is a $(\delta,t,C)$-set, and
\item $|\mathcal{P} \cap \mathbf{p}|_{r} \leq C(R/r)^{t}$ for all dyadic $\delta \leq r \leq R < \infty$, and for all $\mathbf{p} \in \mathcal{D}_{R}$.
\end{enumerate}
Here $\mathcal{P} \cap \mathbf{p} = \{p \in \mathcal{P} : p \subset \mathbf{p}\}$.
\end{definition}

The following lemma provides a dictionary between properties of the set $P$ and its branching function.
\begin{lemma} \label{l:regular-is-frostman-ahlfors}
Let $P$ be a $(\Delta^i)_{i=1}^m$-uniform set in $[0,1]^d$ with associated branching function $\beta$, and let $\delta=\Delta^m$. Below, all implicit constants may depend on $d$.
\begin{enumerate}[(i)]
  \item If $\beta$ is $(\sigma,\e)$-superlinear on $[0,m]$, then $P$ is a $(\delta,\sigma,O_\Delta(1)\delta^{-\e})$- set. Conversely, if $P$ is a $(\delta,s,\delta^{-\epsilon})$-set for some $s \in [0,d]$, then $\beta$ satisfies
\begin{displaymath}
\beta(x) \geq s x - \epsilon m - O(1) \quad\text{for all } x \in [0,m].
\end{displaymath}
  \item If $\beta$ is $\e$-linear on $[0,m]$ then $P$ is $(s_\beta(0,m),O_{\Delta}(1)\delta^{-\e},O_\Delta(1)\delta^{-\e})$-regular between scales $\delta$ and $1$.
\end{enumerate}
\end{lemma}
\begin{proof}
Other than the "conversely" part  in (i), this is \cite[Lemma 8.3]{OS23}. By the piecewise affine definition of $\beta$, it suffices to prove the lower bounds at all integer points $x = j \in \{1,\ldots,m\}$. By the assumption that $P$ is a $(\delta,s,\delta^{-\epsilon})$-set, and noting that $\delta^{-\epsilon} = \Delta^{-\e m}$, we have $|P \cap Q|_{\delta} \lesssim_{d} \Delta^{-\e m} \Delta^{s j} |P|_{\delta}$ for all $Q \in \mathcal{D}_{\Delta^j}(P)$. Therefore,
\begin{displaymath} |P|_{\Delta^j} = \frac{|P|_{\delta}}{|P \cap Q|_{\delta}} \geq c_{d} \Delta^{\e m}\Delta^{-s j}\quad\text{for all } Q \in \mathcal{D}_{\Delta^j}(P), \end{displaymath}
where $c_{d} > 0$ is a constant depending only on $d$. Consequently,
\begin{displaymath} \beta(j) = \frac{\log |P|_{\Delta^j}}{T} \geq j s - \epsilon m - \log c_{d}, \qquad j \in \{1,\ldots,m\}, \end{displaymath}
which proves the lemma. \end{proof}

\begin{definition}[Renormalised set]\label{def:renormalisation} Let $P \subset [0,1)^{d}$, let $r \in 2^{-\N}$, and let $Q \in \mathcal{D}_{r}$. Let $S_{Q} \colon Q \to [0,1)^{d}$ be the homothety with $S_{Q}(Q) = [0,1)^{d}$. We write
\begin{displaymath} 
  P_{Q} := S_{Q}(P\cap Q). 
\end{displaymath}
The set $P_{Q}$ is the \emph{$Q$-renormalisation of $P$}.
\end{definition}

The following simple lemma will prove very useful:
\begin{lemma}\label{lemma3} 
  Let $\Delta\in 2^{-\N}$, and let $P \subset [0,1)^{d}$ be a $\{\Delta^{j}\}_{j = 1}^{m}$-uniform set. Let $\beta \colon [0,m] \to [0,dm]$ be the associated branching function.

For any $a\in \N\cap [0,m)$ and every $Q \in \mathcal{D}_{\Delta^a}(P)$, the $Q$-renormalisation $P_{Q}$ is $\{\Delta^j\}_{j=1}^{m-a}$-uniform, and the associated branching function $\beta_{Q}$ is the shift
\[
\beta_{Q}(x) = \beta(a+x) - \beta(a)\quad\text{for all } x\in [0,m-a].
\]

Moreover, for any $b\in \N\cap (a,m]$, if the function $(\beta,a,b)$ is $(s,\e)$-superlinear, then $P_{Q}$ is a $(\delta,s,C_{\Delta,d}\delta^{-\e})$-set for $\delta:=\Delta^{b-a}$.

\end{lemma}
\begin{proof}
The first claim is a reworking of the definitions of the renormalised set and the branching function. The second claim follows from the first and Lemma \ref{l:regular-is-frostman-ahlfors}.
\end{proof}

\section{The $ABC$ sum-product problem}\label{s:ABC}

In this section, we prove Theorem \ref{thm:ABCConjecture}. We will actually establish the following "weak" version of it:
\begin{thm} \label{thm:ABCWeak}
Let $0 < \beta \leq \alpha < 1$. Then, for every
\begin{equation*}
  \gamma \in (\alpha - \beta,1], 
\end{equation*}
there exist $\chi,\delta_{0} \in (0,\tfrac{1}{2}]$ such that the following holds. Let $\delta \in 2^{-\N}$ with $\delta \in (0,\delta_{0}]$, and let $A,B \subset \delta\cdot \Z \cap [0,1]$ be  sets satisfying the following hypotheses:
\begin{enumerate}
\item[\textup{(A)}]  $|A| \leq \delta^{-\alpha}$.
\item[\textup{(B)}]  $B \neq \emptyset$, and $|B \cap B(x,r)| \leq \delta^{-\chi} r^{\beta}|B|$ for all $x \in \R$ and $r\in [\delta,1]$.
\end{enumerate}
Further, let $C\subset \delta\cdot \Z\cap [1,2]$ be a nonempty set satisfying $|C\cap B(x,r)|\le \delta^{-\chi}\cdot r^{\gamma}$ for $x \in \R$ and $r\in [\delta,1]$.

Then there exists $c \in C$ such that
\begin{equation}
\label{conclusion} | A+c B|_{\delta} \geq \delta^{-\chi}|A|.
 \end{equation}
Additionally, the constants $\chi,\delta_{0} \in (0,\tfrac{1}{2}]$ stay bounded away from $0$ when $(\alpha,\beta,\gamma)$ range in a compact subset of $\{(\alpha,\beta,\gamma) : 0 < \beta \leq \alpha < 1 \text{ and } \gamma \in (\alpha - \beta,1]\}$.
\end{thm}

Note that the only difference with  Theorem \ref{thm:ABCConjecture} is that \eqref{conclusion2} is only claimed to hold for $G=A\times B$. Let $W(\alpha,\beta,\gamma)$ denote the statement of Theorem \ref{thm:ABCWeak} with parameters $\alpha,\beta,\gamma$, and let $S(\alpha,\beta,\gamma)$ denote the statement of Theorem \ref{thm:ABCConjecture} with the same parameters. It was established in \cite[Section 5.1]{Orponen24} (see also \cite[Remark 1.9]{Orponen24}) that $W(\alpha,\beta,\gamma)$ implies $S(\alpha,\bar{\beta},\gamma)$ for any fixed  $\bar{\beta}<\beta$. (In \cite{Orponen24} this was stated under different assumptions on $\alpha,\beta,\gamma$ but the deduction is agnostic to the values of the parameters.) Since the assumption $\gamma \in (\alpha - \beta,1]$ is open in $\beta$, we do have that Theorem \ref{thm:ABCWeak} implies Theorem \ref{thm:ABCConjecture}. We refer the reader to \cite[Remark 1.9 and Section 5.1]{Orponen24} for more details. In the remainder of this section, our goal is to establish Theorem \ref{thm:ABCWeak}.

\subsection{Refined radial projections for tube-Frostman measures}

Our proof of Theorem \ref{thm:ABCWeak} is based on estimates on radial projections obtained in \cite{OSW24}. We start by introducing the necessary concepts.
\begin{definition}[$(t,C)$-Frostman and $(\delta,t,C)$-Frostman measures]\label{def:FrostmanMeasure} Let $t > 0$ and $C \geq 1$. A Borel measure $\mu$ on $\R^{d}$ is called a $(t,C)$-Frostman measure if $\mu(B(x,r)) \leq C r^{t}$ for all $x \in \R^{d}$ and $r > 0$.

If $\delta \in (0,1]$, and the inequality $\mu(B(x,r)) \leq Cr^{t}$ holds for $r \geq \delta$, we say that $\mu$ is a $(\delta,t,C)$-Frostman measure.
\end{definition}
The concept of $(\delta,t,C)$-Frostman measure is the measure analogue of that of $(\delta,s,C)$-set (Definition \ref{def:deltaSSet}). For example, if $\mathcal{P} \subset \mathcal{D}_{\delta}$ is a $(\delta,s,C)$-set, then the measure \[\mu := \frac{1}{|\mathcal{P}|}\sum_{p \in \mathcal{P}} \frac{\mathcal{L}^{d}|_{p}}{\mathcal{L}^{d}(p)}\] is a $(\delta,s,C')$-Frostman measure for some $C' \sim_{d} C$. 

\begin{definition}[Radial projections]
Given $x\in\R^d$, we let $\pi^x: \R^d \, \setminus \, \{x\}\to S^{d-1}$ be the radial projection with centre $x$. Explicitly:
\[
\pi^x(y) = \frac{y-x}{|y-x|}.
\]
\end{definition}

\begin{definition}[Thin tubes]
Let $(\mu,\nu)$ be Borel probability measures on $\R^d$. We say that $(\mu,\nu)$ have $(s,K,c)$-thin tubes if there exists a (Borel) set $H\subset \R^d\times\R^d$ with $(\mu\times\nu)(H)\ge c$, and such that for any $x \in \R^{d}$, for all $r>0$, and for all balls $B_r$ of radius $r>0$ in $S^{d-1}$, we have
\[
\nu\big\{ (\pi^x)^{-1}(B_r) \cap \{ y: (x,y)\in H\} \big\} \le K\cdot r^s.
\]
\end{definition}

The following is a variant of \cite[Corollary 2.22]{OSW24}, in which a stronger conclusion regarding the parameter "c" in the definition of thin tubes is obtained from a stronger hypothesis. The stronger hypothesis is immediate for Cartesian products of the type we consider. In this proposition, an $r$-tube is the intersection of an infinite tube of radius $r$.
\begin{proposition} \label{prop:thin-tubes}
Given $t, s\in (0,1]$ and $\sigma \in (0,s)$, there are $K=K(t,s,\sigma)>0$, $\delta_0=\delta_0(t,s,\sigma)>0$ and $\e_0=\e_0(t,s,\sigma) > 0$ such that the following holds for $\delta\in (0,\delta_0)$ and $\e\in (0,\e_0)$.

Let $\mu_1, \mu_2$ be $(s,\delta^{-\e})$-Frostman probability measures on $B^2(0,1)$ such that
\[
\mu_i(T) \le \delta^{-\e}\cdot r^t \text{ for all $r$-tubes $T$ }, r\in [\delta,1], \quad i \in \{1,2\}.
\]
Assume further that
\[
\dist(\spt\mu_1,\spt\mu_2) \ge \delta^{\e}.
\]
Then $(\mu_1, \mu_2)$ and $(\mu_2,\mu_1)$ have $(\sigma,\delta^{-K\e},1-K\delta^{\e})$-thin tubes.
\end{proposition}
\begin{proof}
By \cite[Lemma 2.9]{OSW24}, there exists $\eta>0$ depending only on $t,s,\sigma$ such that if $\sigma_0\in [t,\sigma]$ and if $K_1>4, L>0$ then
\[
(\mu_{1},\mu_{2}), (\mu_{2},\mu_{1}) \text{ have }  (\sigma_0,\delta^{-K_1\e},1-L\delta^{\e}) \text{-thin tubes}
\]
\[
\Longrightarrow
\]
\[
(\mu_{1},\mu_{2}), (\mu_{2},\mu_{1}) \text{ have } (\sigma_0+\eta,\delta^{-(K_1/\eta)\e},1-5L\delta^{\e})\text{-thin tubes}.
\]
Note that this self-improves the thin tubes exponent at the cost of worsening the constants.  More precisely, \cite[Lemma 2.8]{OSW24} is applied with ``$\delta^{-\e}$'' in place of ``$C$'' and with ``$L\delta^\e$'' in place of ``$\e$''. Note that the expression
\[
\max\{ \delta^{-K_1\e}, O(\delta^{-4\e}), 4\delta^{-\e}, O_\eta(1) \}^{\sigma/\eta+1}
\]
appearing in the statement of  \cite[Lemma 2.9]{OSW24} is indeed bounded by $\delta^{-(K_1/\eta)\e}$ provided $\e$ is small enough in terms of $K_1$.

By assumption, $(\mu_1,\mu_2)$ and $(\mu_2,\mu_1)$ have $(t,\delta^{-\e},1-\delta^{\e})$-thin tubes (in fact, $1-\delta^{\e}$ could be replaced by $1$, but the $\delta^{\e}$-loss is needed to get the desired estimate). Hence, starting with $\sigma_0=t$ and $K_1=K_2=L=1$, and iterating the above self-improving estimate $\lceil \eta^{-1}(\sigma-t) \rceil$ many times,  we get the claim.
\end{proof}

\subsection{A discretised expansion estimate}

Next, we apply Proposition \ref{prop:thin-tubes} to obtain a robust discretised sum-product estimate that may be of independent interest.

\begin{proposition} \label{prop:expansion}
Given $s,t\in (0,1)$ and $\sigma\in [0,\min\{s+t,1\})$, there exist $\e = \e(s,t,\sigma) > 0$ and $\delta_{0} = \delta_{0}(s,t,\sigma,\e) > 0$  such that the following holds for all $\delta \in (0,\delta_{0}]$.

Let $A_{1},A_{2} \subset \delta\cdot\Z\cap  [-2,2]$ be $(\delta,s,\delta^{-\e})$ sets and let $B_{1},B_{2} \subset \delta\cdot\Z\cap [-2,2]$ be $(\delta,t,\delta^{-\e})$ sets. Let $P \subset B(1) \subset \R^{2}$ be a $(\delta,s + t,\delta^{-\epsilon})$-set. Then there exists a set $\mathcal{G}\subset A_{1} \times A_{2} \times B_{1} \times B_{2}$ with
\[
|(A_{1}\times A_{2}\times B_{1}\times B_{2}) \, \setminus \, \mathcal{G}|\le \delta^{\epsilon} |A_{1}||A_{2}||B_{1}||B_{2}|,
\]
such that if $(a_1,a_2,b_1,b_2)\in \mathcal{G}$ and $X\subset P$ satisfies $|X|_{\delta} \ge \delta^{\e}|P|_{\delta}$, then
\begin{equation} \label{eq:robust-sum-prod}
  |\{ (b_1\pm b_2)a + (a_1\pm a_2) b: (a,b)\in X\}|_{\delta} \ge \delta^{-\sigma}.
\end{equation}
\end{proposition}

The proof of Proposition \ref{prop:expansion} is based not only on Theorem \ref{prop:thin-tubes}, but also the following $\delta$-discretised version of Kaufman's classical projection theorem \cite{Ka}. 
\begin{proposition}\label{prop7} Let $\delta \in 2^{-\N}$, $s \in (0,1)$, $\alpha > 0$, and $C_{1},C_{2} \geq 1$. Let $\mathcal{P} \subset \mathcal{D}_{\delta}$ be a non-empty $(\delta,s,C_{1})$-set, and let $\nu$ be an $(s,C_{2})$-Frostman probability measure on $S^{1}$. Then, there exists $\theta \in \spt \nu$ such that, for every $\epsilon > 0$,
\begin{displaymath} 
    |\pi_{\theta}(\mathcal{P}')|_{\delta} \gtrsim_{s} \frac{\alpha^{2}}{C_{1}C_{2}\log^2(1/\delta)} \cdot \delta^{-s}, \qquad \text{ for all } \mathcal{P}' \subset \mathcal{P} \text{ with } |\mathcal{P}'| \geq \alpha |\mathcal{P}|.
\end{displaymath} 
\end{proposition} 
\begin{proof} In this proof all implicit constants are allowed to depend on $s$. Let 
  \begin{displaymath} \mu := \frac{1}{|P|} \sum_{p \in \mathcal{P}} \delta^{-2}\mathbf{1}_{p} \end{displaymath}
  be the normalised Lebesgue measure on $\mathcal{P}$. Since the conclusion does not change if we replace $\theta$ by $\theta'\in B(\theta,\delta)$, we may replace $\nu$ by a smoothened version of itself, equal to a multiple of Lebesgue measure on each $\delta$-arc. 
   
  Recall that the $s$-dimensional Riesz energy of a measure $\rho$ on $\R^{d}$ (or $S^1$) is defined as
  \[
  I_s(\rho) = \iint \frac{d\rho(x)d\rho(y)}{|x-y|^s} \in (0,\infty].
  \]
  Note that 
   \begin{displaymath} 
      \mu(B(x,r)) \lesssim C_{1}\cdot
        \begin{cases}  \delta^{s - 2}r^{2}, & 0 < r \leq \delta, \\ r^{s }, & \delta \leq r \leq 1 \\ 1 & r \geq 1 
      \end{cases}, \qquad 
      \nu(B(\theta,r)) \lesssim C_{2} \cdot
        \begin{cases}  \delta^{s - 1} r, & 0 < r \leq \delta, \\ r^{s}, & \delta \leq r \leq 1 \\ 1 & r \geq 1.
      \end{cases}
  \end{displaymath}
  For fixed $x\in\R^2 \, \setminus \, \{0\}$ and $r>0$, we have that $\{ \theta\in S^1: |\pi_\theta(x)|\le r\}$ is contained in the union of two balls of radius $\sim r/|x|$. Therefore, for $0<|x|\lesssim 1$, Fubini's Theorem yields
  \begin{align*}
    \int_{S^1} |\pi_{\theta} x|^{-s} d\nu(\theta) &=  \int_0^\infty \nu\{ \theta:|\pi_{\theta}(x)|\le r^{-1/s} \} \, dr \\
    &\le C_2 \left(\int_0^{|x|^{-s}} 1 + |x|^{-s}\int_{|x|^{-s}}^{\delta^{-s}|x|^{-s}} r^{-1} + \delta^{s-1}|x|^{-1}\int_{\delta^{-s}|x|^{-s}}^\infty r^{-1/s} \right) \\
   &\lesssim C_{2} \log(1/\delta)|x|^{-s} .
  \end{align*}
   A similar calculation yields 
  \[
    I_s(\mu) \lesssim C_{1} \log(1/\delta).
  \]
  Applying Fubini's theorem, we obtain
  \begin{align*}
    \int_{S^1} I_s(\pi_\theta \mu)\, d\nu(\theta) &= \iiint  |\pi_{\theta}x-\pi_{\theta}y|^{-s} \,d\nu(\theta) d\mu(x)d\mu(y) 
    \\&\lesssim C_2 \iint \left( |x-y|^{-s}\log(1/\delta) + \log(1/\delta) \right) d\mu(x)d\mu(y) \\
    &\lesssim C_1 C_2 \log(1/\delta)^{2}.
  \end{align*}
   In particular, there is a $\theta\in \spt\nu$ such that $I_{s}(\pi_{\theta}\mu) \lesssim C_{1}C_{2} \log(1/\delta)$ that we fix from now on.

  Now suppose that $\mathcal{P}'\subset \mathcal{P}$ with $|\mathcal{P}'|\ge \alpha |\mathcal{P}|$. Let $\mu'$ be normalised Lebesgue measure on $\mathcal{P}'$. Then the densities of $\mu',\mu$ satisfy $\mu' \le \alpha^{-1}\mu$, and the same bound carries over to the densities of the projections. Hence,
  \[
     I_{s}(\pi_{\theta}\mu') \le \alpha^{-2} I_{s}(\pi_{\theta}\mu) \lesssim \alpha^{-2} C_{1}C_{2}\log(1/\delta)^{2}.
  \]
  Restricting the energy integral to pairs of points $(x,y)$ in the same element of $\mathcal{D}_{\delta}$ and such that $|\pi_{\theta}(x)-\pi_{\theta}(y)|\ge \delta/10$, we get
  \[
    I_{s}(\pi_{\theta}\mu') \ge \sum_{q\in \mathcal{D}_{\delta}} \iint_{x,y\in q,|\pi_{\theta}(x-y)|\ge \delta/10} \frac{d\pi_{\theta}\mu'(x)d\pi_{\theta}\mu'(y)}{|x-y|^{s}} \gtrsim \delta^{-s} \sum_{q\in \mathcal{D}_{\delta}} (\pi_{\theta}\mu'(q))^2 \ge \delta^{-s} |\pi_{\theta}(\mathcal{P}')|_{\delta}^{-1}.
  \]
  The last inequality follows from Cauchy-Schwarz. Combining the last two estimates yields the conclusion. \end{proof}

We are then prepared to prove Proposition \ref{prop:expansion}:

\begin{proof}[Proof of Proposition \ref{prop:expansion}]
We only prove the claim for $\{ (b_1 - b_2)a + (a_1 - a_2)b: (a,b)\in X\}$. The other cases follow from this by considering $\pm A_{i}$ and $\pm B_{i}$ instead. Given a finite set $F$ let $\delta_F = |F|^{-1}\sum_{x\in F}\delta_x$, and set
\[
\mu_{i} = \delta_{A_{i}} * \delta^{-1}\mathbf{1}_{[-\delta/2,\delta/2]},\quad \nu_{j} = \delta_{B_{j}} * \delta^{-1}\mathbf{1}_{[-\delta/2,\delta/2]}, \qquad i,j \in \{1,2\}.
\]
Then $\mu_{i}$ is $(s, O(\delta^{-\e}))$-Frostman, $\nu_{j}$ is $(t, O(\delta^{-\e}))$-Frostman, and it is enough to show that \eqref{eq:robust-sum-prod} holds outside a set of $(a_1,b_1,a_2,b_2)$ of $(\mu_{1} \times\nu_{1} \times\mu_{2} \times\nu_{2})$-measure $\delta^{\e}$ (for $\e$ small enough in terms of $s,t,\sigma$). Suppose this is not the case and let $E$ be the exceptional set of $(a_1,b_1,a_2,b_2)$. Then $(\mu_{1} \times\nu_{1} \times\mu_{2} \times\nu_{2} )(E)\ge \delta^{\e}$.

We will need to arrange some separation between the supports of $\mu_{1}$ and $\mu_{2}$. Note that if $I \subset [-2,2]$ is any interval of length $\delta^{3\epsilon/s}$, then $\mu_{i}(I) \lesssim \delta^{-\epsilon}\delta^{3\epsilon} = \delta^{2\epsilon}$ for $i \in \{1,2\}$. Consequently, summing over a disjoint cover $\mathcal{I}$ of $[-2,2]$ by such intervals, we find
\begin{displaymath} \sum_{I \in \mathcal{I}} (\mu_{1,I} \times \nu_{1} \times \mu_{2,3I} \times \nu_{2})(E) \leq \Big( \sum_{I \in \mathcal{I}} \mu_{1}(I)^{2} \Big)^{1/2} \Big( \sum_{I \in \mathcal{I}} \mu_{2}(3I)^{2} \Big)^{1/2} \lesssim \delta^{2\epsilon}. \end{displaymath}
Here $\mu_{1,I}$ and $\mu_{2,3I}$ denote the restrictions of $\mu_{1},\mu_{2}$ to $I,3I$, respectively. Since however the total product measure of $E$ exceeds $\delta^{\epsilon}$, this implies
\begin{displaymath} \sum_{I \in \mathcal{I}} (\mu_{1,I} \times \nu_{1} \times \mu_{2,(3I)^{c}} \times \nu_{2})(E) \gtrsim \delta^{\epsilon}. \end{displaymath}
Noting that $|\mathcal{I}| \lesssim \delta^{-3\epsilon/s}$, we may pigeonhole an interval $I_{0} \in \mathcal{I}$ with the property
\begin{displaymath} 
  (\mu_{1,I_{0}} \times \nu_{1} \times \mu_{2,(3I_{0})^{c}} \times \nu_{2})(E) \gtrsim \delta^{\epsilon + 3\epsilon/s}. 
\end{displaymath}
In particular, we have $\min\{\mu_{1}(I_{0}),\mu_{2}((3I_{0})^{c})\} \geq \delta^{4\epsilon/s}$, provided $\delta > 0$ is small enough. Let $\bar{\mu}_{1},\bar{\mu}_{2}$ be the restrictions of the measures $\mu_{1},\mu_{2}$ to the sets $I_{0}$ and $(3I_{0})^{c}$, normalised to have unit mass. Thus, $\bar{\mu}_{1},\bar{\mu}_{2}$ are $(s,O(\delta^{-4\epsilon/s}))$-Frostman probability measures. Furthermore, $(\bar{\mu}_{1} \times \nu_{1} \times \bar{\mu}_{2} \times \nu_{2})(E) \geq \delta^{4\epsilon/s}$, and
\begin{equation}\label{sep:mu1mu2} \dist(\spt(\bar{\mu}_{1}),\spt(\bar{\mu}_{2})) \geq \delta^{3\epsilon/s} \geq \delta^{6\epsilon/s}. \end{equation}
Let $\rho_i := \bar{\mu}_{i} \times\nu_{i}$ for $i \in \{1,2\}$. Then $(\rho_{1} \times \rho_{2})(E) \geq \delta^{4\epsilon/s}$ and, making $\delta$ smaller if needed, both $\rho_1, \rho_2$ are $(s+t, \delta^{-6\e/s})$-Frostman measures. Moreover, due to the product structure of $\rho_i$,
\[
\rho_i(T) \leq \delta^{-6\e/s}\cdot r^{\min(s,t)}\quad\text{for all $r$-tubes $T$}.
\]
We then apply Proposition \ref{prop:thin-tubes} to obtain $K\ge 2,\e_1>0,\delta_1>0$ depending on $s,t,\sigma$  such that if $\delta\in (0,\delta_1)$ and $\e\in (0,\e_1)$, then $(\rho_1,\rho_2)$ have $(\sigma_0,\delta^{-6K\e}, 1-K\delta^{6\e/s})$-thin tubes, where
\[
\sigma_0 = \frac{\min\{s+t,1\}+\sigma}{2} \in (\sigma,\min\{s+t,1\}),
\]
Let $H$ be the set in the definition of thin tubes. Since $(\rho_1\times\rho_2)(H)\ge 1- K\delta^{6\e/s}$, we have
\begin{align*}
(\rho_1\times \rho_2)(H\cap E) &\geq (\rho_{1} \times \rho_{2})(H) + (\rho_{1} \times \rho_{2})(E) - 1 \\
&\geq -K\delta^{6\epsilon/s} + \delta^{4\epsilon/s} \gtrsim \delta^{4\epsilon/s}.
\end{align*}
In particular, there exists $(a_2,b_2)\in A_{2} \times B_{2}$ such that
\[
\rho_{1}(F)\gtrsim \delta^{4\e/s} \quad\text{where } F:= \{ (a_1,b_1): (a_1,b_1,a_2,b_2)\in H\cap E\}.
\]
By the definition of thin tubes, the measure $\pi^{(a_{2},b_{2})}((\rho_1)_F)$ is $(\sigma_0,O(\delta^{-K\e}))$-Frostman. By the separation \eqref{sep:mu1mu2} between $a_1$ and $a_2$, the projection $\lambda$ of $(\rho_1)_{F}$ under $(x,y)\mapsto (y-b_2)/(x-a_2)$ is also $(\sigma_0,O(\delta^{-K\e}))$-Frostman (possibly with a larger "$K$", but one which still only depends on $s,t,\sigma$).

Applying Proposition \ref{prop7} to $\mathcal{P} := \mathcal{D}_{\delta}(P)$ and $\lambda$, with $\alpha := \delta^{\epsilon}$, we deduce that there is $(b_1-b_2)/(a_1-a_2)\in\spt(\lambda)$ (in particular, $(a_1,b_1,a_2,b_2)\in E$) such that
\[
\left|\{  b + (b_1-b_2)/(a_1-a_2) a : (a,b)\in X \} \right|_{\delta} \gtrsim_{\epsilon} \delta^{-\sigma_0+O(K)\e} \quad\text{whenever } |X|_{\delta} \ge \delta^{\e}|P|_{\delta}.
\]
Since $|a_1 - a_2| \ge \delta^{6\epsilon/s}$, this is a contradiction provided $\e$ is small enough in terms of $s,t,\sigma$, assuming that $\sigma_0-O(K)\e > \sigma+\e$ and $\delta$ is small enough in terms of all parameters.
\end{proof}

\begin{remark}
The proof actually shows the following stronger fact: if $A_i$ are instead $(\delta,s_i,\delta^{-\e})$-sets and $B_i$ are $(\delta,t_i,\delta^{-\e})$-sets with $s_i,t_i\in (0,1)$ (and $P$ remains as in the statement), then the same conclusion holds, with $\e$ depending on all parameters, provided that
\[
\sigma \in [0, \min\{s_1+t_1,s_2+t_2,s+t,1\}).
\]
\end{remark}

\begin{remark}
The proof of Proposition \ref{prop:expansion} also yields the following measure-theoretic statement: given $s,t\in (0,1)$ and $\sigma\in [0,\min\{s+t,1\})$, there exist $\e = \e(s,t,\sigma) > 0$ and $\delta_{0} = \delta_{0}(s,t,\sigma,\e) > 0$  such that the following holds for all $\delta \in (0,\delta_{0}]$. Assume that $\mu_1,\mu_2$ are  $(\delta,s,\delta^{-\e})$-Frostman measures on $[-2,2]$ and $\nu_1,\nu_2$ are $(\delta,t,\delta^{-\e})$-Frostman measures on $[-2,2]$. Assume that $\boldsymbol{\mu}$ is a $(\delta,s + t,\delta^{-\epsilon})$-Frostman measure on $B(1) \subset \R^{2}$. Then there exists a set $\mathcal{G}\subset\R^4$ with
\[
(\mu_1\times\mu_2\times\nu_1\times\nu_2)(\R^4 \, \setminus \, \mathcal{G})\le \delta^{\e}  
\]
such that if $(a_1,a_2,b_1,b_2)\in \mathcal{G}$ and $\boldsymbol{\mu}(X)\ge \delta^{\e}$, then
\[
\big|\{  (a_1-a_2)b + (b_1-b_2) a : (a,b)\in X \} \big|_{\delta} \ge \delta^{-\sigma} .
\]
\end{remark}

\subsection{Proof of Theorem \ref{thm:ABCWeak}} We start by recalling a key tool in additive combinatorics:
\begin{lemma}[Pl\"{u}nnecke-Ruzsa inequalities]  \label{l:PR}
Let $A, B_1,\ldots, B_n\subset\R$ be finite sets and fix $\delta>0$. Assume that $|A+B_i|_{\delta}\le K_i |A|_{\delta}$ for all $i=1,\ldots,n$. Then there is $A'\subset A$, $|A'|_{\delta}\ge |A|_{\delta}/2$ such that
\[
|A'+B_1+\cdots+B_n|_{\delta} \lesssim_n K_1\cdots K_n \cdot|A|_{\delta}.
\]
\end{lemma}

This form of the inequality is due to Ruzsa \cite{MR2314377}.
To be accurate, the statement is not formulated in terms of $\delta$-covering numbers, but the reduction is not difficult: one may consult \cite[Corollary 3.4]{MR4283564}.

\begin{lemma}\label{l:ABC} Let $0 < \beta \leq \alpha < 1$ and $\gamma \in (\alpha - \beta,1]$, and let $A,B,C$ be as in Theorem \ref{thm:ABCWeak}. Suppose the conclusion of Theorem \ref{thm:ABCWeak} does not hold for any $c\in C$, for $\chi,\delta > 0$ sufficiently small in terms of $\alpha,\beta,\gamma$, in particular $\chi \leq (\beta + \gamma - \alpha)/4$. Then there is a set  $E\subset B \times B$ with $|E|\le \delta^{\chi}|B|^2$ such that for each $(b_1,b_2)\in (B \times B) \, \setminus \, E$,
\[
|A'+(b_1-b_2)C'|_{\delta} \ge \delta^{-(\beta+\gamma-\alpha)/4}|A| \geq \delta^{-\chi}|A|
\]
for all subsets $A'\subset A$, $C'\subset C$ with $|A'\times C'|\ge \delta^{\chi}|A||C|$.
\end{lemma}
\begin{proof}
Assume that the conclusion \eqref{conclusion} does not hold for any $c\in C$. Recall that $\gamma \in (\alpha - \beta,1]$, and let
\begin{displaymath}
\sigma := \min\left\{\frac{\alpha+\beta+\gamma}{2},\frac{1 + \alpha}{2}\right\} \in (\alpha,\min\{\beta + \gamma,1\}).
\end{displaymath}
Apply Proposition \ref{prop:expansion} with constant $\sigma$, and the sets $B$ and $C$, which are assumed to be a $(\delta,\beta,\delta^{-\chi})$-set and a $(\delta,\gamma,\delta^{-\chi})$-set, respectively, in Theorem \ref{thm:ABCWeak} (more precisely, we apply Proposition \ref{prop:expansion} in the case where the sets $A_{1},A_{2}$ coincide with $B$, the sets $B_{1},B_{2}$ coincide with $C$, and the set $P$ coincides with $C\times B$). Assume that $\chi\le \e(\alpha,\beta,\gamma)$, the constant given by the proposition. It follows from Proposition \ref{prop:expansion} that there are $(c_1,c_2)\in C^2$ and a set $E\subset B \times B$ with $|E|\le \delta^{\chi}|B|^2$ such that
\begin{equation} \label{eq:PR-1}
|(b_1-b_2)C' +  c_1 B+c_2 B |_{\delta}\ge  \delta^{-\sigma}, \quad C'\subset C, |C'|\ge \delta^{\chi}|C|,\, (b_1,b_2)\in (B \times B) \, \setminus \, E.
\end{equation}
Since  \eqref{conclusion} does not hold for either $c_1$ or $c_2$, we know that
\begin{equation} \label{eq:PR-2}
|A + c_i B|_{\delta}\le \delta^{-\chi}|A|, \quad i=1,2.
\end{equation}
Fix $A'\times C'\subset A\times C$ with $|A'\times C'|\ge \delta^{\chi}|A||C|$, and $(b_1,b_2)\in (B \times B) \, \setminus \, E$. Suppose that 
\[
  |A' + (b_1-b_2)C'|_{\delta} < K|A|.
\] 
Combining the Pl\"{u}nnecke-Ruzsa inequality (Lemma \ref{l:PR}) with \eqref{eq:PR-2}, we find $A''\subset A'$ with $|A''|_{\delta}\ge |A'|_{\delta}/2$ such that
\[
   |A'' +(b_1-b_2)C' + c_1 B + c_2 B|_{\delta} \lesssim K \delta^{-2\chi}|A|.
\]
Plugging this into \eqref{eq:PR-1}, we obtain $K\lesssim \delta^{-(\sigma+2\chi)}$. Taking $\chi<(\sigma-\alpha)/20$, we deduce that
\[
|A' + (b_1-b_2)C'|_{\delta} \ge \delta^{-(\sigma-\alpha)/2}|A|.
\]
This is the desired conclusion.
\end{proof}

We are then nearly ready to prove Theorem \ref{thm:ABCWeak}. Before doing this, we recap a few facts about \emph{restricted sumsets} and \emph{additive energies}.

Let $\delta \in (0,1]$ and $A,B \subset \R$. The \emph{$\delta$-discretised additive energy} between $A,B$ is the quantity
\begin{displaymath} \mathcal{E}_{\delta}(A,B) := |\{(a_{1},a_{2},b_{1},b_{2}) \in A^{2} \times B^{2} : |(a_{1} + b_{1}) - (a_{2} + b_{2})| \leq \delta\}|. \end{displaymath}
If $A$ is $\delta$-separated (but $B \subset \R$ may be arbitrary), then $\mathcal{E}_{\delta}(A,B) \leq 2|A||B|^{2}$, since for $(a_{2},b_{1},b_{2}) \in A \times B^{2}$ fixed, the constraint $|(a_{1} + b_{1}) - (a_{2} + b_{2})| \leq \delta$ determines an interval of length $2\delta$ where $a_{1}$ needs to lie.

\begin{remark}\label{rem4}
Assume $A$ is $\delta$-separated, and there exists a subset $\mathcal{G} \subset A \times B$ such that $|\mathcal{G}| \geq |A||B|/K$, and the \emph{restricted sumset} $A +_{\mathcal{G}} B := \{a + b : (a,b) \in \mathcal{G}\}$ satisfies $|A +_{\mathcal{G}} B|_{\delta} \leq K|A|$, $K \geq 1$. Let $H \subset A +_{\mathcal{G}} B$ be a maximal $\delta$-separated set, thus $|H| \leq K|A|$. Then, by the triangle inequality and Cauchy-Schwarz,
\begin{equation}\label{form104} \mathcal{E}_{2\delta}(A,B) \geq \sum_{h \in H} |\{(a,b) \in \mathcal{G} : |(a + b) - h| \leq \delta\}|^{2} \geq |A||B|^{2}/K^{2}. \end{equation}
Conversely, assume that $\mathcal{E}_{\delta}(A,B) \geq |A||B|^{2}/K$, and also that $A$ is $\delta$-separated; this guarantees that
\begin{equation}\label{form34} \sum_{(a_{1},b_{2}) \in \mathcal{G}} |\{(a_{2},b_{2}) \in A \times B : |(a_{1} + b_{1}) - (a_{2} + b_{2})| \leq \delta\}| \leq 2|\mathcal{G}||B| \end{equation}
for all $\mathcal{G} \subset A \times B$. In particular, this implies that the "popular pairs"
\begin{displaymath} \mathcal{G} := \{(a_{1},b_{1}) : |\{(a_{2},b_{2}) \in A \times B : |(a_{1} + b_{1}) - (a_{2} + b_{2})| \leq \delta\}| \geq |B|/(2K)\} \end{displaymath}
satisfy $|\mathcal{G}| \geq |A||B|/(4K)$, since otherwise $\mathcal{E}_{\delta}(A,B) < |A||B|^{2}/K$, using \eqref{form34}. On the other hand, letting again $H \subset A +_{\mathcal{G}} B$ be a maximal $\delta$-separated set, it holds
\begin{equation}\label{form105}|A +_{\mathcal{G}} B|_{\delta} = |H| \leq \frac{2K}{|B|} \sum_{h \in H} |\{(a_{2},b_{2}) \in A \times B : |(a_{2} + b_{2}) - h| \leq \delta\}| \leq 2K|A|.  \end{equation}
Summarising \eqref{form104}-\eqref{form105}, one can roughly conclude that if $A$ is $\delta$-separated, then $\mathcal{E}_{\delta}(A,B) \approx |A||B|^{2}$ if and only if there exists $\mathcal{G} \subset A \times B$ with $|\mathcal{G}| \approx |A||B|$ such that $|A +_{\mathcal{G}} B|_{\delta} \lessapprox |A|$.
\end{remark}

Finally, we will need the asymmetric Balog-Szemer\'edi-Gowers theorem, see the book of Tao and Vu, \cite[Theorem 2.35]{MR2289012}. We state the result in the following slightly weaker form (following \cite[Theorem 3.2]{Sh}):
\begin{thm}[Asymmetric Balog-Szemer\'edi-Gowers theorem]\label{thm:BSG} Given $\eta > 0$, there exists $\zeta > 0$ such that the following holds for $\delta \in 2^{-\N}$ small enough. Let $A,B \subset (\delta \cdot \Z) \cap [0,1]$, and assume that there exist $c \in [\delta^{\zeta},1]$ and $\mathcal{G} \subset A \times B$ satisfying
\begin{equation} \label{eq:BSG-assump}
   |\mathcal{G}| \geq \delta^{\zeta}|A||B| \quad \text{and} \quad |\{a + cb : (a,b) \in \mathcal{G}\}|_{\delta} \leq \delta^{-\zeta}|A|. 
\end{equation}
Then there exist subsets $A' \subset A$ and $B' \subset B$ with the properties
\begin{displaymath} 
    |A'||B'| \geq \delta^{\eta}|A||B| \quad \text{and} \quad |A' + cB'|_{\delta} \leq \delta^{-\eta}|A|. 
\end{displaymath}
\end{thm}
\begin{proof}
  We explain how to derive this formulation from \cite[Theorem 3.2]{Sh}. Given $x\in\R$, let $(x)_{\delta} \in \delta \cdot \Z$ be such that $x \in [(x)_{\delta},(x)_{\delta}+\delta)$. Let
  \[
    \widetilde{B}=\bigl\{ \delta j : [\delta j, \delta (j+1)) \cap cB \neq \emptyset\bigr\} \subset \delta \cdot \Z
  \]
  be the $\delta$-discretisation of $cB$. Likewise, let
  \[
    \widetilde{\mathcal{G}} = \{(a, (cb)_{\delta}) : (a,b) \in \mathcal{G}\} \subset A \times \widetilde{B}.
  \]
  Using the assumptions $c\ge \delta^{\zeta}$ and \eqref{eq:BSG-assump}, one checks that
  \[
    |\widetilde{\mathcal{G}}| \gtrsim \delta^{2\zeta}|A||\widetilde{B}|, \quad |\{a + b : (a,b) \in \widetilde{\mathcal{G}}\}| \lesssim \delta^{-2\zeta}|A|.
  \]
  By Cauchy-Schwarz,
  \[  
    \mathcal{E}(A,\widetilde{B}) \gtrsim \delta^{4\zeta}|A||\widetilde{B}|^{2}.
  \]
  We can now apply \cite[Theorem 3.2]{Sh} to the sets $A,\widetilde{B}$ to conclude that, if $\zeta$ is small enough in terms of $\eta$, and $\delta$ is small enough in terms of all other parameters, then there are $X,H$ such that, setting $A'=A\cap (X+H)$ and $\widetilde{B}'=\widetilde{B}\cap H$, we have 
  \[
    |A'||\widetilde{B}'| \geq \delta^{2\eta}|A||\widetilde{B}|, 
  \]
  and
  \[ 
  |A' + \widetilde{B}'| \leq |X+H+H| \le |X||H+H|  \le \delta^{-\eta}|X||H| \leq \delta^{-2\eta}|A|.
  \]
  One can now check that the conclusion holds with $B' = \{b \in B : (cb)_{\delta} \in \widetilde{B}'\}$ and slightly smaller $\eta$.
\end{proof}

We are then prepared to prove Theorem \ref{thm:ABCWeak}.

\begin{proof}[Proof of Theorem \ref{thm:ABCWeak}] We fix two parameters $0 < \chi \ll \chi_{0} \leq 1$, both depending only on $\alpha,\beta,\gamma$. In fact, $\chi_{0} > 0$ is the constant provided by Lemma \ref{l:ABC} applied to the triple $(\alpha,\beta,\gamma)$. The relationship between $\chi,\chi_{0}$ is determined by Theorem \ref{thm:BSG}: we will need that $C\chi$ is smaller than the constant $\zeta = \zeta(\chi_{0}/2) > 0$ provided by Theorem \ref{thm:BSG} applied with $\eta := \chi_{0}/2$.

We now make a counter assumption:
\begin{equation}\label{form106} |A + cB|_{\delta} \leq \delta^{-\chi}|A| \leq \delta^{-\chi_{0}}|A|\quad\text{for all } c \in C. \end{equation}
By the choice of $\chi_{0}$, Lemma \ref{l:ABC} outputs an exceptional set $E \subset B \times B$ with $|E| \leq \delta^{\chi_{0}}|B|^{2}$ such that
\begin{equation}\label{form107} |A' + (b_{2} - b_{1})C'|_{\delta} \geq \delta^{-\chi_{0}}|A|\quad\text{for all } (b_{1},b_{2}) \in (B \times B) \, \setminus \, E, \end{equation}
and for all subsets $A' \subset A$ and $C' \subset C$ with $|A' \times C'| \geq \delta^{\chi_{0}}|A||C|$.

In the remaining of the proof, we use "$\approx$" notation to hide constants of the form $\delta^{-O(\chi)}$. From \eqref{form104} and \eqref{form106}, we deduce that
\begin{equation}\label{form108} \mathcal{E}_{2\delta}(A,cB) \approx |A||B|^{2}, \quad c \in C. 
\end{equation}

Note that $|(a + cb_{1}) - (a' + cb_{2})| \leq 2\delta$ is equivalent to $|a - (a' + c(b_{2} - b_{1}))| \leq 2\delta$. Therefore,
\begin{displaymath} \sum_{b_{1},b_{2}} |\{(a,a',c) \in A^{2} \times C : |a - (a' + c(b_{2} - b_{1}))| \leq 2\delta\}| \approx |A||B|^{2}|C|. \end{displaymath}
Since we have the uniform upper bound
\begin{displaymath}  |\{(a,a',c) \in A^{2} \times C : |a - (a' + c(b_{2} - b_{1}))| \leq 2\delta\}| \lesssim |A||C|\quad\text{for all } b_{1},b_{2} \in B, \end{displaymath}
by the $\delta$-separation of $A$, we may find a set $G \subset B \times B$ with $|G| \approx |B|^{2}$ such that
\begin{equation}\label{form94} |\{(a,a',c) : |a - (a' + c(b_{2} - b_{1}))| \leq 2\delta\}| \approx |A||C|\quad\text{for all } (b_{1},b_{2}) \in G. \end{equation}
Since $|E| \leq \delta^{\chi_{0}}|B|^{2}$ (recall below \eqref{form107}), and $|G| \geq \delta^{O(\chi)}|B|^{2}$, we may ensure
\begin{displaymath} |G \cap ((B \times B) \, \setminus \, E)| \geq \tfrac{1}{2}|G| \approx |B|^{2} \end{displaymath}
by taking $C\chi \leq \chi_{0}$. We fix $(b_{1},b_{2}) \in G \cap ((B \times B) \, \setminus \, E)$ with $|b_{1} - b_{2}| \approx 1$ for the remainder of the argument. Then, using the $\delta$-separation of $A$ in "$\gtrsim$",
\begin{align*} & \mathcal{E}_{4\delta}(A,(b_{2} - b_{1})C) \stackrel{\mathrm{def.}}{=} |\{(a_{1}',a_{2}',c_{1},c_{2}) : |(a_{1}' + c_{1}(b_{2} - b_{1})) - (a_{2}' + c_{2}(b_{2} - b_{1}))| \leq 4\delta\}|\\
& \gtrsim \sum_{a \in A} |\{(a_{1}',a_{2}',c_{1},c_{2}) : |(a_{1}' + c_{1}(b_{2} - b_{1})) - a| , |(a_{2}' + c_{2}(b_{2} - b_{1})) - a| \leq 2\delta\}|\\
 & = \sum_{a \in A} |\{(a',c) : |a - (a' + c(b_{2} - b_{1}))| \leq 2\delta\}|^{2}\\
& \geq \frac{1}{|A|}\Big( \sum_{a \in A} |\{(a',c) : |a - (a' + c(b_{2} - b_{1}))| \leq 2\delta\}| \Big)^{2} \stackrel{\eqref{form94}}{\approx} |A||C|^{2}. \end{align*}

By Remark \ref{rem4}, now applied to $(A,(b_{2} - b_{1})C)$ instead of $(A,cB)$, this implies that there exists a subset $\mathcal{G} \subset A \times C$ with $|\mathcal{G}| \approx |A||C|$ such that
\[
|\{a + (b_{2} - b_{1})c : (a,c) \in \mathcal{G}\}|_{\delta} \approx |A|.
\]
By Theorem \ref{thm:BSG}, since $|b_{2} - b_{1}| \approx 1$, and the choice of "$\chi$" at the beginning of the proof, this further implies the existence of subsets $A' \subset A$ and $C' \subset C$ such that
\begin{displaymath} |A'| \ge \delta^{\chi_{0}/2}|A| \quad \text{and} \quad |C'| \geq \delta^{\chi_{0}/2}|C|, \end{displaymath}
and $|A' + (b_{2} - b_{1})C'|_{\delta} \leq \delta^{-\chi_{0}/2}|A|$. Since $(b_{1},b_{2}) \in (B \times B) \, \setminus \, E$, this contradicts \eqref{form107}, and completes the proof of Theorem \ref{thm:ABCWeak}. \end{proof}

\section{Projections of regular sets}\label{s:ProjRegular}

\subsection{High multiplicity sets, and a \texorpdfstring{$\delta$}{delta}-discretised version} 

In this section we state a $\delta$-discretised version of the part of Theorem \ref{thm:projection} concerning sets of equal Hausdorff dimension -- Theorem \ref{thm1} below.  The proof of Theorem \ref{thm:projection}, assuming Theorem \ref{thm1}, is deferred to \S\ref{s:proof-of-proj-thm}. 

We repeat the definition of regular sets (Definition \ref{def:deltaTRegularSet}) for the convenience of the reader.
\begin{definition}[$(\delta,t,C)$-regular set]\label{def:deltaTRegularSet2} Let $t > 0$ and $C \geq 1$. A non-empty set $\mathcal{P} \subset \mathcal{D}_{\delta}$ is called \emph{$(\delta,t,C)$-regular} if
\begin{enumerate}[(i)]
\item $\mathcal{P}$ is a $(\delta,t,C)$-set, and
\item $|\mathcal{P} \cap \mathbf{p}|_{r} \leq C(R/r)^{t}$ for all dyadic $\delta \leq r \leq R < \infty$, and for all $\mathbf{p} \in \mathcal{D}_{R}$.
\end{enumerate}
\end{definition}
Note that if property (ii) holds, then $\mathcal{P}$ is automatically a $(\delta,t,C \delta^{-t}|\mathcal{P}|_{\delta}^{-1})$-set. We also need the corresponding notion for measures:
 \begin{definition}[$(t,C)$-regular and $(\delta,t,C)$-regular measures]\label{def:regularity} Let $t > 0$ and $C \geq 1$. A non-trivial Borel measure $\mu$ with $K := \spt \mu \subset \R^{d}$ is called $(t,C)$-regular if
 \begin{enumerate}[(i)]
 \item $\mu$ is a $(t,C)$-Frostman measure, and
 \item $|K \cap B(x,R)|_{r} \leq C (R/r)^{t}$ for all $x \in \R$ and $0 < r \leq R < \infty$.
 \end{enumerate}

 If $\delta \in (0,1]$, and $\mu$ is a $(\delta,t,C)$-Frostman measure (Definition \ref{def:FrostmanMeasure}) satisfying property (ii) for all $\delta \leq r \leq R < \infty$, we say that $\mu$ is $(\delta,t,C)$-regular.
\end{definition}

We note that if $\mathcal{P} \subset \mathcal{D}_{\delta}$ is $(\delta,t,C)$-regular with $t \leq d$, then the unit-normalised Lebesgue measure on $\cup \mathcal{P}$ is a $(t,C')$-regular measure with $C' \sim C$.

\begin{remark}\label{rem:regular} If $\mu$ is a $(t,C)$-regular measure on $\R^{d}$, and $B = B(z,r) \subset \R^{2}$ is a disc, we write 
\begin{displaymath} \mu_{B} := r^{-t} \cdot T_{B}\mu, \end{displaymath}
where $T_{B}(x) = (x - z)/r$ is the homothety mapping $B$ to $B(1)$. We emphasise that we do not restrict the rescaled measure to the unit ball here (in contrast to Definition \ref{def:renormalisation} concerning sets). Then $\mu_{B}$ is also a $(t,C)$-regular measure on $\R^{d}$. More accurate notation would be $\mu_{B,t}$, but the index "$t$" should always be clear from context. Similarly, if $\mu$ is $(\delta,t,C)$-regular with $\delta \in (0,1]$, and $r \in [\delta,1]$, then $\mu_{B}$ is $(\delta/r,t,C)$-regular.

In contrast, if $\mathcal{P} \subset \mathcal{D}_{\delta}$ is a $(\delta,t,C)$-regular set, and $Q \in \mathcal{D}_{\Delta}$ with $\delta \leq \Delta \leq 1$, it is not always true that the renormalisation $\mathcal{P}_{Q} \subset \mathcal{D}_{\delta/\Delta}$ (Definition \ref{def:renormalisation}) is $(\delta/\Delta,t,O(C))$-regular. Property (ii) in Definition \ref{def:deltaTRegularSet2}  is scale-invariant, but property (i) may fail for $\mathcal{P}_{Q}$ in a case where $|\mathcal{P} \cap Q| \ll \Delta^{t}|P|$. In fact, it is easy to check that if $\mathcal{P}$ is $(\delta,t,C)$-regular, and $\mathcal{P} \cap Q \neq \emptyset$, then $\mathcal{P}_{Q}$ is $(\delta/\Delta,t,\bar{C})$-regular with
\begin{displaymath} 
  \bar{C} = C \cdot \frac{\Delta^{t}|\mathcal{P}|}{|\mathcal{P} \cap Q|}. 
\end{displaymath} 
\end{remark}

Next we define the notion of "high multiplicity". This terminology is taken verbatim from \cite{MR4388762}.
\begin{definition}\label{def:mult} Let $K \subset \R^{2}$, let $0 < r \leq R \leq \infty$ be dyadic numbers, and let $x \in K$. Recall that $K^{(r)}$ stands for the $r$-neighbourhood of $K$. For $\theta \in S^{1}$, we define the following \emph{multiplicity number}:
\begin{displaymath} \m_{K,\theta}(x \mid [r,R]) := |B(x,R) \cap K^{(r)} \cap \pi_{\theta}^{-1}\{\pi_{\theta}(x)\}|_{r}.  \end{displaymath}
 Thus, $\m_{K,\theta}(x \mid [r,R])$ keeps track of the (smallest) number of dyadic $r$-squares needed to cover the intersection between $B(x,R) \cap K^{(r)}$ and the line $\pi_{\theta}^{-1}\{\pi_{\theta}(x)\}$. Often the set "$K$" is clear from the context, and we abbreviate $\m_{K,\theta} =: \m_{\theta}$. \end{definition}

\begin{definition}[High multiplicity sets]\label{def:highMult} Let $0 < r \leq R \leq \infty$, $M > 0$, and let $\theta \in S^{1}$. For $K \subset \R^{2}$, we define the \emph{high multiplicity set}
\begin{displaymath} H_{\theta}(K,M, [r,R]) := \{x \in K : \m_{K,\theta}(x \mid [r,R]) \geq M\}. \end{displaymath}  \end{definition}

The next lemma discusses how the high multiplicity sets are affected by scalings. For $z_{0} \in \R^{2}$ and $r_{0} > 0$, we write $T_{z_{0},r_{0}}$ for the homothety which sends $B(z_{0},r_{0})$ to $B(1)$, namely $T_{z_{0},r_{0}}(z) := (z - z_{0})/r_{0}$ for $z \in \R^{2}$.

\begin{lemma}\label{lemma7} Let $K \subset \R^{2}$ be arbitrary, let $0 < r \leq R \leq \infty$, $M > 0$, and $\theta \in [0,1]$. Then,
\begin{displaymath} T_{z_{0},r_{0}}(H_{\theta}(K,M,[r,R])) = H_{\theta}(T_{z_{0},r_{0}}(K),M,[\tfrac{r}{r_{0}},\tfrac{R}{r_{0}}])\quad\text{for all } z_{0} \in \R^{2}, \, r_{0} > 0. \end{displaymath}
\end{lemma}

\begin{proof} This follows from the definitions, or see \cite[Lemma 2.11]{MR4388762}. \end{proof}

Here is the $\delta$-discretised version of Theorem \ref{thm:projection} when $\Hd K=\Pd K$:

\begin{thm}\label{thm1} Let $t \in (0,2)$ and $s \in (0,\min\{t,2 - t\})$. For every $\sigma > (t - s)/2$ there exist $\epsilon,\delta_{0} > 0$ such that the following holds for all $\delta \in (0,\delta_{0}]$. Let $\mu$ be $(\delta,t,\delta^{-\epsilon})$-regular measure, and let $E \subset S^{1}$ be a $(\delta,s,\delta^{-\epsilon})$-set. Then, there exists $\theta \in E$ such that
\begin{equation}\label{form7}
\mu\big(B(1) \cap H_{\theta}(\spt(\mu),\delta^{-\sigma},[\delta,1])\big) \leq \delta^{\epsilon}.
\end{equation}
 \end{thm}

The formulation of Theorem \ref{thm1} is convenient to prove, but the following version will be more convenient to apply (e.g. in the context of Furstenberg sets):
\begin{cor}\label{cor2} Let $s \in (0,1]$ and $t \in [s,2]$. For every $0 \leq u < \min\{(s + t)/2,1\}$, there exist $\epsilon,\delta_{0} > 0$ such that the following holds for all $\delta \in (0,\delta_{0}]$. Let $\mathcal{P} \subset \mathcal{D}_{\delta}$ be a $(\delta,t,\delta^{-\epsilon})$-regular set, and let $E \subset S^{1}$ be a $(\delta,s,\delta^{-\epsilon})$-set. Then, there exists $\theta \in E$ such that
\begin{equation}\label{eq:cor-proj} 
  |\pi_{\theta}(\mathcal{P}')|_{\delta} \geq \delta^{-u}\quad\text{for all } \mathcal{P}' \subset \mathcal{P}, \, |\mathcal{P}'| \geq \delta^{\epsilon}|\mathcal{P}|. 
\end{equation}
\end{cor}

\begin{remark} \label{rem:cor2} In applications, it is sometimes useful to notice that Corollary \ref{cor2} is formally equivalent to the superficially stronger conclusion that \eqref{eq:cor-proj} holds for all $\theta \in E'$, where $E' \subset E$ is subset with $|E'|_{\delta} \geq \tfrac{1}{2}|E|_{\delta}$. The reason is simply that if this failed, we might re-apply Theorem \ref{thm1} to $E'$. \end{remark}

\begin{proof}[Proof of Corollary \ref{cor2}] Fix $s \in [0,2]$, $t \in [s,2]$, and $0 \leq u < \min\{(s + t)/2,1\}$. Let $\mathcal{P},E$ be as in the statement, with $\epsilon,\delta > 0$ sufficiently small (the constraints will be essentially those inherited from Theorem \ref{thm1}).

We first dispose of a special case, where $\min\{(s + t)/2,1\} = 1$, or in other words $s \geq 2 - t$. In that case $u < 1$, so it is possible to choose a new parameter $\bar{s} < s$ with $\bar{s} < 2 - t$ such that still $u < (\bar{s} + t)/2$. Now, we note that $E$ is also a $(\delta,\bar{s},\delta^{-\epsilon})$-set. So, it suffices to consider the case where $s < 2 - t$. With similar arguments, it suffices to consider the case where we have a strict inequality $s < t$.

Now, pick parameters $\sigma' > \sigma > (t - s)/2$ so that $u < t - \sigma'$. Then, apply Theorem \ref{thm1} to the measure $\mu$ obtained by unit-normalising Lebesgue measure to $\cup \mathcal{P}$, and with the parameter $\sigma$. The outcome is a vector $\theta \in E$ such that
\begin{displaymath} |\mathcal{P} \cap H_{\theta}(\cup \mathcal{P},\delta^{-\sigma},[\delta,1])| \leq \delta^{\epsilon}|\mathcal{P}|. \end{displaymath}
Here, we may assume $\epsilon > 0$ is so small that $t - \sigma - 2\epsilon > t - \sigma' > u$.

Finally, let $\mathcal{P}' \subset \mathcal{P}$ be any subset with $|\mathcal{P}'| \geq \delta^{\epsilon/2}|\mathcal{P}|$. Then, writing
\begin{displaymath} \mathcal{G}' := \mathcal{P}' \, \setminus \, H_{\theta}(\cup \mathcal{P},\delta^{-\sigma},[\delta,1]), \end{displaymath}
we have $|\mathcal{G}'| \geq \tfrac{1}{2}|\mathcal{P}'| \gtrsim \delta^{\epsilon/2}|\mathcal{P}|$. As a technical point, we may assume that $\diam(\mathcal{G}') \leq \tfrac{1}{2}$ by selecting a square $Q_{0} \in \mathcal{D}_{1/8}$ with $|\mathcal{G}' \cap Q_{0}| \sim |\mathcal{G}'|$.

We claim that $|\pi_{\theta}(\mathcal{G}')|_{\delta} \geq \delta^{-u}$, which will complete the proof of the corollary. To see this, fix $x \in \mathcal{G}'$. Then,
\begin{equation}\label{form111} |(\cup \mathcal{G}') \cap \pi_{\theta}^{-1}\{\pi_{\theta}(x)\}|_{\delta} \leq \mathfrak{m}_{\cup \mathcal{P},\theta}(x \mid [\delta,1]) \leq \delta^{-\sigma} \end{equation}
(We used the assumption $\diam(\mathcal{G}') \leq \tfrac{1}{2}$ to be able to omit the intersection with $B(x,1)$ on the left-hand side.) The inequality \eqref{form111} implies that every fibre of the projection $\pi_{\theta}$ intersects $\lesssim \delta^{-\sigma}$ squares in $\mathcal{G}'$. Since $|\mathcal{G}'| \gtrsim \delta^{\epsilon/2}|\mathcal{P}| \geq \delta^{2\epsilon - t}$, this yields
\begin{displaymath} |\pi_{\theta}(\mathcal{G}')|_{\delta} \gtrsim \delta^{\sigma}|\mathcal{G}'| \geq \delta^{2\epsilon + \sigma - t}. \end{displaymath}
Since $t - \sigma - 2\epsilon > t - \sigma' > u$, we in particular have $|\pi_{\theta}(\mathcal{G}')|_{\delta} \geq \delta^{-u}$, if $\delta > 0$ is small enough. \end{proof}
\subsection{An inductive scheme to prove Theorem \ref{thm1}}\label{s:induction}

We then begin the proof of Theorem \ref{thm1} by setting up the statement that will be iterated.
\begin{terminology}[Projection$(s,\sigma,t)$]\label{term1} Let $s,\sigma \in (0,1]$ and $t \in (0,2]$. We say that 
\begin{displaymath} \text{Projection}(s,\sigma,t) \end{displaymath}
holds if the conclusion of Theorem \ref{thm1} holds with the parameters $s,\sigma,t$. In other words, there exists $\epsilon,\delta_{0} > 0$ such that whenever $\delta \in (0,\delta_{0}]$, and $\mu$ is a $(\delta,t,\delta^{-\epsilon})$-regular measure, and $E \subset S^{1}$ is a $(\delta,s,\delta^{-\epsilon})$-set, then there exists $\theta \in E$ such that \eqref{form7} is valid. \end{terminology}

\begin{remark}\label{rem2} It is clear that
\begin{displaymath} \mathrm{Projection}(s,\sigma,t) \quad \Longrightarrow \quad \mathrm{Projection}(s,\sigma',t) \text{ for all } \sigma' \geq \sigma, \end{displaymath}
in fact with the same implicit constants "$\delta_{0},\epsilon$". This follows immediately from the inclusion $H_{\theta}(K,\delta^{-\sigma'},[\delta,1]) \subset H_{\theta}(K,\delta^{-\sigma},[\delta,1])$ for $\sigma' \geq \sigma$.  \end{remark}

The proof of Theorem \ref{thm1} will be based on iterating the following proposition:

\begin{proposition}\label{prop2} Let $s,t,\sigma$ satisfy
  \begin{equation} \label{eq:params}
    0 < t < 2, \quad 0 < s < \min\{t,2 - t\},  \quad \text{and} \quad \frac{t - s}{2} < \sigma < \frac{t}{2}.
  \end{equation}
Then, there exists $\zeta = \zeta(s,\sigma,t) > 0$, which is bounded away from $0$ on compact subsets of the parameter space defined by \eqref{eq:params}, such that
\begin{displaymath} \mathrm{Projection } (s,\sigma,t) \quad \Longrightarrow \quad \mathrm{Projection } (s,\sigma - \zeta,t). \end{displaymath}
More precisely, assume that there exist $\epsilon_{0},\Delta_{0} > 0$ such that whenever $\Delta\in (0,\Delta_0)$, $\mu$ is a $(\Delta,t,\Delta^{-\epsilon_{0}})$-regular measure, and $E \subset S^{1}$ is a  $(\Delta,s,\Delta^{-\epsilon_{0}})$-set, then there exists $\theta \in E$ such that
\begin{displaymath}
\mu\big(B(1) \cap H_{\theta}(\spt(\mu),\Delta^{-\sigma},[\Delta,1])\big) \leq \Delta^{\epsilon_{0}}.
\end{displaymath}
Then, there exist $\epsilon = \epsilon(\epsilon_{0},s,\sigma,t) > 0$ and $\delta_{0} = \delta_{0}(\epsilon,\Delta_{0},s,\sigma,t) > 0$ such whenever $\delta\in (0,\delta_0)$, $\mu$ is a $(\delta,t,\delta^{-\epsilon})$-regular measure, and $E \subset S^{1}$ is a  $(\delta,s,\delta^{-\epsilon})$-set, then there exists $\theta \in E$ such that
\begin{displaymath}
\mu\big(B(1) \cap H_{\theta}(\spt(\mu),\Delta^{-\sigma + \zeta},[\delta,1])\big) \leq \delta^{\epsilon}.
\end{displaymath}
\end{proposition}

We then complete the proof of Theorem \ref{thm1} assuming Proposition \ref{prop2}.

\begin{proposition}[Base case] Let $t \in (0,2]$ and $s > 0$. Then, there exists $\eta := \eta(s,t) > 0$ such that $\mathrm{Projection}(s,\tfrac{t}{2} - \eta,t)$ holds. \end{proposition}

\begin{proof} We claim that if $\eta,\delta_0 > 0$ are small enough (depending on $s,t$), the following holds for all $0<\delta\le \delta_0$, all  $(\delta,t,\delta^{-\eta})$-regular measures $\mu$, and all  $(\delta,s,\delta^{-\eta})$-sets $E$:
\begin{equation}
    \label{form92}
\mu\big(B(1) \cap H_{\theta}(\spt(\mu),\delta^{-t/2 + \eta},[\delta,1])\big) \leq \delta^{\eta} 
\end{equation}
for at least one direction $\theta \in E$. Let $B_{\theta} := B(1) \cap H_{\theta}(\spt(\mu),\delta^{-t/2 + \eta},[\delta,1])$. It is fairly straightforward to check from the definition of $H_{\theta}(\ldots)$ and the $(\delta,t,\delta^{-\eta})$-regularity of $\mu$ that
\begin{equation}\label{form91} |\pi_{\theta}(B_{\theta})|_{\delta} \lesssim \delta^{-t/2 - 2\eta}\quad\text{for all } \theta \in E. \end{equation}
We leave this to the reader, although virtually the same details in a slightly more advanced context will be recorded in Lemma \ref{lemma6}. Now, it follows from Bourgain's projection theorem \cite[Theorem 3]{Bo2} (or Theorem \ref{thm:Bourgain-Proj}) that \eqref{form91} is not possible if $\delta$ is small enough, $\mu(B_{\theta}) > \delta^{\eta}$ for all $\theta \in E$, and $\eta = \eta(s,t) > 0$ is small enough. In other words, there exists $\theta \in E$ such that \eqref{form92} holds.  \end{proof}

We point out that the regularity assumption is not actually used in the base case, since it is not needed in Bourgain's projection theorem, but it is crucial in the proof of Proposition \ref{prop2}.

\begin{proof}[Proof of Theorem \ref{thm1} assuming Proposition \ref{prop2}] Fix $t \in (0,2)$ and $s \in (0,\min\{t,2 - t\})$. Let $\Sigma(s,t) \geq (t - s)/2$ be the infimum of the parameters $\sigma > (t - s)/2$ for which $\mathrm{Projection}(s,\sigma,t)$ holds (for all $\sigma' \geq \sigma$, recall Remark \ref{rem2}). In other words, for $\sigma > \Sigma(s,t)$ there exist constants $\Delta_{0}(s,t,\sigma) > 0$ and $\epsilon_{0} = \epsilon_{0}(s,t,\sigma) > 0$ such that the hypothesis of Proposition \ref{prop2} holds. We already know from the base case recorded above that $\Sigma \leq \sigma_{1} := t/2 - \eta(s,t)$ for some $\eta(s,t) > 0$, and we claim that $\Sigma = (t - s)/2$.

We make the counter assumption that $\Sigma > (t - s)/2$. Now, fix $\sigma > \Sigma$ so close to $\Sigma$ that $\sigma - \zeta < \Sigma$, where $\zeta := \zeta(s,\sigma,t) > 0$ is the constant provided by Proposition \ref{prop2}. Such a choice of "$\sigma$" is possible, since $\zeta(s,\sigma,t)$ stays bounded away from $0$ on the interval $[\Sigma,\sigma_{1}]$ (using the counter assumption $\Sigma > (t - s)/2$). But now Proposition \ref{prop2} tells us that $\mathrm{Projection}(s,\sigma - \zeta,t)$ holds, and this contradicts the definition of "$\Sigma$", since $\sigma - \zeta < \Sigma$.
\end{proof}

In the remainder of this section, we tackle the proof of  Proposition \ref{prop2}.

\subsection{Small slices imply sparse slices} This section contains an auxiliary result (Theorem \ref{thm2}) which allows us to "upgrade" the hypothesis of Proposition \ref{prop2} into a stronger one. In the sense of Terminology \ref{term1}, $\mathrm{Projection}(s,\sigma,t)$ tells us something about the "slices" of $(\delta,t)$-regular measures in directions perpendicular to the directions $\theta \in E$, where $E$ is a $(\delta,s)$-set. Namely, for typical $\theta \in E$, only a $\delta^{\epsilon}$-proportion of $\mu|_{B(1)}$ can be lie on "slices" with multiplicity $\geq \delta^{-\sigma}$. More informally still, the typical slices of $\mu$ have multiplicity $\leq \delta^{-\sigma}$. How is $\mu$ distributed along these typical slices? To answer this question, we define the following "local" variant of the high multiplicity sets:

\begin{definition}[Local high multiplicity sets] Let $\delta \in (0,\tfrac{1}{2}]$, $\rho \in 2^{-\N}$, $\sigma \in (0,1]$, and let $\theta \in S^{1}$. For $K \subset \R^{2}$, we define the \emph{local high multiplicity set}
\begin{equation*}
  H_{\theta,\mathrm{loc}}(K,\sigma,\delta,\rho) := \bigcup_{\delta \leq r \leq R \leq 8} H_{\theta}(K,4(R/r)^{\sigma},[r,R]), 
\end{equation*}
where the union ranges over \textbf{dyadic} radii $r,R \in 2^{-\N} \cap [\delta,8]$ satisfying $r/R \leq \rho$.
\end{definition}

Unwrapping the notation, we have $x \in H_{\theta,\mathrm{loc}}(K,\sigma,\delta,\rho)$ if and only if there exist radii $r= r_{x},R = R_{x} \in [\delta,8]$ satisfying $r \leq \rho R$ and with the property
\begin{displaymath} |B(x,R) \cap K^{(r)} \cap \pi_{\theta}^{-1}\{\pi_{\theta}(x)\}|_{r} = \mathfrak{m}_{K,\theta}(x, \mid [r,R]) \geq 4\left(\tfrac{R}{r} \right)^{\sigma}. \end{displaymath}
(The constants "$4,8$" are a little arbitrary, but it turns out useful to keep here constants somewhat bigger than "$1$".) The definition of $H_{\theta,\mathrm{loc}}(K,\sigma,\delta,\rho)$ is mainly useful if $\rho\ge\delta^{\e}$ for some small fixed $\e$, so the ratio $R/r$ is fairly large for any admissible $r,R$ (without this requirement, the local high multiplicity set can easily be "everything"). Perhaps another helpful observation is that if $\rho \leq \delta/8$, then
\begin{displaymath} H_{\theta,\mathrm{loc}}(K,\sigma,\delta,\rho) = H_{\theta}(K,C_{\sigma}\delta^{-\sigma},[\delta,8]). \end{displaymath}
Indeed, the only possible radii $\delta \leq r \leq R \leq 8$ satisfying $r \leq \rho R$ are $(r,R) = (\delta,8)$. On the other hand, if $\rho \gg \delta$, then the set $H_{\theta,\mathrm{loc}}(K,\sigma,\delta,\rho)$ is \emph{a priori} much larger than $H_{\theta}(K,\delta^{-\sigma},[\delta,8])$. This motivates the next theorem.

\begin{thm}\label{thm2} Let $s,\sigma \in [0,1]$ and $t \in [0,2]$, and assume that $\mathrm{Projection}(s,\sigma,t)$ holds with constants $\Delta_{0},\epsilon_{0} > 0$. In other words, whenever $\Delta \in (0,\Delta_{0}]$, $\mu$ is a $(\Delta,t,\Delta^{-\epsilon_{0}})$-regular measure, and $E \subset S^{1}$ is a $(\Delta,s,\Delta^{-\epsilon_{0}})$-set, then there exists $\theta \in E$ such that
\begin{displaymath}
\mu\big(B(1) \cap H_{\theta}(\spt(\mu),\Delta^{-\sigma},[\Delta,1])\big) \leq \Delta^{\epsilon_{0}}.
\end{displaymath}
Then, for every $\eta \in (0,1]$, there exist $\epsilon = \epsilon(\eta,\epsilon_{0}) > 0$ and $\delta_{0} = \delta_{0}(\Delta_{0},\epsilon,\eta) > 0$, such that the following holds for all $\delta \in (0,\delta_{0}]$. Let $\mu$ be a $(\delta,t,\delta^{-\epsilon})$-regular measure, and let $E \subset S^{1}$ be a $(\delta,s,\delta^{-\epsilon})$-set. Then, there exists $\theta \in E$ such that
\begin{equation}\label{form29}
\mu\big(B(1) \cap H_{\theta,\mathrm{loc}}(\spt(\mu),\sigma,\delta,\delta^{\eta})\big) \leq \delta^{\epsilon}.
\end{equation}
\end{thm}

\begin{remark} This theorem crucially uses the regularity assumption on $\mu$. \end{remark}

\begin{remark}\label{rem1} The proof shows that it is enough to take $\epsilon < \eta \epsilon_{0}/C$ for a certain absolute constant $C \geq 1$. \end{remark}

\begin{proof}[Proof of Theorem \ref{thm2}] Assume to the contrary that \eqref{form29} fails for every $\theta \in E$. We may and will assume that $E$ is $\delta$-separated. Abbreviate $K := \spt(\mu)$, and
\begin{displaymath} K_{\theta} := B(1) \cap H_{\theta,\mathrm{loc}}(K,\sigma,\delta,\delta^{\eta}). \end{displaymath}
Thus, $\mu(K_{\theta}) \geq \delta^{\epsilon}$ for all $\theta \in E$. By definition, for every $x \in K_{\theta}$, there exist dyadic radii $\delta \leq r \leq \delta^{\eta}R \leq \delta^{\eta}8$, depending on both $x$ and $\theta$, such that
$x \in H_{\theta}\bigl(K,4\left(\tfrac{R}{r} \right)^{\sigma},[r,R]\bigr)$ or, equivalently, 
  \begin{displaymath}  \mathfrak{m}_{K,\theta}(x \mid [r,R]) \geq 4\left(\tfrac{R}{r} \right)^{\sigma}. \end{displaymath}
Note that both $r,R$ take $\lesssim \log(1/\delta)$ possible values. By dyadic pigeonholing, and at the cost of replacing the lower bound $\mu(K_{\theta}) \geq \delta^{\epsilon}$ by $\mu(K_{\theta}) \geq \delta^{2\epsilon}$, we may assume that $r_{x,\theta} = r_{\theta}$ and $R_{x,\theta} = R_{\theta}$ for all $x \in K_{\theta}$. Note that we are pruning $K_{\theta}$ while leaving $K$ unchanged.

Similarly, by replacing $E$ by a subset which remains a $(\delta,s,\delta^{-2\epsilon})$-set, we may assume that $r_{\theta} = r \in [\delta,8]$ and $R_{\theta} = R \in [r,8]$ for all $\theta \in E$ (we keep the notation "$E$" for this subset). Applying Corollary \ref{cor:uni} (replacing $E$ by a further $(\delta,s,\delta^{-2\epsilon})$-subset), we may assume that $E$ is $\{2^{-jT}\}_{j = 1}^{m}$-uniform for some $T \sim_{\epsilon} 1$, and that $\delta = 2^{-mT}$.

Next, let $\mathcal{B}_{R}$ be a minimal cover of $B(1) \cap K$ by discs of radius $R$. By the $(\delta,t,\delta^{-\epsilon})$-regularity of $\mu$, we have $|\mathcal{B}_{R}| \leq \delta^{-\epsilon}R^{-t}$. Notice that
\begin{displaymath} \sum_{B \in \mathcal{B}_{R}} \sum_{\theta \in E} \mu(K_{\theta} \cap B) = \sum_{\theta \in E} \sum_{B \in \mathcal{B}_{R}} \mu(K_{\theta} \cap B) \geq \sum_{\theta \in E} \mu(K_{\theta}) \geq \delta^{\epsilon}|E|. \end{displaymath}
Consequently, there exists a disc $B \in \mathcal{B}_{R}$ with the property
\begin{displaymath} \sum_{\theta \in E} \mu(K_{\theta} \cap B) \geq \frac{\delta^{\epsilon}|E|}{|\mathcal{B}_{R}|} \geq \delta^{2\epsilon}|E|R^{t}. \end{displaymath}
As a further consequence, and using the Frostman bound $\mu(K_{\theta} \cap B) \leq \delta^{-\epsilon}R^{t}$, there exists a subset $E' \subset E$ of cardinality $|E'| \geq \delta^{4\epsilon}|E|$ such that
\begin{equation}\label{form32} \mu(K_{\theta} \cap B) \geq 16 \cdot \delta^{4\epsilon}R^{t}\quad\text{for all } \theta \in E'. \end{equation}
(The "$16$" is a useful factor, understandable in a moment.) We note that
\begin{equation}\label{form31} K_{\theta} \cap B \subset H_{\theta}(K \cap 2B,4\left(\tfrac{R}{r}\right)^{\sigma},[r,R]), \end{equation}
where $2B$ is the disc concentric with $B$ and radius $2R$. This nearly follows from the definition of $K_{\theta}$ (and our pigeonholing of $r,R$), but we added the intersection with the disc $2B$. This is legitimate, since (recalling that $B$ is a disc of radius $R$)
\begin{displaymath} 4\left(\tfrac{R}{r} \right)^{\sigma} \leq \mathfrak{m}_{K,\theta}(x \mid [r,R]) \stackrel{\mathrm{def.}}{=} |B(x,R) \cap K^{(r)} \cap \pi_{\theta}^{-1}\{\pi_{\theta}(x)\}|_{r}, \quad x \in B \cap K_{\theta}, \end{displaymath}
and the right-hand side does not change if we replace "$K$" by "$K \cap 2B$".
Now, as in Remark \ref{rem:regular}, the measure $\mu_{4B} = (4R)^{-t} T_{4B}(\mu)$ is $(\delta/(4R),t,\delta^{-\epsilon})$-regular. Since $r/R \leq \delta^{\eta}$, we see that $\mu_{4B}$ is also $(\Delta,t,\Delta^{-\epsilon/\eta})$-regular with $\Delta := r/(4R) \leq \delta^{\eta}$. Furthermore, by Lemma \ref{lemma7} applied with $r_{0} = 4R$, we have
\begin{align*} T_{4B}(H_{\theta}(K \cap 2B,M,[r,R])) & = H_{\theta}(B(\tfrac{1}{2}) \cap T_{4B}(K),M,\left[\tfrac{r}{4R},\tfrac{1}{4} \right]) \\
& \subset B(1) \cap H_{\theta}(T_{4B}(K),M,\left[\tfrac{r}{4R},1\right]), \end{align*}
where we abbreviated $M := 4(R/r)^{\sigma} \geq (4R/r)^{\sigma}$ (recall that $\sigma \leq 1$). As a consequence,
\begin{align} \mu_{4B}\big(B(1) \cap & H_{\theta}(T_{4B}(K), \left(\tfrac{4R}{r} \right)^{\sigma},\left[\tfrac{r}{4R},1\right])\big) \notag\\
& \geq (4R)^{-t}\mu(H_{\theta}(K \cap 2B,M,[r,R])) \notag\\
&\label{form33} \stackrel{\eqref{form31}}{\geq} \frac{\mu(K_{\theta} \cap B)}{16R^{t}} \stackrel{\eqref{form32}}{\geq} \delta^{4\epsilon} \geq \left(\tfrac{r}{4R} \right)^{4\epsilon/\eta}\quad\text{for all } \theta \in E'. \end{align}
We now claim that \eqref{form33} violates our assumption that $\mathrm{Projection}(s,\sigma,t)$ holds at scale $\Delta = r/(4R)$. Namely, since $E$ is a $\{2^{-jT}\}_{j = 1}^{m}$-uniform $(\delta,s,\delta^{-\epsilon})$-set (with $T \sim_{\epsilon} 1$), and $|E'| \geq \delta^{4\epsilon}|E|$, we infer from Corollary \ref{cor1} that $E'$ is a $(\Delta,s,\delta^{-C\epsilon})$-set, assuming that $\delta > 0$ is small enough in terms of $\epsilon$. Further, recalling that $\Delta \leq \delta^{\eta}$, we see that $E'$ is in fact a $(\Delta,s,\Delta^{-C\epsilon/\eta})$-set.

Finally, recall that $\mu_{4B}$ is a $(\Delta,t,\Delta^{-\epsilon/\eta})$-regular measure. Now, if $\epsilon = \epsilon(\eta,\epsilon_{0}) > 0$ is small enough, the inequality \eqref{form33} (for all $\theta \in E'$) violates our assumption that $\mathrm{Projection}(s,\sigma,t)$ holds at scale $\Delta$. To be precise, recalling the parameters "$\Delta_{0},\epsilon_{0}$" in the statement of the theorem, the contradiction will ensue if we have taken $\delta > 0$ so small that $r/(4R) \leq \delta^{\eta} \leq \Delta_{0}$, and $\epsilon > 0$ so small that $C\epsilon/\eta \leq \epsilon_{0}$. \end{proof}

\subsection{Fixing parameters} We may now begin the proof of Proposition \ref{prop2} in earnest. Fix the triple $(s,\sigma,t)$ as in Proposition \ref{prop2}:
\begin{displaymath} t \in (0,2), \quad s \in (0,\min\{t,2 - t\}) \quad \text{and} \quad \tfrac{t - s}{2} < \sigma_{0} \leq \sigma \leq \sigma_{1} < \tfrac{t}{2}. \end{displaymath}
The role of "$\sigma_{0},\sigma_{1}$" is to quantify that $\sigma$ stays bounded away from both $(t - s)/2$ and $t/2$. The plan of proof is to make a counter assumption to Proposition \ref{prop2} with these parameters, and eventually find a contradiction to the $\delta$-discretised $ABC$ sum-product theorem, Theorem \ref{thm:ABCConjecture}. We now discuss the relevant parameters "$\alpha,\beta,\gamma$" to which Theorem \ref{thm:ABCConjecture} will be applied. A good starting point is
\begin{displaymath} \alpha' = t - \sigma, \quad \beta' = \sigma, \quad \text{and} \quad \gamma' := s. \end{displaymath}
We make a few remarks about these parameters:
\begin{itemize}
\item $0 < \beta'$ since $\sigma > (t - s)/2 > 0$.
\item $\beta' < \alpha'$, since $\sigma < t/2$.
\item $\alpha' < 1$, since $t - \sigma < t - (t - s)/2 = (s + t)/2 < 1$ by the assumption $s < 2 - t$.
\item $\gamma' > \alpha' - \beta'$ (as required to apply Theorem \ref{thm:ABCConjecture}) since $\sigma > (t - s)/2$.
\end{itemize}

Evidently, there exists a constant $\zeta_{0} = \zeta_{0}(s,\sigma_{0},\sigma_{1},t) > 0$ such that all the inequalities above remain valid if $s$ is replaced by $s - \zeta_{0}$, and $\alpha',\beta',\gamma'$ are replaced by
\begin{equation}\label{eq:parameters} 
  \alpha := \alpha' + \zeta_{0}, \quad \beta := \beta' - \zeta_{0}, \quad \text{and} \quad \gamma := \gamma' - \zeta_{0} = s - \zeta_{0}, 
\end{equation}
and for all $\sigma_{0} \leq \sigma \leq \sigma_{1}$. We now apply Theorem \ref{thm:ABCConjecture} with the parameters $\alpha,\beta,\gamma$. The conclusion is that there exist constants $\chi,\delta_{0} > 0$ such that \eqref{conclusion2} holds for all $\delta \in (0,\delta_{0}]$: if $A,B,\nu$ satisfy the hypotheses of Theorem \ref{thm:ABCConjecture} (thus $\nu(B(x,r)) \leq \delta^{-\chi}r^{s - \zeta_{0}}$ for $r \in [\delta,1]$, as in Remark \ref{rem5}), then there exists $c \in \spt(\nu)$ with
\begin{equation}\label{form82} 
  |\pi_{c}(G)|_{\delta} \geq \delta^{-\chi}|A|\quad\text{for all } G \subset A \times B, \, |G| \geq \delta^{\chi}|A||B|. 
\end{equation}
Let us emphasise that
\begin{equation*}
  \chi = \chi(\alpha,\beta,\gamma) = \chi(t - \sigma + \zeta_{0},\sigma - \zeta_{0},s - \zeta_{0}) > 0, 
\end{equation*}
and as long as $\sigma_{0} \leq \sigma \leq \sigma_{1}$, the triple $(t - \sigma + \zeta_{0},\sigma - \zeta_{0},s)$ ranges in a compact subset of the domain $\Omega_{\mathrm{ABC}}$ in Theorem \ref{thm:ABCConjecture}. Therefore,
\begin{equation}\label{form85} \chi \geq \chi_{0}  := \chi_{0}(s,\sigma_{0},\sigma_{1},t) > 0. \end{equation}
Now, we fix small parameters $\delta_{0},\epsilon,\zeta \in (0,\tfrac{1}{2}]$ (whose values will be discussed in a moment), and then make our main counter assumption: there exists a $(\delta,t,\delta^{-\epsilon})$-regular measure $\mu$ with support $K$, and a $(\delta,s,\delta^{-\epsilon})$-set $E \subset S^{1}$ with the property
\begin{equation}\label{form9}
\mu\big(B(1) \cap H_{\theta}(K,\delta^{-\sigma + \zeta},[\delta,1])\big) \geq \delta^{\epsilon} \quad\text{for all } \theta \in E.
\end{equation}
By passing to a maximal $\delta$-separated subset of $E$ if necessary, we may and will assume that $E$ is a $\delta$-separated $(\delta,s,\delta^{-\epsilon})$-set.

Whichever is more convenient, we will either view $E$ as a subset of $S^{1}$, or as a subset of $[0,1)$: the latter interpretation will be applied when we ask whether $E$, or a subset thereof, is $\{\Delta_{j}\}_{j = 1}^{m}$-uniform for a suitable sequence of scales $\delta = \Delta_{m} < \ldots < \Delta_{1} \leq \Delta_{0} = 1$.

Let us discuss the parameters $\zeta,\epsilon,\delta$ further. The parameter "$\zeta$" is the most important one. We will see that \eqref{form9} implies a contradiction against \eqref{eq:parameters} if $\zeta > 0$ is chosen sufficiently small in terms of both the constant $\zeta_{0} > 0$ in \eqref{eq:parameters} and the constant $\chi_{0} > 0$ introduced at \eqref{form85}. Thus, the contradiction will ensue if
\begin{equation}\label{form84} \zeta = o_{\zeta_{0},\chi_{0}}(1) = o_{s,\sigma_{0},\sigma_{1},t}(1). \end{equation}
Here $o_{p_{1},\ldots,p_{n}}(1)$ refers -- and will refer -- to a function of the parameters $p_{1},\ldots,p_{n}$ which is continuous at $0$ and vanishes at $0$. This will show that Proposition \ref{prop2} actually holds with some (maximal) constant "$\zeta$" satisfying \eqref{form84}. The constant $\epsilon > 0$ in Proposition \ref{prop2} is allowed to depend on both $\zeta$, and also the constants $\epsilon_{0}$ for which the "inductive hypothesis" in Proposition \ref{prop2} holds. The constant $\delta_{0} > 0$ is additionally allowed to depend on $\epsilon$ and $\Delta_{0}$. Thus, to reach a contradiction, we will need that
\begin{equation}\label{parameters} \epsilon = o_{\zeta,\epsilon_{0}}(1) \quad \text{and} \quad \delta_{0} = o_{\Delta_{0},\epsilon,\zeta}(1). \end{equation}
When in the sequel we write "...by choosing $\delta > 0$ sufficiently small" we in fact mean "...by choosing the upper bound $\delta_{0}$ for $\delta$ sufficiently small". Also, we may write \emph{Note that $\epsilon \leq \zeta^{10}\epsilon_{0}$ by \eqref{parameters}}, or something similar.  This simply means that the requirement "$\epsilon \leq \zeta^{10}\epsilon_{0}$" should -- at that moment -- be added to the list of restrictions for the function $o_{\zeta,\epsilon_{0}}(1)$. Finally, we will often use the following abbreviation: an upper bound of the form $C_{\epsilon,\zeta}\delta^{-C\epsilon}$ will be abbreviated to $\delta^{-O(\epsilon)}$. Indeed, if $\delta = o_{\zeta,\epsilon}(1)$ is sufficiently small, then $C_{\epsilon,\zeta} \leq \delta^{-\epsilon}$, and hence $C_{\epsilon,\zeta}\delta^{-C\epsilon} \leq \delta^{-(C + 1)\epsilon}$.

\subsection{Finding a branching scale for \texorpdfstring{$E$}{E}} We need to discuss the "branching numbers" of the set $E$ from the counter assumption \eqref{form9}, so we need to know that $E$ is uniform. Given the constant $\epsilon > 0$ as in \eqref{parameters}, we fix $T \sim_{\epsilon} 1$ as in Corollary \ref{cor:uni}. Then, we may find a $\{2^{-jT}\}_{j = 1}^{m}$-uniform subset $E' \subset E$ with $|E'| \geq \delta^{\epsilon}|E|$, which is automatically a $(\delta,s,\delta^{-2\epsilon})$-set. We replace $E$ by this subset without changing notation: thus, we assume that $E$ is $\{2^{-jT}\}_{j = 1}^{m}$-uniform with $2^{-(m + 1)T} < \delta \leq 2^{-mT}$, and $T \sim_{\epsilon} 1$. It is easy to reduce matters to the case $\delta = 2^{-mT}$, and we will make this assumption in the sequel.

What now follows would be unnecessary if the set $E$ happened to be $s$-regular. More precisely, to skip this section, we would need to know that if $I \in \mathcal{D}_{\delta^{1/2}}(E)$, then the renormalisation $E_{I}$ is a $(\delta^{1/2},s)$-set. This may not be the case. The problem will be solved by "replacing" $\delta$ by a larger scale $\bar{\delta} \in [\delta,\delta^{\sqrt{\zeta}/12}]$. The scale $\bar{\delta}$ will be chosen in such a way that if $I \in \mathcal{D}_{\bar{\delta}^{1/2}}(E)$, then the renormalisation $E_{I}$ is a $(\bar{\delta}^{1/2},s)$-set. Lemma \ref{l:regular-is-frostman-ahlfors} plays a central role. The main technicality caused by the "replacement" action is that our main counter assumption \eqref{form9} concerns the scale "$\delta$", not "$\bar{\delta}$". However, by virtue of the regularity of $\mu$, it turns out that a sufficiently strong version of \eqref{form9} will remain valid at scale $\bar{\delta}$. After these observations have been consolidated, we may assume "without loss of generality" that $\delta = \bar{\delta}$, and thus $E_{I}$, $I \in \mathcal{D}_{\delta^{1/2}}(E)$, is a $(\delta^{1/2},s)$-set, as initially desired.

Recall that after our reductions, $E\subset [0,1)$ is now a $(\delta,s,\delta^{-2\epsilon})$-set which is also $(2^{-j T})_{j=1}^m$-uniform (and $\delta=2^{-jm}$), while $\mu$ continues to be the $(\delta,t,\delta^{-\epsilon})$-regular measure from Theorem \ref{thm1}. Here is precisely what we claim:
\begin{proposition}\label{prop4}
  
  There exist a scale $\bar{\delta} =2^{-2\bar{m}T}\in [\delta,\delta^{\sqrt{\zeta}/12}]$ and a $\bar{\delta}$-separated $(\bar{\delta},s,\bar{\delta}^{-O_{\zeta}(\epsilon)})$-subset $E_{\bar{\delta}} \subset E$ which is $\{2^{-jT}\}_{j = 1}^{\bar{m}}$-uniform and has the following properties.
\begin{itemize}
\item[(a)] For $\bar{\Delta} := \bar{\delta}^{1/2}$ and $I \in \mathcal{D}_{\bar{\Delta}}(E_{\bar{\delta}})$, the set $(E_{\bar{\delta}})_{I}$ is a $(\bar{\Delta},s - \sqrt{\zeta},\bar{\Delta}^{-O_{\zeta}(\epsilon)})$-set.
\item[(b)] We have
\begin{equation}\label{form36}
\mu\big(B(1) \cap H_{\theta}(K,\bar{\delta}^{-\sigma + \bar{\zeta}},[\bar{\delta},5])\big) \geq \bar{\delta}^{O_{\zeta}(\epsilon)}\quad\text{for all } \theta \in E_{\bar{\delta}}.
\end{equation}
\end{itemize}
Here in fact $O_{\zeta}(\epsilon) = 12\epsilon/\sqrt{\zeta}$, and $\bar{\zeta} = 7\sqrt{\zeta}$.
\end{proposition}

\begin{remark} Note that \eqref{form36} is an analogue of our initial counter assumption \eqref{form9}, except that the scale $\delta$ has been replaced by $\bar{\delta}$, and we have gained the property (a). Since $\zeta > 0$ is a constant depending only on $s,\sigma_{0},\sigma_{1},t$, the constant $O_{\zeta}(\epsilon)$ can still be made arbitrarily small compared to $\epsilon_{0}$ by choosing $\epsilon$ small enough, in terms of $\epsilon_{0},s,\sigma,\zeta$.  \end{remark}

We then prove Proposition \ref{prop4}. Since $E$ is $\{\Delta_{j}\}_{j = 1}^{m}$-uniform, and a subset of $[0,1)$, we may associate to it the $1$-Lipschitz branching function $\beta \colon [0,m] \to [0,m]$ as in Definition \ref{branchingFunction}. Recall that $\beta$ is the linear interpolation between the conditions $\beta(0) = 0$, and
\begin{displaymath} \beta(j) := \frac{\log |E|_{2^{-jT}}}{T}, \qquad j \in \{1,\ldots,m\}. \end{displaymath}
Since $E$ is a $(\delta,s,\delta^{-\epsilon})$-set, it follows from Lemma \ref{l:regular-is-frostman-ahlfors} that
\begin{displaymath} \beta(x) \geq sx - \epsilon m - C\quad\text{for all } x \in [0,m]. \end{displaymath}
Therefore, the renormalised function $f(x) := \tfrac{1}{m}\beta(mx)$, defined on $[0,1]$, is also $1$-Lipschitz and satisfies
\begin{displaymath} f(x) \geq \tfrac{1}{m}(smx - \epsilon m - C) \geq sx - \epsilon - C/m\quad\text{for all } x \in [0,1]. \end{displaymath}
Since $\delta = 2^{-m T}$ with $T \sim_{\epsilon} 1$, we may assume that $C/m \leq \epsilon$ by choosing $\delta = o_{\epsilon}(1)$.  Therefore, $f(x) \geq sx - 2\epsilon$ for $x \in [0,1]$. Since $2\epsilon \in (0,\sqrt{\zeta}/6]$ by \eqref{parameters}, Corollary \ref{cor:lemma2} allows us to find a point $a \in [\sqrt{\zeta}/12,\tfrac{1}{3}]$ with the property
\begin{displaymath} f(x) - f(a) \geq (s - \sqrt{\zeta})(x - a)\quad\text{for all } x \in [a,1]. \end{displaymath}
In terms of the original branching function $\beta$, this means that there exists a point $\bar{m} := am \in [\sqrt{\zeta} m/12,m/3]$ with the property
\begin{equation}\label{form15} 
    \beta(x) - \beta(\bar{m}) \geq (s - \sqrt{\zeta})(x - \bar{m})\quad\text{for all } x \in [\bar{m},m]. 
\end{equation}
If $\bar{m}$ is not an integer, then the inequality still holds for $\floor{\bar{m}}$ in place of $\bar{m}$, since the slope of the function $\beta$ on $[\floor{\bar{m}},\ceil{\bar{m}}]$ is at least $s-\sqrt{\zeta}$, as seen by taking $x=\ceil{\bar{m}}$ in \eqref{form15}. We therefore assume from now on that $\bar{m}$ is an integer.  We only need \eqref{form15} for $x \in [\bar{m},2\bar{m}]$. 

Set
\begin{displaymath} \bar{\delta} := 2^{-2\bar{m}T} \quad \text{and} \quad \bar{\Delta} := \bar{\delta}^{1/2} = 2^{-\bar{m}T}. \end{displaymath}
It now follows from \eqref{form15} and Lemma \ref{lemma3} that if $I \in \mathcal{D}_{\bar{\Delta}}(E)$, then $E_{I}$ is a $(\bar{\Delta},s - \sqrt{\zeta},C_{\epsilon})$-set, where $C_{\epsilon} \lesssim 2^{T} \lesssim_{\epsilon} 1$. Since $\bar{m} \geq \sqrt{\zeta} m/12$, we notice that
\begin{equation}\label{form17} \delta = 2^{-mT} \geq (2^{-2\bar{m}T})^{6/\sqrt{\zeta}} = \bar{\delta}^{6/\sqrt{\zeta}}, \end{equation}
as desired in Proposition \ref{prop4}. On the other hand, since $\bar{m} \leq m/3$, we have $2\bar{m} \leq 2m/3$, and therefore $\bar{\delta}$ is also substantially larger than $\delta$:
\begin{equation}\label{form19} \delta/\bar{\delta} = 2^{(2\bar{m} - m)T} \leq 2^{-mT/3} = \delta^{1/3}. \end{equation}
The ratio $\delta/\bar{\delta}$ will appear in our calculations in a moment, and \eqref{form19} will allow us to assume that it is "as small as needed" by choosing $\delta > 0$ small enough.

The scale $\bar{\delta} \in [\delta,\delta^{\sqrt{\zeta}/12}]$ has now been fixed, and simply the choice $E_{\bar{\delta}} := E$ (or at least a $\bar{\delta}$-net inside $E$) would satisfy Proposition \ref{prop4}(a). To reach (b), we need to work a little more, and eventually replace $E$ by the final subset $E_{\bar{\delta}}$. The following auxiliary result is \cite[Proposition 5.1]{MR4388762}.

\begin{proposition}\label{prop3} Let $\theta \in S^{1}$, $1 \leq M \leq N < \infty$, let $\delta < r \leq R \leq 1$, and let $\mu$ be a $(\delta,t,C_{\mathrm{reg}})$-regular measure with $t \in [0,2]$, $C_{\mathrm{reg}} > 0$, and $K := \spt \mu \subset \R^{2}$. Abbreviate $\mu_{\rho} := \mu|_{B(\rho)}$ for $\rho > 0$. Then, there exist absolute constants $c,C > 0$ such that
\begin{equation}\label{form14}
\mu_{1}\big(H_{\theta}(CN,[r,1])\big) \leq \mu_{1}\big(H_{\theta}(cM,[4R,5])\big) + CC_{\mathrm{reg}}^{2} \cdot \mu_{4}\big(H_{\theta}(c\tfrac{N}{M},[4r,7R])\big).  \end{equation}
\end{proposition}
Here we have omitted the set "$K$" from the $H_{\theta}$-notation, but it should appear in all three instances of $H_{\theta}$. We will apply the proposition to the $(\delta,t,\delta^{-\epsilon})$-regular measure $\mu$ with the parameters $M \leq N$ satisfying
\begin{displaymath} CN = \delta^{-\sigma + \zeta} \quad \text{and} \quad c \cdot \tfrac{N}{M} = \left(\delta/\bar{\delta} \right)^{-\sigma}, \end{displaymath}
from which we may solve that
\begin{equation*} 
  c M = c^{2}\left(\tfrac{\delta}{\bar{\delta}} \right)^{\sigma} \cdot N = \tfrac{c^{2}}{C} \left(\tfrac{\delta}{\bar{\delta}} \right)^{\sigma} \cdot \delta^{-\sigma + \zeta} = \tfrac{c^{2}}{C} \cdot \bar{\delta}^{-\sigma}\cdot \delta^{\zeta} \stackrel{\eqref{form17}}{\geq} \bar{\delta}^{-\sigma + 7\sqrt{\zeta}}. 
\end{equation*}
In the final inequality, we also took $\delta$ so small that $(c^{2}/C) \geq \delta^{\sqrt{\zeta}}$. In the sequel we abbreviate $\bar{\zeta} := 7\sqrt{\zeta}$.

We have now defined the parameters $M,N$ relevant for the application of Proposition \ref{prop3}, but we still need to specify the radii $0 < r \leq R \leq 1$. We set $r := \delta$, and $R := \bar{\delta}/4$. (We have $R > r$ by \eqref{form19}.) Let us then rewrite the conclusion \eqref{form14} with these parameters. Before doing this, notice that the left-hand side is lower bounded by $\delta^{\epsilon}$ by our counter assumption \eqref{form9}, for $\theta \in E$. Therefore, for $\theta \in E$ fixed, we obtain
\begin{align}
\delta^{\epsilon} & \leq \mu\big(B(1) \cap H_{\theta}(K,\delta^{-\sigma + \zeta},[\delta,1])\big) \notag\\
& \leq \mu\big(B(1) \cap H_{\theta}(K,\bar{\delta}^{-\sigma + \bar{\zeta}},[\bar{\delta},5])\big) \notag\\
&\label{form20}\qquad + C\delta^{-2\epsilon}\mu\big(B(4) \cap H_{\theta}(K,\left(\delta/\bar{\delta}\right)^{-\sigma},[4\delta,4\bar{\delta}])\big).
\end{align}
The plan is, next, to use the validity of $\mathrm{Projection}(s,\sigma,t)$ to show that for typical $\theta \in E$, the first term must dominate, or in other words the term \eqref{form20} is substantially smaller than $\delta^{\epsilon}$. More precisely, we claim that if $\epsilon > 0$ is sufficiently small relative to $\epsilon_{0}$, then
\begin{equation}\label{form23} \sum_{\theta \in E} C\delta^{-2\epsilon}\mu\big(B(4) \cap H_{\theta}(K,\left(\delta/\bar{\delta}\right)^{-\sigma},[4\delta,4\bar{\delta}])\big) \leq \delta^{2\epsilon}|E|. \end{equation}
To prove \eqref{form23}, let $\mathcal{K}$ be a minimal cover of $K \cap B(4)$ by discs of radius $4\bar{\delta}$. By the $(t,\delta^{-\epsilon})$-regularity of $\mu$, we have
\begin{equation}\label{form26} |\mathcal{K}| \leq \delta^{-\epsilon} \cdot \bar{\delta}^{-t}. \end{equation}
Then, we decompose
\begin{equation*}
  \eqref{form20} \leq C\delta^{-2\epsilon} \cdot \sum_{B \in \mathcal{K}} \mu\left(B \cap H_{\theta} \left(K,\left(\tfrac{\delta}{\bar{\delta}} \right)^{-\sigma},[4\delta,4\bar{\delta}] \right) \right), \quad \theta \in E. 
\end{equation*}
To treat the individual terms on the right-hand side, we consider the rescaled and renormalised measures $\mu_{B} = (4\bar{\delta})^{-t} \cdot T_{B}\mu$ familiar from Remark \ref{rem:regular}, and we write
\begin{equation}\label{form22} \mu\left(B \cap H_{\theta} \left(K,\left(\tfrac{\delta}{\bar{\delta}} \right)^{-\sigma},[4\delta,4\bar{\delta}] \right) \right) = (4\bar{\delta})^{t} \mu_{B}\left(B(1) \cap H_{\theta}\left(T_{B}(K),\Delta^{-\sigma},[\Delta,1] \right) \right), \end{equation}
for any $\theta \in [0,1]$, where $\Delta := \delta/\bar{\delta}$. This equation is easily deduced from Lemma \ref{lemma7} with $r_{0} = 4\bar{\delta}$. In \eqref{form22}, the measure $\mu_{B}$ is $(\Delta,t,\delta^{-\epsilon})$-regular, where
\begin{displaymath} \delta^{-\epsilon} \leq (\delta/\bar{\delta})^{-3\epsilon} = \Delta^{-3\epsilon} \end{displaymath}
by \eqref{form19}. In particular, $\mu_{B}$ is $(\Delta,t,\Delta^{-\epsilon_{0}})$-regular, assuming $\epsilon \leq \epsilon_{0}/3$. Therefore, since we may assume that $\Delta = \delta/\bar{\delta} \leq \delta^{1/3} \leq \Delta_{0}$, the hypothesis of Proposition \ref{prop2} is applicable to the measure $\mu_{B}$. We claim, as a corollary of this hypothesis applied to $\mu_{B}$, that
\begin{equation}\label{form24}
\sum_{\theta \in E} \mu_{B}\big(B(1) \cap H_{\theta}(T_{B}(K),\Delta^{-\sigma},[\Delta,1])\big) \leq \delta^{10\epsilon}|E|.
\end{equation}
Assume for a moment that \eqref{form24} fails. Then, since $\mu_{B}(B(1)) \leq \delta^{-\epsilon}$, there exists a subset $E' \subset E$ with the properties $|E'| \geq \delta^{20\epsilon}|E|$ and
\begin{equation}\label{form25}
\mu_{B}\big(B(1) \cap H_{\theta}(T_{B}(K),\Delta^{-\sigma},[\Delta,1])\big) \geq \delta^{20\epsilon}\quad\text{for all } \theta \in E'.
\end{equation}
Since $E$ was a $\{2^{-jT}\}_{j = 1}^{m}$-uniform $(\delta,s,\delta^{-\epsilon})$-set, it now follows from Corollary \ref{cor1} that $E'$ is a $(\Delta,s,\delta^{-O(\epsilon)})$-set, and since $\Delta = \delta/\bar{\delta} \leq \delta^{1/3}$ by \eqref{form19}, in fact $E'$ is a $(\Delta,s,\Delta^{-O(\epsilon)})$-set. Therefore, if we take $\epsilon \leq \epsilon_{0}/C$ for an absolute constant $C > 0$, our hypothesis implies that
\begin{displaymath}
\mu_{B}\big(B(1) \cap H_{\theta}(T_{B}(K),\Delta^{-\sigma},[\Delta,1])\big) \leq \Delta^{\epsilon_{0}} = (\delta/\bar{\delta})^{\epsilon_{0}} \leq \delta^{\epsilon_{0}/3}
\end{displaymath}
for some $\theta \in E'$. This violates \eqref{form25} if $20\epsilon < \epsilon_{0}/3$, and the ensuing contradiction shows that \eqref{form24} must be valid. Consequently,
\begin{align*} \sum_{\theta \in E} & C\delta^{-2\epsilon} \sum_{B \in \mathcal{K}} \mu\left(B \cap H_{\theta} \left(K,\left(\tfrac{\delta}{\bar{\delta}} \right)^{-\sigma},[4\delta,4\bar{\delta}] \right) \right)\\
& \stackrel{\eqref{form22}}{=} (4\bar{\delta})^{t}\sum_{B \in \mathcal{K}} C\delta^{-2\epsilon} \sum_{\theta \in E} \mu_{B}\left(B(1) \cap H_{\theta}\left(T_{B}(K),\Delta^{-\sigma},[\Delta,1] \right) \right)\\
& \stackrel{\eqref{form24}}{\leq} (4\bar{\delta})^{t} \sum_{B \in \mathcal{K}} C\delta^{-2\epsilon}\delta^{10\epsilon}|E| \stackrel{\eqref{form26}}{\lesssim} \delta^{7\epsilon}|E|.  \end{align*}
The left-hand side of this chain is an upper bound for the left-hand side of \eqref{form23}, so we have now established \eqref{form23}. Now, inspecting \eqref{form20}, and plugging in \eqref{form23}, we have
\begin{align*}
\delta^{\epsilon}|E| &\leq \sum_{\theta \in E} \mu\big(B(1) \cap H_{\theta}(K,\delta^{-\sigma + \zeta},[\delta,1])\big)\\
& \leq \sum_{\theta \in E} \mu\big(B(1) \cap H_{\theta}(K,\bar{\delta}^{-\sigma + \bar{\zeta}},[\bar{\delta},5])\big) + \delta^{2\epsilon}|E|,
\end{align*}
and consequently
\begin{displaymath}
\sum_{\theta \in E} \mu\big(B(1) \cap H_{\theta}(K,\bar{\delta}^{-\sigma + \bar{\zeta}},[\bar{\delta},5])\big) \geq \tfrac{1}{2}\delta^{\epsilon}|E|. \end{displaymath}
It follows that there exists a subset $E' \subset E$ with $|E'| \geq \delta^{2\epsilon}|E| \geq \bar{\delta}^{12\epsilon/\sqrt{\zeta}}|E|$ with the property
\begin{equation}\label{form28}
\mu\big(B(1) \cap H_{\theta}(K,\bar{\delta}^{-\sigma + \bar{\zeta}},[\bar{\delta},5])\big) \geq \delta^{2\epsilon} \geq \bar{\delta}^{12\epsilon/\sqrt{\zeta}}\quad\text{for all } \theta \in E'.
\end{equation}
This verifies \eqref{form36}, and therefore Proposition \ref{prop4}(b). A small technicality remains: the scale $\bar{\delta}$ was chosen so that $E_{I}$ is a $(\bar{\Delta},s - \sqrt{\zeta},C(\epsilon))$-set for all $I \in \mathcal{D}_{\bar{\Delta}}(E)$ (with $\bar{\Delta} = \bar{\delta}^{1/2}$), but since $E' \subset E$ is only a subset with $|E'| \geq \delta^{2\epsilon}$, this property may now be violated. To fix this, apply Corollary \ref{cor:uni} to find a further $\{2^{-jT}\}_{j = 1}^{m}$-uniform subset $E'' \subset E'$ with $|E''| \geq \delta^{\epsilon}|E'|$. Now, let $I \in \mathcal{D}_{\bar{\Delta}}(E'')$. Then $|E'' \cap I| \geq \delta^{3\epsilon}|E \cap I|$, since otherwise
\begin{displaymath} |E''| = \sum_{I \in \mathcal{D}_{\bar{\Delta}}(E'')} |E'' \cap I| < \delta^{3\epsilon} \sum_{I \in \mathcal{D}_{\bar{\Delta}}(E)} |E \cap I| = \delta^{3\epsilon}|E| \leq \delta^{\epsilon}|E'|. \end{displaymath}
Now it follows easily from a combination of the $(\bar{\Delta},s - \sqrt{\zeta},C(\epsilon))$-set property of $E_{I}$, and $|E'' \cap I| \geq \delta^{3\epsilon}|E \cap I|$, that $E_{I}''$ is a $(\bar{\Delta},s - \sqrt{\zeta},\delta^{-O(\epsilon)})$-set for all $I \in \mathcal{D}_{\Delta}(E'')$. Finally, since $\delta^{O(\epsilon)} \leq \bar{\Delta}^{-O_{\zeta}(\epsilon)}$, we see that $E_{I}''$ is a $(\bar{\Delta},s - \sqrt{\zeta},\bar{\Delta}^{-O_{\zeta}(\epsilon)})$-set for all $I \in \mathcal{D}_{\bar{\Delta}}(E'')$.

In Proposition \ref{prop4}, we desired the set $E_{\bar{\delta}}$ to be $\bar{\delta}$-separated. This is finally achieved by choosing one point from each interval $J \in \mathcal{D}_{\bar{\delta}}(E'')$, and calling the result $E_{\bar{\delta}}$. This does not violate the property of the blow-ups $E_{I}''$ established just above, since the $(\bar{\Delta},s - \sqrt{\zeta})$-set property of $E_{I}''$ only cares about the behaviour of $E''$ between the scales $\bar{\delta}$ and $\bar{\Delta}$. The proof of Proposition \ref{prop4} is complete.

\begin{notation}\label{not1} Now that Proposition \ref{prop4} has been established, we remove the "bars" from the notation. We assume that $\bar{\delta} = \delta$ (thus $\bar{\Delta} = \sqrt{\delta}$) and $E_{\bar{\delta}} = E$. We also rename $\bar{\zeta} =: \zeta$. The only difference to our starting position is, then, that some constants of the form $\delta^{\epsilon}$ have to be replaced by $\delta^{O_{\zeta}(\epsilon)}$. Notably, $E$ is a $(\delta,s,\delta^{-O_{\zeta}(\epsilon)})$-set, and $\mu$ is $(t,\delta^{-O_{\zeta}(\epsilon)})$-regular. We are seeking a contradiction if $\zeta > 0$ is small enough in terms of $(s,\sigma,t)$, and $\epsilon$ is small enough in terms of $(\epsilon_{0},s,\sigma,t,\zeta)$, so this difference will be completely irrelevant. Additionally, $E_{I}$ is a $(\Delta,s - \zeta,\Delta^{-O_{\zeta}(\epsilon)})$-set for all $I \in \mathcal{D}_{\Delta}(E)$.  \end{notation}

\subsection{Defining the sets $K_{\theta}$} We start again with a brief heuristic discussion. Recall that by assumption \eqref{form9}, or more precisely \eqref{form36}, we have
\begin{equation}\label{form48}
\mu\big(B(1) \cap H_{\theta}(K,\delta^{-\sigma + \zeta},[\delta,1])\big) \geq \delta^{C_{\zeta}\epsilon}\quad\text{for all } \theta \in E,
\end{equation}
where $C_{\zeta} \geq 1$ is a constant depending only on $\zeta$. On the other hand, by the hypothesis that $\mathrm{Projection}(s,\sigma,t)$ is valid, we have the (nearly) opposite inequality
\begin{displaymath}
\mu\big(B(1) \cap H_{\theta}(K,\delta^{-\sigma},[\delta,1])\big) \leq \delta^{\epsilon_{0}} \ll \delta^{C_{\zeta}\epsilon}
\end{displaymath}
for at least $\tfrac{1}{2}$ of the points in $E$, assuming that $\epsilon < \epsilon_{0}/C_{\zeta}$, simply because any such half is a $(\delta,s,\delta^{-\epsilon_{0}})$-set. In particular, for $\tfrac{1}{2}$ of the points in $E$, the difference set
\begin{equation}\label{form53} B(1) \cap H_{\theta}(K,\delta^{-\sigma + \zeta},[\delta,1]) \, \setminus \, H_{\theta}(K,\delta^{-\sigma},[\delta,1]) \end{equation}
has $\mu$ measure at least $\tfrac{1}{2}\delta^{C_{\zeta}\epsilon}$. In this section, we apply the same idea to remove from $H_{\theta}(K,\delta^{-\sigma + \zeta},[\delta,1])$ a more complicated set of high multiplicity.

In fact, we apply Theorem \ref{thm2} with parameter $\eta := \sqrt{\epsilon}$. Recall also Remark \ref{rem1}, and note that with our notation $\eta = \sqrt{\epsilon}$, we have $2C_{\zeta}\epsilon < \eta\epsilon_{0}/C$ as soon as $\epsilon < (\epsilon_{0}/C_{\zeta}')^{2}$ -- as we may assume. Now, Theorem \ref{thm2} yields the following conclusion for at least $\tfrac{1}{2}$ of the points in $E$:
\begin{equation}\label{form51}
\mu\big(B(1) \cap H_{\theta,\mathrm{loc}}(K,\sigma,\delta,\delta^{\sqrt{\epsilon}})\big) \leq \delta^{2C_{\zeta}\epsilon}.
\end{equation}
We then replace $E$ by the acquired subset without changing notation. At this point, the $(\Delta,s - \zeta,\Delta^{-O_{\zeta}(\epsilon)})$-property of the renormalisations $E_{I}$ might have failed, but this property can be restored by replacing the "new" $E$ by a further $\{2^{-jT}\}_{j = 1}^{m}$-regular subset, just like we did in the argument after \eqref{form28}.

For the remaining $\theta \in E$, we now define the set
\begin{equation}\label{form54} K_{\theta} := B(1) \cap H_{\theta}(K,\delta^{-\sigma + \zeta},[\delta,1]) \, \setminus \, H_{\theta,\mathrm{loc}}(K,\sigma,\delta,\delta^{\sqrt{\epsilon}}). \end{equation}
Comparing \eqref{form48} and \eqref{form51}, we obtain
\begin{equation}\label{form52} \mu(K_{\theta}) \geq \delta^{C_{\zeta}\epsilon} - \delta^{2C_{\zeta}\epsilon} = \delta^{O_{\zeta}(\epsilon)}\quad\text{for all } \theta \in E. \end{equation}
It is easy to check from the definitions that $K_{\theta} \cap H_{\theta}(K,4\delta^{-\sigma},[\delta,1]) = \emptyset$, so removing $H_{\theta,\mathrm{loc}}(\ldots)$ is a strictly more powerful manoeuvre than what we initially discussed at \eqref{form53} (up to the irrelevant constant "$4$").

\subsection{Projecting the sets $K_{\theta}$} We record the following useful lemma whose proof is a good exercise in applying the definition of $H_{\theta,\mathrm{loc}}$:
\begin{lemma}\label{lemma5} Let $B \subset \R^{2}$ be a disc of radius $\Delta \in [\delta^{1 - \sqrt{\epsilon}},1]$. Then,
\begin{equation}\label{form59} |\pi_{\theta}(B \cap K_{\theta})|_{\delta} \gtrsim \delta^{O_{\zeta}(\epsilon) - t} \cdot \left(\tfrac{\delta}{\Delta} \right)^{\sigma}\mu(B \cap K_{\theta})\quad\text{for all } \theta \in E. \end{equation}
\end{lemma}

\begin{remark} Eventually the lemma will be applied with $\Delta = \delta^{1/2}$. \end{remark}

\begin{proof} Let $\mathcal{T}_{\delta}$ be a minimal cover of $B \cap K_{\theta}$ by $\delta$-tubes perpendicular to the projection $\pi_{\theta}$, i.e. parallel to $\pi_{\theta}^{-1}\{0\}$. Then evidently
\begin{equation}\label{form58} |\pi_{\theta}(B \cap K_{\theta})|_{\delta} \gtrsim |\mathcal{T}_{\delta}|. \end{equation}
We now claim that
\begin{equation}\label{form56} |(B \cap K) \cap T|_{\delta} \lesssim \left(\tfrac{\Delta}{\delta} \right)^{\sigma}\quad\text{for all } T \in \mathcal{T}_{\delta}. \end{equation}
To see this, fix $T \in \mathcal{T}_{\delta}$. Note that since $T$ is a minimal cover of $B \cap K_{\theta}$, the tube $T$ contains at least one point $x_{0} \in B \cap K_{\theta}$, and by the definition of "$K_{\theta}$" we have
\begin{displaymath} x_{0} \notin H_{\theta,\mathrm{loc}}(K,\sigma,\delta,\delta^{\sqrt{\epsilon}}) \quad \Longrightarrow \quad x_{0} \notin H_{\theta}(K,4\left(\tfrac{\Delta}{\delta} \right)^{\sigma},[8\delta,8\Delta]). \end{displaymath}
(Here we also used $(8\delta)/(8\Delta) \leq \delta^{\sqrt{\epsilon}}$.) Unwrapping the definitions even further,
\begin{equation}
      \label{form57} |B(x_{0},8\Delta) \cap K^{(8\delta)} \cap \pi_{\theta}^{-1}\{\pi_{\theta}(x_{0})\}|_{8\delta} \leq 4\left(\tfrac{\Delta}{\delta} \right)^{\sigma}. 
\end{equation}
Now, notice that since $B$ is a disc of radius $\Delta$ containing $x_{0}$, we have $B \subset B(x_{0},3\Delta)$. If the intersection $(B \cap K) \cap T$ contained $\gg (\Delta/\delta)^{\sigma}$ points which are at least $\delta$-separated (i.e. \eqref{form56} failed), then a little argument using the triangle inequality would show that the line $\pi_{\theta}^{-1}\{\pi_{\theta}(x_{0})\} \subset T$ would intersect $B(x_{0},8\Delta) \cap K^{(8\delta)}$ in many more than $(\Delta/\delta)^{\sigma}$ points which are $8\delta$-separated. In other words, a failure of \eqref{form56} leads to the failure of \eqref{form57}. Thus \eqref{form56} holds.

We have now shown that each intersection $(B \cap K) \cap T$, $T \in \mathcal{T}_{\delta}$, can be covered by $\lesssim (\Delta/\delta)^{\sigma}$ discs of radius $\delta$. It follows from the $(\delta,t,\delta^{-O_{\zeta}(\epsilon)})$-regularity of $\mu$ that
\begin{displaymath} \mu(B \cap K_{\theta}) \leq \sum_{T \in \mathcal{T}_{\delta}} \mu((B \cap K) \cap T) \lesssim |\mathcal{T}_{\delta}| \cdot \delta^{t - O_{\zeta}(\epsilon)}\left(\tfrac{\Delta}{\delta} \right)^{\sigma}, \end{displaymath}
and finally,
\begin{displaymath} |\pi_{\theta}(K_{\theta} \cap B)|_{\delta} \stackrel{\eqref{form58}}{\gtrsim} |\mathcal{T}_{\delta}| \gtrsim \delta^{O_{\zeta}(\epsilon) - t} \cdot \left(\tfrac{\delta}{\Delta} \right)^{\sigma}\mu(B \cap K_{\theta}). \end{displaymath}
This completes the proof of \eqref{form59}. \end{proof}

The lower bound in Lemma \ref{lemma5} was based on the fact that $K_{\theta}$ is disjoint from the set $H_{\theta,\mathrm{loc}}(\ldots)$. In similar spirit, the fact that $K_{\theta}$ is a subset of $H_{\theta}(K,\delta^{-\sigma + \zeta},[\delta,1])$ yields an upper bound the $\delta$-covering number of $\pi_{\theta}(K_{\theta})$, as follows:

\begin{lemma}\label{lemma6} Let $\theta \in E$, let $\Delta \in [\delta^{1 - \sqrt{\epsilon}},1]$, and let $\mathbf{T}$ be a tube of width $\Delta$ parallel to $\pi_{\theta}^{-1}\{0\}$. Then,
\begin{equation}\label{form61} |\pi_{\theta}(K_{\theta} \cap \mathbf{T})|_{\delta} \lesssim \delta^{-O_{\zeta}(\epsilon) - \zeta} \cdot \left(\tfrac{\Delta}{\delta} \right)^{t - \sigma}.  \end{equation}
In particular, $|\pi_{\theta}(K_{\theta})|_{\delta} \lesssim \delta^{-O_{\zeta}(\epsilon) - \zeta} \cdot \delta^{\sigma - t}$.
\end{lemma}

\begin{proof} We may assume that $K_{\theta} \cap \mathbf{T} \neq \emptyset$, otherwise there is nothing to prove. Let $\mathcal{T}_{\delta}$ be a minimal cover of $K_{\theta} \cap \mathbf{T}$ by $\delta$-tubes $T \subset \mathbf{T}$ parallel to $\pi_{\theta}^{-1}\{0\}$. It suffices to show that
\begin{equation}\label{form62} |\mathcal{T}_{\delta}| \lesssim \delta^{-O_{\zeta}(\epsilon) - \zeta} \cdot \left(\tfrac{\Delta}{\delta} \right)^{t - \sigma}. \end{equation}
To see this, fix $T \in \mathcal{T}_{\delta}$ and $x \in K_{\theta} \cap T$. Thus $x \in H_{\theta}(K,\delta^{-\sigma + \zeta},[\delta,1])$, so
\begin{displaymath} |B(2) \cap K^{(\delta)} \cap \pi_{\theta}^{-1}\{\pi_{\theta}(x)\}|_{\delta} \geq \mathfrak{m}_{K,\theta}(x \mid [\delta,1]) \geq \delta^{-\sigma + \zeta}. \end{displaymath}
Summing over the tubes $T \in \mathcal{T}_{\delta}$, we infer that
\begin{equation}\label{form60} |B(2) \cap K^{(\delta)} \cap \mathbf{T}|_{\delta} \gtrsim |\mathcal{T}_{\delta}| \cdot \delta^{-\sigma + \zeta}. \end{equation}
To find a useful upper bound for the left-hand side, fix $x_{0} \in K_{\theta} \cap \mathbf{T}$ arbitrary, and recall that
\begin{displaymath} x_{0} \notin H_{\theta,\mathrm{loc}}(K,\sigma,\delta,\delta^{\sqrt{\epsilon}}) \quad \Longrightarrow \quad x_{0} \notin H_{\theta}(K,4\Delta^{-\sigma},[8\Delta,8]), \end{displaymath}
or in other words
\begin{displaymath} |B(x_{0},8) \cap K^{(8\Delta)} \cap \pi_{\theta}^{-1}\{\pi_{\theta}(x_{0})\}|_{8\Delta} \leq 4\Delta^{-\sigma}. \end{displaymath}
This easily implies that $B(2) \cap K^{(\delta)} \cap \mathbf{T}$ can be covered by $\lesssim \Delta^{-\sigma}$ discs "$B$" of radius $\Delta$. Since $|B \cap K|_{\delta} \leq \delta^{-O_{\zeta}(\epsilon)}(\Delta/\delta)^{t}$ for each "$B$" by the $(\delta,t,\delta^{-O_{\zeta}(\epsilon)})$-regularity of $\mu$, we obtain
\begin{displaymath} |\mathcal{T}_{\delta}| \cdot \delta^{-\sigma + \zeta} \stackrel{\eqref{form60}}{\lesssim} |B(2) \cap K^{(\delta)} \cap \mathbf{T}|_{\delta} \lesssim \delta^{-O_{\zeta}(\epsilon)}\Delta^{-\sigma} \cdot \left(\tfrac{\Delta}{\delta} \right)^{t}. \end{displaymath}
Dividing by $\delta^{-\sigma + \zeta}$ implies \eqref{form62} and therefore \eqref{form61}. \end{proof}

\subsection{Choosing a good \texorpdfstring{$\Delta$}{Delta}-tube} Recall the sets $K_{\theta}$ defined at \eqref{form54}, which had measure $\mu(K_{\theta}) \geq \delta^{O_{\zeta}(\epsilon)}$ for all $\theta \in E$ by \eqref{form52}. Further, recall that if $I \subset \mathcal{D}_{\Delta}(E)$ is arbitrary, then $E_{I}$ is a $(\Delta,s - \zeta,\Delta^{-O_{\zeta}(\epsilon)})$-set (Proposition \ref{prop4}(a)). Here $\Delta = \delta^{1/2}$. From $\mu(B(1)) \leq \delta^{-\epsilon}$ and Cauchy-Schwarz, it easily follows that
\begin{displaymath} \sum_{\theta,\theta' \in E \cap I} \mu(K_{\theta} \cap K_{\theta'}) \geq \delta^{O_{\zeta}(\epsilon)}|E \cap I|^{2}, \end{displaymath}
and in particular there exists $\theta_{0} \in E \cap I$ with the property
\begin{displaymath} \sum_{\theta \in E \cap I} \mu(K_{\theta_{0}} \cap K_{\theta}) \geq \delta^{O_{\zeta}(\epsilon)}|E \cap I|. \end{displaymath}
Further, from this inequality it follows that there exists a subset of the form $E' \cap I \subset E \cap I$ with $|E' \cap I| \geq \delta^{O_{\zeta}(\epsilon)}|E \cap I|$ such that $\mu(K_{\theta} \cap K_{\theta_{0}}) \geq \delta^{O_{\zeta}(\epsilon)}$ for all $\theta \in E' \cap I$. Since the renormalisation $E_{I}'$ remains a $(\Delta,s - \zeta,\Delta^{-O_{\zeta}(\epsilon)})$-set, the difference between $E \cap I$ and $E' \cap I$ will be irrelevant to us, and we simplify notation by assuming that
\begin{equation}\label{form37} \mu(K_{\theta_{0}} \cap K_{\theta}) \geq \delta^{O_{\zeta}(\epsilon)}\quad\text{for all } \theta \in E \cap I. \end{equation}
The arc $I \in \mathcal{D}_{\Delta}(E)$ and the point $\theta_{0} \in E \cap I$ will remain fixed for the remainder of the proof. Since our problem is rotation-invariant, we may assume that $\theta_{0} = (1,0)$, so the projection $\pi_{\theta_{0}}(x,y) = x$ is the projection to the $x$-axis, and $I$ is an arc of length $\Delta$ around $(1,0)$. For technical convenience, it will be useful to re-parametrise the projections $\pi_{\theta}$, $\theta \in I$, in the following standard way:
\begin{equation}\label{form55} \pi_{\theta}(x,y) = x + \theta y\quad\text{for all } \theta \in I = [0,\Delta] = [0,\delta^{1/2}]. \end{equation}
So, when we apply the definition of $K_{\theta}$ in the near future and write "$\pi_{\theta}^{-1}\{\pi_{\theta}(x)\}$" we refer precisely to the maps in \eqref{form55}. We abbreviate
\begin{displaymath} K_{0} := K_{\theta_{0}} \quad \text{and} \quad \pi := \pi_{\theta_{0}}, \end{displaymath}
so the lines $\pi^{-1}\{\pi(x)\}$, $x \in \R$, are parallel to the $y$-axis. For $\theta \in I$, the lines $\pi_{\theta}^{-1}\{\pi_{\theta}(x)\}$ make an angle $\leq \Delta$ with the $y$-axis.

The plan is, next, to investigate the intersection of $K_{0}$ with a "typical" vertical tube $\mathbf{T}$ of width $\Delta$. Roughly speaking, it turns out that the minimal cover of $K_{0} \cap \mathbf{T}$ with $\Delta$-discs consists of $\approx \Delta^{-\sigma}$ discs satisfying a $\sigma$-dimensional non-concentration condition. Once this has been verified, we (still roughly speaking) restrict attention to one of these "typical" tubes $\mathbf{T}_{0}$ for the remainder of the argument.

Let $\mathcal{B}_{\Delta}$ be a minimal cover of $B(1) \cap K$ with discs of radius $\Delta$, satisfying $|\mathcal{B}_{\Delta}| \leq \Delta^{-O_{\zeta}(\epsilon) - t}$.
We note that
\begin{equation}\label{form38} \sum_{\theta \in E \cap I} \sum_{B \in \mathcal{B}_{\Delta}} \mu(K_{0} \cap K_{\theta} \cap B) \geq \sum_{\theta \in E \cap I} \mu(K_{0} \cap K_{\theta}) \stackrel{\eqref{form37}}{\geq} \delta^{O_{\zeta}(\epsilon)}|E \cap I|. \end{equation}
A disc $B \in \mathcal{B}_{\Delta}$ is called \emph{light} (denoted $B \in \mathcal{B}_{\Delta}^{\mathrm{light}}$) if
\begin{displaymath} \frac{1}{|E \cap I|} \sum_{\theta \in E \cap I} \mu(K_{0} \cap K_{\theta} \cap B) \leq \Delta^{t + C_{\zeta}\epsilon}, \end{displaymath}
where $C_{\zeta} \geq 1$ is a constant to be determined momentarily. Observe that
\begin{displaymath} \sum_{\theta \in E \cap I} \sum_{B \in \mathcal{B}_{\Delta}^{\mathrm{light}}} \mu(K_{0} \cap K_{\theta} \cap B) \leq |\mathcal{B}_{\Delta}||E \cap I| \Delta^{t + C_{\zeta}\epsilon} \leq \Delta^{C_{\zeta}\epsilon - O_{\zeta}(\epsilon)}|E \cap I|, \end{displaymath}
so in particular if $\mathcal{B}_{\Delta}^{\mathrm{heavy}} := \mathcal{B}_{\Delta} \, \setminus \, \mathcal{B}_{\Delta}^{\mathrm{light}}$, then
\begin{equation}\label{form43} \sum_{\theta \in E \cap I} \sum_{B \in \mathcal{B}_{\Delta}^{\mathrm{heavy}}} \mu(K_{0} \cap K_{\theta} \cap B) \stackrel{\eqref{form38}}{\geq} (\delta^{O_{\zeta}(\epsilon)} - \Delta^{C_{\zeta}\epsilon - O_{\zeta}(\epsilon)})|E \cap I| \geq \delta^{O_{\zeta}(\epsilon)}|E \cap I|, \end{equation}
assuming that the constant "$C_{\zeta}$" in the definition of "lightness" was chosen five times larger than the "$O_{\zeta}(\epsilon)$" constants.

We make the following simple observation about the heavy discs:
\begin{equation}\label{form39} \mu(K_{0} \cap B) \geq \frac{1}{|E \cap I|} \sum_{\theta \in E \cap I} \mu(K_{0} \cap K_{\theta} \cap B) \geq \delta^{O_{\zeta}(\epsilon)}\Delta^{t}\quad\text{for all } B \in \mathcal{B}_{\Delta}^{\mathrm{heavy}}. \end{equation}
Consequently, it follows from Lemma \ref{lemma5} (and $\delta/\Delta = \Delta$) that
\begin{equation}\label{form40} |\pi(B \cap K_{0})|_{\delta} \geq \delta^{O_{\zeta}(\epsilon)}\Delta^{\sigma - t}\quad\text{for all } B \in \mathcal{B}_{\Delta}^{\mathrm{heavy}}. \end{equation}
Next, let $\mathcal{T}_{\Delta}$ be a minimal cover of the heavy discs by disjoint $\Delta$-tubes perpendicular to $\theta_{0}$ (that is, parallel to the $y$-axis). In particular, every tube in $\mathcal{T}_{\Delta}$ meets at least one disc in $\mathcal{B}_{\Delta}^{\mathrm{heavy}}$. We claim that
\begin{equation}\label{form42} |\mathcal{T}_{\Delta}| \leq \Delta^{\sigma - t - 3\zeta}. \end{equation}
Since each of the tubes $\mathbf{T} \in \mathcal{T}_{\Delta}$ meets at least one disc $B \in \mathcal{B}_{\Delta}^{\mathrm{heavy}}$ (and each of these discs can meet at most $3$ tubes), we deduce that
\begin{displaymath} |\pi(K_{0})|_{\delta} \stackrel{\eqref{form40}}{\gtrsim} \delta^{O_{\zeta}(\epsilon)}|\mathcal{T}_{\Delta}|\Delta^{\sigma - t}. \end{displaymath}
On the other hand, a special case of Lemma \ref{lemma6} states that $|\pi(K_{0})|_{\delta} \lesssim \delta^{-O_{\zeta}(\epsilon) - \zeta} \cdot \delta^{\sigma - t}$. The upper bound \eqref{form42} follows by combining these two inequalities (note that $\delta/\Delta = \Delta$), and choosing $\epsilon = o_{\zeta}(1)$ so small that $O_{\zeta}(\epsilon) \leq \zeta$.

Next, for every $\mathbf{T} \in \mathcal{T}_{\Delta}$, write
\begin{displaymath} \mathcal{B}(\mathbf{T}) := \{B \in \mathcal{B}_{\Delta}^{\mathrm{heavy}} : B \cap \mathbf{T} \neq \emptyset\}. \end{displaymath}
Since the union of the tubes in $\mathcal{T}_{\Delta}$ cover all the heavy discs, we have
\begin{equation}\label{form44} \sum_{T \in \mathcal{T}_{\Delta}} \sum_{\theta \in E \cap I} \sum_{B \in \mathcal{B}(\mathbf{T})} \mu(K_{0} \cap K_{\theta} \cap B) \stackrel{\eqref{form43}}{\geq} \delta^{O_{\zeta}(\epsilon)}|E \cap I|. \end{equation}
A tube $\mathbf{T} \in \mathcal{T}_{\Delta}$ is called \emph{heavy} if
\begin{equation}\label{form64} \sum_{\theta \in E \cap I} \sum_{B \in \mathcal{B}(\mathbf{T})} \mu(K_{0} \cap K_{\theta} \cap B) \geq \Delta^{t - \sigma + 4\zeta}|E \cap I|. \end{equation}
The heavy tubes are denoted $\mathcal{T}_{\Delta}^{\mathrm{heavy}}$. With this notation,
\begin{displaymath} \sum_{T \in \mathcal{T}_{\Delta} \, \setminus \, \mathcal{T}_{\Delta}^{\mathrm{heavy}}} \sum_{\theta \in E \cap I} \sum_{B \in \mathcal{B}(\mathbf{T})} \mu(K_{0} \cap K_{\theta} \cap B) \leq |\mathcal{T}_{\Delta}| \cdot \Delta^{t - \sigma + 4\zeta}|E \cap I| \end{displaymath}
Combining this estimate with the upper bound $|\mathcal{T}_{\Delta}| \leq \Delta^{\sigma - t - 3\zeta}$ established in \eqref{form42}, and inspecting \eqref{form44}, we see that the sum over the light tubes is less than half the total value of the sum in \eqref{form44}. As a consequence, the set of heavy tubes is non-empty. For the remainder of the whole proof, we fix one heavy tube
\begin{equation*} 
  \mathbf{T}_{0} \in \mathcal{T}_{\Delta}^{\mathrm{heavy}}. 
\end{equation*}
We record the following consequence of Lemma \ref{lemma6}:
\begin{equation} \label{form46} 
    |\pi_{\theta}(K_{\theta} \cap \mathbf{T}_{0})|_{\delta} \leq \delta^{-O_{\zeta}(\epsilon) - \zeta}\Delta^{\sigma - t} \leq \Delta^{\sigma - t - 4\zeta}\quad\text{for all } \theta \in E \cap I. 
\end{equation}
For the second inequality we took $\epsilon > 0$ so small depending on $\zeta$ that $O_{\zeta}(\epsilon) \leq \zeta$. Inequality \eqref{form46} looks like an immediate consequence of \eqref{form61} with $\Delta = \delta^{1/2}$, but the tube $\mathbf{T}_{0}$ is not exactly parallel to $\pi_{\theta}^{-1}\{0\}$. However, $\mathbf{T}_{0}$ is parallel to $\pi^{-1}\{0\} = \pi_{\theta_{0}}^{-1}\{0\}$, and since $|\theta - \theta_{0}| \leq \Delta$, we have $K_{\theta} \cap \mathbf{T}_{0} \subset K_{\theta} \cap 2\mathbf{T}_{\theta}$, where $\mathbf{T}_{\theta}$ is a $\Delta$-tube parallel to $\pi_{\theta}^{-1}\{0\}$. Thus, \eqref{form46} follows from Lemma \ref{lemma6} applied to $2\mathbf{T}_{\theta}$.

We record a $\sigma$-dimensional non-concentration condition for $\mathcal{B}(\mathbf{T}_{0})$:
\begin{lemma}\label{lemma8} We have
\begin{equation}\label{form47} |\mathcal{B}(\mathbf{T}_{0}) \cap B(x,R)| \lesssim \left(\tfrac{R}{\Delta} \right)^{\sigma}\quad\text{for all } x \in \R^{2}, \, R \in [\delta^{-\sqrt{\epsilon}}\Delta,1]. \end{equation}
Here $\mathcal{B}(\mathbf{T}_{0}) \cap B(x,R) := \{B \in \mathcal{B}(\mathbf{T}_{0}) : B \cap B(x,R) \neq \emptyset\}$. In particular, $|\mathcal{B}(\mathbf{T}_{0})| \lesssim \Delta^{-\sigma}$. \end{lemma}

\begin{proof} To prove \eqref{form47}, fix $x \in \R^{2}$ and $R \in [\delta^{-\sqrt{\epsilon}}\Delta,1]$, and let $B \in \mathcal{B}(\mathbf{T}_{0}) \cap B(x,R)$. Then in particular $B \in \mathcal{B}_{\Delta}^{\mathrm{heavy}}$, so $B \cap K_{0} \neq \emptyset$ according to \eqref{form39}. Fix $x_{0} \in B \cap K_{0}$, and recall that
\begin{displaymath} x_{0} \notin H_{\theta_{0},\mathrm{loc}}(K,\sigma,\delta,\delta^{\sqrt{\epsilon}}) \quad \Longrightarrow \quad x_{0} \notin H_{\theta_{0}}(K,4\left(\tfrac{R}{\Delta} \right)^{\sigma},[8\Delta,8R]). \end{displaymath}
The implication is valid since $\Delta/R \leq \delta^{\sqrt{\epsilon}}$, and $8R \leq 8$. Now, by the definition of $H_{\theta_{0}}(\ldots)$, we deduce that
\begin{displaymath} |B(x_{0},8R) \cap K^{(8\Delta)} \cap \pi^{-1}\{\pi(x_{0})\}|_{8\Delta} = \mathfrak{m}_{\theta_{0}}(x_{0} \mid [8\Delta,8R]) \leq 4\left(\tfrac{R}{\Delta} \right)^{\sigma}. \end{displaymath}
Note that $B(x,R) \subset B(x_{0},8R)$ since $x_{0} \in B$ and $B \cap B(x,R) \neq \emptyset$. Recalling (above \eqref{form38}) that $\mathcal{B}_{\Delta}$ is a minimal cover of $B(1) \cap K$, the previous inequality even shows that
\begin{equation*}
  |\{B \in \mathcal{B}_{\Delta} : B \cap \mathbf{T}_{0} \cap B(x,R) \neq \emptyset\}| \lesssim \left(\tfrac{R}{\Delta} \right)^{\sigma}, 
\end{equation*}
and this implies \eqref{form47}. \end{proof}

By the definition of heaviness, the tube $\mathbf{T}_{0}$ satisfies the lower bound \eqref{form64}. We claim that, as a consequence, there exists a subset $E' \cap I \subset E \cap I$ of cardinality $|E' \cap I| \geq \Delta^{5\zeta}|E \cap I|$ and such that
\begin{equation}\label{form67} \sum_{B \in \mathcal{B}(\mathbf{T}_{0})} \mu(B \cap K_{0} \cap K_{\theta}) \geq \Delta^{t - \sigma + 5\zeta}\quad\text{for all } \theta \in E' \cap I. \end{equation}
This is a straightforward consequence of \eqref{form64}, and the following inequality which is based on Lemma \ref{lemma8} and the $(t,\delta^{-O_{\zeta}(\epsilon)})$-regularity of $\mu$:
\begin{displaymath} \sum_{B \in \mathcal{B}(\mathbf{T}_{0})} \mu(B \cap K_{0} \cap K_{\theta}) \leq |\mathcal{B}(\mathbf{T}_{0})| \cdot \Delta^{t - O_{\zeta}(\epsilon)} \lesssim \Delta^{t - \sigma - O_{\zeta}(\epsilon)} \qquad \theta \in E \cap I. \end{displaymath}
This shows that in order for \eqref{form64} to be true, the inequality \eqref{form67} must hold for all $\theta \in E' \cap I$ with $|E' \cap I| \geq \Delta^{5\zeta}|E \cap I|$. Now, as we have done many times before, we replace $E \cap I$ by $E' \cap I$ without changing notation: the only property of $E' \cap I$ we will need eventually is that $E_{I}'$ is a $(\Delta,s - \zeta,\Delta^{-O(\zeta)})$-set.\footnote{This is a little weaker than the information we have had so far that $E_{I}'$ is a $(\Delta,s - \zeta,\Delta^{-O_{\zeta}(\epsilon)})$-set. Fortunately, this is all we will need in the sequel.} Thus, we assume in the sequel that
\begin{equation}\label{form68} \sum_{B \in \mathcal{B}(\mathbf{T})} \mu(B \cap K_{0} \cap K_{\theta}) \geq \Delta^{t - \sigma + 5\zeta}\quad\text{for all } \theta \in E \cap I. \end{equation}

\subsection{The sets $\mathcal{A}$ and $\mathcal{A}_{\theta}$} 
Let
\begin{displaymath} \mathcal{A} := \mathcal{D}_{\delta}(\pi(K_{0} \cap \mathbf{T}_{0})). \end{displaymath}
We record that
\begin{equation}\label{form71} |\mathcal{A}| = |\pi(K_{0} \cap \mathbf{T}_{0})|_{\delta} \stackrel{\eqref{form46}}{\leq} \Delta^{\sigma - t - 4\zeta}. \end{equation}
Fix $\theta \in E \cap I$, and define the following subset $\mathcal{A}_{\theta} \subset \mathcal{A}$. We declare that $I \in \mathcal{A}_{\theta}$ if $I \in \mathcal{A}$, and
\begin{equation}\label{form75} |\{B \in \mathcal{B}(\mathbf{T}_{0}) : \pi^{-1}(I) \cap (B \cap K_{0} \cap K_{\theta}) \neq \emptyset\}| \geq \Delta^{-\sigma + 11\zeta}. \end{equation}
We claim that
\begin{equation}\label{form70} |\mathcal{A}_{\theta}| \geq \Delta^{\sigma - t + 6\zeta}. \end{equation}
The proof is, once again, based on the fact that $K_{0}$ lies in the complement of $H_{\theta_{0},\mathrm{loc}}(\ldots)$. This is used via the following lemma:
\begin{lemma} Let $B \in \mathcal{B}(\mathbf{T}_{0})$. Then,
\begin{equation}\label{form65} \mu(\pi^{-1}(I) \cap (B \cap K_{0})) \lesssim \Delta^{-\sigma} \cdot \delta^{t - O_{\zeta}(\epsilon)} \quad\text{for all } I \in \mathcal{D}_{\delta}(\R). \end{equation}
\end{lemma}

\begin{proof} Fix $I \in \mathcal{D}_{\delta}(\R)$ and write $T := \pi^{-1}(I)$. If $T \cap (B \cap K_{0}) = \emptyset$, there is nothing to prove, so assume that there exists at least one point $x_{0} \in T \cap (B \cap K_{0})$. In particular,
\begin{displaymath} x_{0} \notin H_{\theta_{0},\mathrm{loc}}(K,\sigma,\delta,\delta^{\sqrt{\epsilon}}) \quad \Longrightarrow \quad x_{0} \notin H_{\theta_{0}}(K,4\left(\tfrac{\Delta}{\delta} \right)^{\sigma},[8\delta,8\Delta]), \end{displaymath}
or in other words
\begin{displaymath} |B(x_{0},8\Delta) \cap K^{(8\delta)} \cap \pi^{-1}\{\pi(x_{0})\}|_{8\delta} \leq 4\left(\tfrac{\Delta}{\delta} \right)^{\sigma} \sim \Delta^{-\sigma}, \end{displaymath}
using $\Delta = \delta^{1/2}$. Since $\pi^{-1}\{\pi(x_{0})\} \subset T$ and $B \subset B(x_{0},8\Delta)$, this easily implies that the intersection $T \cap (B \cap K_{0})$ can be covered by $\lesssim \Delta^{-\sigma}$ discs of radius $\delta$, and now the inequality \eqref{form65} follows from the $(\delta,t,\delta^{-O_{\zeta}(\epsilon)})$-regularity of $\mu$. \end{proof}

To proceed with the proof of \eqref{form70}, let $\mathcal{T}_{\theta} := \{\pi^{-1}(I) : I \in \mathcal{A} \, \setminus \, \mathcal{A}_{\theta}\}$. Thus, the tubes $\pi^{-1}(I)$, $I \in \mathcal{A} \, \setminus \, \mathcal{A}_{\theta}$, can intersect $B \cap K_{0} \cap K_{\theta}$ for at most $\leq \Delta^{-\sigma + 11\zeta}$ different discs $B \in \mathcal{B}(\mathbf{T}_{0})$. Applying \eqref{form65} for each of those discs individually leads to
\begin{displaymath}
\sum_{B \in \mathcal{B}(\mathbf{T}_{0})} \mu\big(\pi^{-1}(I) \cap (B \cap K_{0} \cap K_{\theta})\big) \leq \Delta^{10\zeta} \delta^{t - \sigma}\quad\text{for all } I \in \mathcal{A} \, \setminus \, \mathcal{A}_{\theta},
\end{displaymath}
assuming that $O_{\zeta}(\epsilon) < \zeta$. Summing over $I \in \mathcal{A} \, \setminus \, \mathcal{A}_{\theta}$, it follows that
\begin{equation}\label{form69}
\sum_{I \in \mathcal{A} \, \setminus \, \mathcal{A}_{\theta}} \sum_{B \in \mathcal{B}(\mathbf{T}_{0})} \mu\big(\pi^{-1}(I) \cap (B \cap K_{0} \cap K_{\theta})\big) \leq |\mathcal{A}| \cdot \Delta^{10\zeta}\delta^{t - \sigma} \stackrel{\eqref{form71}}{\leq} \Delta^{t - \sigma + 6\zeta}.
\end{equation}
On the other hand, the "full sum" over $I \in \mathcal{A}$ has the lower bound
\begin{displaymath}
\sum_{I \in \mathcal{A}} \sum_{B \in \mathcal{B}(\mathbf{T}_{0})} \mu\big(\pi^{-1}(I) \cap (B \cap K_{0} \cap K_{\theta})\big) \geq \sum_{B \in \mathcal{B}(\mathbf{T}_{0})} \mu(B \cap K_{0} \cap K_{\theta}) \stackrel{\eqref{form68}}{\geq} \Delta^{t - \sigma + 5\zeta},
\end{displaymath}
so by \eqref{form69} the full sum cannot be dominated by the part over $I \in \mathcal{A} \, \setminus \, \mathcal{A}_{\theta}$. Consequently,
\begin{align*} \Delta^{t - \sigma + 5\zeta} & \leq 2\sum_{I \in \mathcal{A}_{\theta}} \sum_{B \in \mathcal{B}(\mathbf{T}_{0})} \mu(\pi^{-1}(I) \cap (B \cap K_{0} \cap K_{\theta}))\\
& \stackrel{\eqref{form65}}{\lesssim} |\mathcal{A}_{\theta}| \cdot |\mathcal{B}(\mathbf{T}_{0})| \cdot \Delta^{-\sigma} \cdot \delta^{t - O_{\zeta}(\epsilon)} \stackrel{\eqref{form47}}{\lesssim} |\mathcal{A}_{\theta}| \cdot \Delta^{-2\sigma} \cdot \delta^{t - O_{\zeta}(\epsilon)},  \end{align*}
and therefore $|\mathcal{A}_{\theta}| \geq \Delta^{\sigma - t + 6\zeta}$, as claimed in \eqref{form70}.

\subsection{Violating the \texorpdfstring{$ABC$}{ABC} sum-product theorem} Let $A$ be the left end-points of the intervals in the collection $\mathcal{A} = \mathcal{D}_{\delta}(\pi(K_{0} \cap \mathbf{T}_{0}))$. Therefore $A$ is a $\delta$-separated subset of the interval $\pi(\mathbf{T}_{0})$. This interval has length $\Delta$, and there is no loss of generality in assuming that
\begin{displaymath} A \subset \pi(\mathbf{T}_{0}) = [0,\Delta]. \end{displaymath}
Next, we define the set "$B$" to consist of the $y$-coordinates of the centres of the discs in $\mathcal{B}(\mathbf{T}_{0})$. (We are aware of the risk of notational confusion, but we will make sure that the writing is clear; we prefer the notation "$B$" to make the connection with the $ABC$ theorem transparent.) Since $\mathbf{T}_{0}$ is a vertical tube of width $\Delta$, and the discs in $\mathcal{B}(\mathbf{T}_{0})$ all intersect $B(1)$, there is no loss of generality in assuming that $B$ is a $\Delta$-separated subset of $[0,1]$. Moreover, the "non-concentration" of the discs in $\mathcal{B}(\mathbf{T}_{0})$ recorded in Lemma \ref{lemma8} is inherited by the set $B$. We claim the following corollary:

\begin{cor} The set $B$ satisfies the following non-concentration condition if $\delta,\zeta > 0$ are sufficiently small:
\begin{equation}\label{form72} |B \cap B(x,r)|_{\Delta} \leq r^{\sigma - 6\sqrt{\zeta}}|B| \quad\text{for all } x \in \R, \, r \in [\Delta,\Delta^{\sqrt{\zeta}}], 
\end{equation}
\end{cor}

\begin{proof} To begin with, we observe that
\begin{equation}\label{form73} |B| = |\mathcal{B}(\mathbf{T}_{0})| \geq \Delta^{-\sigma + 6\zeta} \end{equation}
by \eqref{form67} and the $(t,\delta^{-O_{\zeta}(\epsilon)})$-regularity of $\mu$. Therefore, the non-concentration condition recorded in Lemma \ref{lemma8} implies that
\begin{displaymath} 
  |B \cap B(x,r)|_{\Delta} \lesssim \left(\tfrac{r}{\Delta} \right)^{\sigma} \leq \Delta^{-6\zeta}r^{\sigma}|B| \quad\text{for all } r \in [\delta^{-\sqrt{\epsilon}}\Delta,1]. 
\end{displaymath}
For $r \leq \Delta^{\sqrt{\zeta}}$, we have $\Delta^{-6\zeta} \leq r^{-6\sqrt{\zeta}}$. Thus, the inequality \eqref{form72} holds at least for $r \in [\delta^{-\sqrt{\epsilon}}\Delta,\Delta^{\sqrt{\zeta}}]$.  For $r \in [\Delta,\delta^{-\sqrt{\epsilon}}\Delta]$, we can simply use the trivial estimate
\begin{displaymath} |B \cap B(x,r)|_{\Delta} \lesssim r/\Delta \leq \delta^{-\sqrt{\epsilon}} \leq (\delta^{-\sqrt{\epsilon}}\Delta)^{-\zeta} \leq r^{-\zeta} \stackrel{\eqref{form73}}{\leq}r^{\sigma - 7\zeta}|B|. \end{displaymath}
 Since $7\zeta \leq 6\sqrt{\zeta}$ for $\zeta > 0$ small (as we assume), this proves \eqref{form72}. \end{proof}

In summary, $A \times B$ is a $(\delta \times \Delta)$-separated product subset of $[0,\Delta] \times [0,1]$ with the properties
\begin{equation}\label{form74} \Delta^{\sigma - t + 6\zeta} \stackrel{\eqref{form70}}{\leq} |A| = |\mathcal{A}| \stackrel{\eqref{form71}}{\leq} \Delta^{\sigma - t - 2\zeta} \quad \text{and} \quad |B| \stackrel{\eqref{form73}}{\geq} \Delta^{-\sigma + 6\zeta}, \end{equation}
and such that $B$ satisfies the $(\sigma - o_{\zeta}(1))$-dimensional non-concentration condition recorded in \eqref{form72}. For each $\theta \in E$, we next proceed to define a substantial subset $G_{\theta} \subset A \times B$ with a small $\pi_{\theta}$-projection at scale $\delta$. The starting point is the interval collection $\mathcal{A}_{\theta} \subset \mathcal{A}$ defined at \eqref{form75}. Let $A_{\theta} \subset A$ be the left end-points of the intervals in $\mathcal{A}_{\theta}$. Then, for each $I \in \mathcal{A}_{\theta}$, consider the subset $B_{I,\theta} \subset B$ defined by
\begin{displaymath} B_{I,\theta} := \{B \in \mathcal{B}(\mathbf{T}_{0}) : \pi^{-1}(I) \cap (B \cap K_{0} \cap K_{\theta}) \neq \emptyset\}. \end{displaymath}
(To be accurate, $B_{I,\theta}$ consists of the $y$-coordinates of the centres of the indicated discs.) There is a $1$-to-$1$ correspondence between the points $x \in A_{\theta}$ and $I \in \mathcal{A}_{\theta}$, so we may denote $B_{I,\theta} =: B_{x,\theta}$ for $x \in A_{\theta}$, and define
\begin{equation}\label{form87} G_{\theta} := \{(x,y) : x \in A_{\theta} \text{ and } y \in B_{x,\theta}\}. \end{equation}
By the defining property \eqref{form75} of the family $\mathcal{A}_{\theta}$, we have $|B_{x,\theta}| \geq \Delta^{-\sigma + 11\zeta}$ for all $x \in A_{\theta}$, and therefore
\begin{equation}\label{form88} |G_{\theta}| = \sum_{x \in A_{\theta}} |B_{x,\theta}| \stackrel{\eqref{form70}}{\geq} \Delta^{-t + 17\zeta} \stackrel{\eqref{form74}}{\geq} \Delta^{29\zeta}|A \times B|\quad\text{for all } \theta \in E \cap I. \end{equation}
On the other hand, it turns out that the $\pi_{\theta}$-projection of $G_{\theta}$ is controlled by the $\pi_{\theta}$-projection of $K_{\theta} \cap \mathbf{T}_{0}$:
\begin{lemma} We have
\begin{equation}\label{form76} |\pi_{\theta}(G_{\theta})|_{\delta} \lesssim |\pi_{\theta}(K_{\theta} \cap \mathbf{T}_{0})|_{\delta} \stackrel{\eqref{form46}}{\leq} \Delta^{\sigma - t - 2\zeta} \stackrel{\eqref{form74}}{\leq} \Delta^{-8\zeta}|A|\quad\text{for all } \theta \in E \cap I. \end{equation}
\end{lemma}

\begin{proof} Let $(x,y) \in G_{\theta}$. This implies, by definition, that $x \in I \in \mathcal{A}_{\theta}$ and $\pi^{-1}(I) \cap (B \cap K_{\theta}) \neq \emptyset$ for some $\Delta$-disc $B \in \mathcal{B}(\mathbf{T}_{0})$ whose centre has second coordinate "$y$". In particular, there exists a point $(x',y') \in \pi^{-1}(I) \cap K_{\theta}$ with the properties
\begin{displaymath} |x' - x| \leq \delta \quad \text{and} \quad |y' - y| \leq \Delta. \end{displaymath}
Now, observe that
\begin{displaymath} |\pi_{\theta}(x,y) - \pi_{\theta}(x',y')| \leq |x' - x| + |\theta||y' - y| \leq 2\delta\quad\text{for all } \theta \in E \cap I = E \cap [0,\Delta]. \end{displaymath}
In other words $\pi_{\theta}(G_{\theta})$ is contained in the $(2\delta)$-neighbourhood of $\pi_{\theta}(K_{\theta} \cap \mathbf{T}_{0})$, for every $\theta \in E \cap I$. This proves the lemma. \end{proof}

This is nearly what we need in order to apply -- or rather violate -- the $ABC$ sum-product theorem, Theorem \ref{thm:ABCConjecture}. To make the conclusion of the argument precise, we apply the dilation $(x,y) \mapsto D(x,y) := (\Delta^{-1}x,y)$ to the set $A \times B$, and also to the subsets $G_{\theta}$. Then, writing $A' := \Delta^{-1}A$, we find that $A' \times B = D(A \times B)$ is a $\Delta$-separated product set, and $G_{\theta}' := D(G_{\theta}) \subset A' \times B$ is a subset satisfying $|A'| = |A|$, and
\begin{equation}\label{form86} |G_{\theta}'| = |D(G_{\theta})| \stackrel{\eqref{form88}}{\geq} \Delta^{29\zeta}|A' \times B|\quad\text{for all } \theta \in E \cap I. \end{equation}
Moreover,
\begin{align*} \pi_{\Delta^{-1}\theta}(G_{\theta}') & = \{x + (\Delta^{-1}\theta)y : (x,y) \in G_{\theta}'\}\\
& = \Delta^{-1}\{\Delta x + \theta y : (x,y) \in G_{\theta}'\} = \Delta^{-1}\pi_{\theta}(G_{\theta})\quad\text{for all } \theta \in E \cap I.  \end{align*}
Since the renormalisation $E_{I}$ consists exactly of the points $\Delta^{-1}\theta$ with $\theta \in E \cap I$, the previous equation yields
\begin{equation}\label{form89} |\pi_{\theta}(G_{\theta}')|_{\Delta} = |\pi_{\Delta\theta}(G_{\theta})|_{\delta} \stackrel{\eqref{form76}}{\leq} \Delta^{-9\zeta}|A'|\quad\text{for all } \theta \in E_{I}. \end{equation}
Proposition \ref{prop4}(a) suggests that the renormalisation $E_{I}$ is a $(\Delta,s - \zeta,\Delta^{-O_{\zeta}(\epsilon)})$-set, but keeping in mind the various refinements to $E \cap I$, and in particular the latest one above \eqref{form68}, the correct conclusion is that $E_{I}$ is a $(\Delta,s - \zeta,\Delta^{-O(\zeta)})$-set.

It is time to apply Theorem \ref{thm:ABCConjecture} to the values
\begin{displaymath} \alpha := t - \sigma + \zeta_{0}, \quad \text{and} \quad \beta := \sigma - \zeta_{0}, \end{displaymath}
as announced at \eqref{eq:parameters}. We will apply the theorem at the scale $\Delta$. According to \eqref{form74}, the sets $A'$ and $B$ are $\Delta$-separated sets satisfying $|A'| \leq \Delta^{-\alpha}$ and $|B| \geq \Delta^{-\beta}$, assuming that $\zeta > 0$ was taken sufficiently smaller than $\zeta_{0} = \zeta_{0}(s,\sigma_{1},\sigma_{2},t) > 0$. Also, according to \eqref{form72}, the set $B$ satisfies a $\beta$-dimensional non-concentration condition if $6\sqrt{\zeta} < \zeta_{0}$.\footnote{Since the current "$\zeta$" actually stands for $\bar{\zeta} = 7\sqrt{\zeta}$ in the original notation of \eqref{form9}, recall Notation \ref{not1}, it would be more accurate to require here that $C\zeta^{1/4} < \zeta_{0}$.} We already noted above \eqref{eq:parameters} that with this notation,
\begin{displaymath} \gamma = s - \zeta_{0} > \alpha - \beta. \end{displaymath}
Moreover, the set $E_{I}$ is a $(\Delta,s - \zeta,\Delta^{-O(\zeta)})$-set, or in other words the normalised counting measure $\nu = |E_{I}|^{-1}\mathcal{H}^{0}|_{E_{I}}$ satisfies the Frostman condition
\begin{displaymath} \nu(B(x,r)) \leq \Delta^{-O(\zeta)}r^{s - \zeta}\quad\text{for all } x \in S^{1}, \, \Delta \leq r \leq 1. \end{displaymath}
Consequently, if $\zeta > 0$ is small enough depending on $\zeta_{0}$ and $\chi_{0} = \chi_{0}(s,\sigma_{0},\sigma_{1},t) > 0$, namely the constant from \eqref{form85}, then $\nu$ also satisfies
\begin{displaymath} \nu(B(x,r)) \leq r^{s - O(\sqrt{\zeta})} \leq r^{s - \zeta_{0}}\quad\text{for all } x \in S^{1}, \, \Delta \leq r \leq \Delta^{\chi_{0}} \leq \Delta^{\sqrt{\zeta}}. \end{displaymath}
In other words, $\nu$ satisfies the hypothesis of Theorem \ref{thm:ABCConjecture}, in the notation of \eqref{form82}. As a final piece of information, recall that the sets $G_{\theta}' \subset A' \times B$ defined above \eqref{form86} (see also \eqref{form87}) satisfy $|G_{\theta}'| \geq \Delta^{\chi_{0}}|A'||B|$, assuming that $29\zeta < \chi_{0}$.

Therefore, by Theorem \ref{thm:ABCConjecture}, or more precisely \eqref{form82}, if $\delta > 0$ (hence $\Delta > 0$) is sufficiently small, there exists $\theta \in \spt(\nu) = E_{I}$ with the property
\begin{displaymath} |\pi_{\theta}(G_{\theta})|_{\Delta} \geq \Delta^{-\chi_{0}}|A'|. \end{displaymath}
However, this lower bound contradicts \eqref{form89} if $9\zeta < \chi_{0}$. This completes the proof of Proposition \ref{prop2}.

\section{Furstenberg set estimates} \label{s:Furstenberg}

We recommend that the reader reviews the definitions of dyadic tubes and their slopes from Section \ref{s:tube-preliminaries}. We will use the following terminology of \emph{nice configurations}.
\begin{definition}
Fix $\delta\in 2^{-\mathbb{N}}$, $s\in [0,1]$, $C>0$,  $M\in\mathbb{N}$. We say that a pair $(\mathcal{P},\mathcal{T}) \subset \mathcal{D}_{\delta} \times \mathcal{T}^{\delta}$ is a \emph{$(\delta,s,C,M)$-nice configuration} if for every $p\in\mathcal{P}$ there exists a $(\delta,s,C)$-set $\mathcal{T}(p) \subset\mathcal{T}$ with $ |\mathcal{T}(p)| = M$ and such that $T \cap p \neq\emptyset$ for all $T\in\mathcal{T}(p)$.
\end{definition}

Theorem \ref{thm:Furstenberg-Ahlfors} below is the discretised version of Theorem \ref{thm:FurstIntro}, or to be accurate, a "dual version" of Theorem \ref{thm:FurstIntro}. Theorem \ref{thm:FurstIntro} can be deduced from Theorem \ref{thm:Furstenberg-Ahlfors} with the aid of Proposition \ref{prop5}, see Remark \ref{rem3}.
\begin{thm}\label{thm:Furstenberg-Ahlfors} For every $s \in (0,1]$, $t \in [s,2]$, and $0 \leq u < \min\{(s + t)/2,1\}$, there exist $\epsilon = \epsilon(s,t,u) > 0$ and $\delta_{0} := \delta_{0}(s,t,u) > 0$ such that the following holds for all $\delta \in (0,\delta_{0}]$. Let $(\mathcal{P},\mathcal{T})$ be a $(\delta,s,\delta^{-\e},M)$-nice configuration, where $\mathcal{P}$ is a  $(\delta,t,\delta^{-\e})$-regular set. Then,
\begin{equation}\label{eq:FurstenbergSTU}
\Big| \bigcup_{p \in \mathcal{P}} \mathcal{T}(p) \Big| \ge M\cdot \delta^{-u}.
\end{equation}
\end{thm}

\subsubsection*{Choosing parameters} We start the proof by fixing some parameters. Let
\begin{equation}\label{form146}
v = \tfrac{1}{2}\big(u+ \min\{\tfrac{s+t}{2},1\}\big) \in [0,1). \end{equation}
We apply Corollary \ref{cor2} (and Remark \ref{rem:cor2}) with $v$ in place of $u$; we let $\epsilon,\delta_{0} > 0$ be so small that the conclusion of Corollary \ref{cor2} holds with constants $10\epsilon$ and $\delta_{0}$. We may assume that $\epsilon > 0$ here is so small that
\begin{equation}\label{form90} v - 6\epsilon > u. \end{equation}

In this proof, the notation $A\lessapprox_{\delta} B$ stands for $A \le C\, \log(1/\delta)^C\, B$, where $C$ is a constant that may depend (only) on $s,t,u$. We define $A\approx_{\delta} B$, $A\gtrapprox_{\delta} B$ similarly.

We choose $\delta'_0,\e' \in (0,\tfrac{1}{2}]$ so that
\begin{align}
\delta'_0 &\le (\delta_0)^{\e^{-3}}, \label{eq:delta'_0} \\
\e' &\le \tfrac{1}{10}\e^{4}.  \label{eq:eps_0}
\end{align}
Then $\delta'_0$, $\e'$ depend only on $s,t,u$. In particular, the constants in the "$\approx_{\delta}$" notation may depend on $\epsilon'$. In fact, there will be several further (implicit) requirements for the smallness of $\delta_{0}'$, but we will not attempt to keep track of them explicitly; when they occur, we will use the phrase "assuming that $\delta$ is small enough". We now claim that \eqref{eq:FurstenbergSTU} holds whenever $(\mathcal{P},\mathcal{T})$ is a $(\delta,s,\delta^{-\epsilon'},M)$-nice configuration, where $\mathcal{P} \subset \mathcal{D}_{\delta}$ is $(\delta,t,\delta^{-\epsilon'})$-regular.

For notational purposes, let $(\mathcal{P}_0,\mathcal{T}_0)$ be the given $(\delta,s,\delta^{-\e'},M_0)$-nice configuration, where $M_{0} := M$, $0<\delta\le \delta'_0$, and $\mathcal{P}_0$ is $(\delta,t,\delta^{-\e'})$-regular. Write
\begin{align*}
\Delta &= \delta^{\e^3} \stackrel{\eqref{eq:delta'_0}}{\leq} \delta_{0},\\
\Delta_j &= \Delta^{-j}\cdot\delta, \quad 0\le j \le \e^{-3}.
\end{align*}
Notice that
\begin{equation}\label{form134} \Delta_{j} \leq \Delta\quad\text{for all } 0 \leq j \leq \epsilon^{-3} - 1 =: J. \end{equation}

\subsubsection*{A recursive construction} We will inductively construct a sequence $(\mathcal{P}_j,\mathcal{T}_j)_{j=0}^J$ such that:
\begin{itemize}
\item $\mathcal{P}_j$ is a $(\Delta_j,s,\lessapprox_{\delta}\delta^{-\e'},M_j)$-nice configuration and, further, 
\item $\mathcal{P}_j$ is $(\Delta_j,t,\lessapprox_{\delta}\delta^{-\e'})$-regular.
\end{itemize}   
The configuration $(\mathcal{P}_0,\mathcal{T}_0)$ is given. There is no loss of generality in assuming that the families $\mathcal{T}_{0}(p)$, $p \in \mathcal{P}_{0}$ have the following uniformity property at scale $\Delta$: there exists an integer $H_{0} \in \{1,\ldots,\sim \Delta/\delta\}$ such that
\begin{equation}\label{form142} H_{0} = |\mathcal{T}_{0}(p) \cap \mathbf{T}|\quad\text{for all } p \in \mathcal{P}_{0}, \, \mathbf{T} \in \mathcal{T}^{\Delta}(\mathcal{T}_{0}(p)). \end{equation}
Indeed, this property can be obtained by pigeonholing, at the cost of reducing the cardinalities of $\mathcal{P}_{0},\mathcal{T}_{0}(p)$ by a factor of $\approx_{\delta} 1$.

Suppose the  $(\Delta_j,s,\lessapprox_{\delta}\delta^{-\e'},M_j)$-nice configuration $(\mathcal{P}_j,\mathcal{T}_j)$ has been defined for some $j \leq J - 1$, and $\mathcal{P}_j$ is $(\Delta_j,t,\lessapprox_{\delta}\delta^{-\e'})$-regular. We begin the construction of $(\mathcal{P}_{j + 1},\mathcal{T}_{j + 1})$ by selecting a subset $\overline{\mathcal{P}}_{j} \subset \mathcal{P}_{j}$ with the following properties:
\begin{enumerate}[(a)]
\item \label{it-a} $|\overline{\mathcal{P}}_{j}| \approx_{\Delta_{j}} |\mathcal{P}_{j}|$,
\item \label{it-b} $|\overline{\mathcal{P}}_{j} \cap \mathbf{q}| \equiv A_{j}$ for some constant $A_{j}\in \N$, and for all $\mathbf{q} \in \mathcal{D}_{\Delta_{j + 1}}(\overline{\mathcal{P}}_{j})$.
\end{enumerate}
Such a subset can be located by elementary pigeonholing. We then apply \cite[Proposition 5.2]{OS23} to $(\overline{\mathcal{P}}_{j},\mathcal{T}_{j})$, at scales $\Delta_{j}$ and $\Delta_{j + 1}$. (Notice that both scales are smaller than $\Delta$ by \eqref{form134}.) The argument in \cite{OS23} is a careful inductive pigeonholing. We find:
\begin{itemize}
\item A "refinement" configuration $(\widehat{\mathcal{P}}_{j},\widehat{\mathcal{T}}_{j})\subset\mathcal{D}_{\Delta_{j}}\times \mathcal{T}^{\Delta_{j}}$ with $\widehat{\mathcal{P}}_{j} \subset \overline{\mathcal{P}}_{j}\subset \mathcal{P}_{j}$, and $\widehat{\mathcal{T}}_{j}(\mathbf{p}) \subset \mathcal{T}_{j}(\mathbf{p})$ for all $\mathbf{p} \in \widehat{\mathcal{P}}_{j}$,
\item a "covering" configuration $(\mathcal{P}_{j + 1},\mathcal{T}_{j + 1}) \subset \mathcal{D}_{\Delta_{j + 1}} \times \mathcal{T}^{\Delta_{j + 1}}$,
\end{itemize}
with the following more precise properties:
\begin{enumerate}[(\rm i)]
\item \label{it-induction-i-body} $|\mathcal{D}_{\Delta_{j+1}}(\widehat{\mathcal{P}}_{j})| \approx_{\delta} |\mathcal{D}_{\Delta_{j+1}}(\overline{\mathcal{P}}_{j})|$, and $|\widehat{\mathcal{P}}_{j}\cap \mathbf{q}| \approx_{\delta} |\overline{\mathcal{P}}_{j} \cap \mathbf{q}|$ for all $\mathbf{q}\in\mathcal{D}_{\Delta_{j+1}}(\widehat{\mathcal{P}}_{j})$. It follows from the properties \eqref{it-a}-\eqref{it-b} of the set $\overline{\mathcal{P}}_{j}$ that
\begin{displaymath} |\widehat{\mathcal{P}}_{j}| \approx_{\delta} |\mathcal{D}_{\Delta_{j + 1}}(\overline{\mathcal{P}}_{j})| \cdot A_{j} = |\overline{\mathcal{P}}_{j}| \approx_{\delta} |\mathcal{P}_{j}|. \end{displaymath}
\item \label{it-induction-ii-body} $|\widehat{\mathcal{T}}_{j}(\mathbf{p})| \approx_{\delta} |\mathcal{T}_{j}(\mathbf{p})|$ for $\mathbf{p}\in\widehat{\mathcal{P}}_{j}$.
\item \label{it-induction-iii-body} The configuration $(\mathcal{P}_{j + 1},\mathcal{T}_{j + 1})$ is $(\Delta_{j + 1},s,\lessapprox_{\delta} \delta^{-\epsilon'},M_{j + 1})$-nice for some $M_{j + 1} \geq 1$, and  $\mathcal{P}_{j + 1} = \mathcal{D}_{\Delta_{j + 1}}(\widehat{\mathcal{P}}_{j})$. Moreover,
\begin{equation}\label{form98} \mathcal{T}^{\Delta_{j + 1}}(\widehat{\mathcal{T}}_{j}(\mathbf{p})) \subset \mathcal{T}_{j + 1}(\mathbf{q})\quad\text{for all } \mathbf{p} \in \widehat{\mathcal{P}}_{j} \cap \mathbf{q}, \, \mathbf{q} \in \mathcal{P}_{j + 1}. \end{equation}

\item \label{it-induction-iv-body} For each $\mathbf{q}\in\mathcal{P}_{j+1}$ there exists $M_\mathbf{q}\ge 1$,  and a family of tubes $\mathcal{T}_{\mathbf{q}}\subset\mathcal{T}^{\Delta}$ such that $(S_{\mathbf{q}}(\widehat{\mathcal{P}}_{j} \cap \mathbf{q}),\mathcal{T}_\mathbf{q}) \subset \mathcal{D}_{\Delta} \times \mathcal{T}^{\Delta}$ is $(\Delta,s,\lessapprox_{\delta} \delta^{-\e'},M_{\mathbf{q}})$-nice, where
\begin{displaymath} S_{\mathbf{q}}(\widehat{\mathcal{P}}_{j} \cap \mathbf{q}) := \{S_{\mathbf{q}}(\mathbf{p}) : \mathbf{p} \in \widehat{\mathcal{P}}_{j} \cap \mathbf{q}\}. \end{displaymath}
 (We remind that by definition $M_{\mathbf{q}} \equiv |\mathcal{T}_{\mathbf{q}}(S_{\mathbf{q}}(\mathbf{p}))|$ is independent of $\mathbf{p} \in \widehat{\mathcal{P}}_{j}\cap\mathbf{q}$.) Moreover,
     \begin{equation} \label{eq:slopes-rescaled-tubes}
    \mathcal{D}_{\Delta} \left[ \sigma\big(\mathcal{T}_\mathbf{q}(S_{\mathbf{q}}(\mathbf{p}))\big) \right]  = \mathcal{D}_{\Delta}\left[\sigma\big(\widehat{\mathcal{T}}_{j}(\mathbf{p})\big)\right]\quad\text{for all } \mathbf{p}\in\widehat{\mathcal{P}}_{j}\cap \mathbf{q}.
    \end{equation}
\item \label{it-induction-v-body}
\begin{equation*} 
\frac{|\mathcal{T}_{j}|}{M_{j}} \gtrapprox_{\delta} \frac{|\mathcal{T}_{j+1}|}{M_{j+1}}\cdot \max_{\mathbf{q}\in \mathcal{P}_{j+1} }\frac{|\mathcal{T}_{\mathbf{q}}|}{M_{\mathbf{q}}}.
\end{equation*}
\end{enumerate}
In fact, above, we apply \cite[Proposition 5.2]{OS23} in the slightly stronger form stated in \cite[Remark 5.5]{OS23}, applied at scale $\bar{\Delta} := \Delta = \delta^{\epsilon^{3}} \geq \Delta_{j + 1}$. This allows us to ensure that the following number is well-defined, that is, depends only on $j \in \{0,\ldots,J - 1\}$:
\begin{equation}\label{form133} H_{j + 1} := |\mathcal{T}_{j + 1}(\mathbf{q}) \cap \mathbf{T}|\quad\text{for all } \mathbf{q} \in \mathcal{P}_{j + 1}, \, \mathbf{T} \in \mathcal{T}^{\Delta}(\mathcal{T}_{j + 1}(\mathbf{q})). \end{equation}
(Note the analogue with \eqref{form142}.) This uniformity has the following useful consequence: whereas \eqref{it-induction-ii-body} says that $|\widehat{\mathcal{T}}_{j}(\mathbf{q})| \approx |\mathcal{T}_{j}(\mathbf{q})|$ for $\mathbf{q} \in \widehat{\mathcal{P}}_{j}$, we may use \eqref{it-induction-i-body} and \eqref{form133} (or \eqref{form142} when $j = 0$) to upgrade this to
\begin{equation}\label{form138} |\widehat{\mathcal{T}}_{j}(\mathbf{q})|_{\Delta} \approx_{\delta} |\mathcal{T}_{j}(\mathbf{q})|_{\Delta}\quad\text{for all } j \in \{0,\ldots,J\}, \, \mathbf{q} \in \widehat{\mathcal{P}}_{j}. \end{equation}
Indeed, this follows from rearranging the following inequality:
\begin{displaymath} H_{j} \cdot |\mathcal{T}_{j}(\mathbf{q})|_{\Delta} = |\mathcal{T}_{j}(\mathbf{q})| \approx_{\delta} |\widehat{\mathcal{T}}_{j}(\mathbf{q})| \leq H_{j} \cdot |\widehat{\mathcal{T}}_{j}(\mathbf{q})|_{\Delta}. \end{displaymath}
Often, without separate remark, we will use the fact that cardinalities of a tube family and its slope sets are comparable if the tubes in the family all intersect a common square. For instance, the numbers in \eqref{form138} could be replaced by $|\sigma(\widehat{\mathcal{T}}_{j}(\mathbf{q}))|_{\Delta}$ and $|\sigma(\mathcal{T}_{j}(\mathbf{q}))|_{\Delta}$ without affecting the conclusion.

We also record at this point that for $0 \leq j \leq J - 1$ and $\mathbf{q} \in \mathcal{P}_{j + 1}$ we have
\begin{equation}\label{form140} M_{\mathbf{q}} \stackrel{\mathrm{def.}}{=} |\mathcal{T}_{\mathbf{q}}(S_{\mathbf{q}}(\mathbf{p}))| \stackrel{\eqref{eq:slopes-rescaled-tubes}}{\sim} |\sigma(\widehat{\mathcal{T}}_{j}(\mathbf{p}))|_{\Delta} \stackrel{\eqref{form138}}{\approx_{\delta}} |\mathcal{T}_{j}(\mathbf{p})|_{\Delta}\quad\text{for all } \mathbf{p} \in \widehat{\mathcal{P}}_{j} \cap \mathbf{q}.  \end{equation}

\subsubsection*{Regularity estimates}

\begin{lemma}
For $0 \leq j \leq J$, the family $\mathcal{P}_{j}$ is $(\Delta_{j},t,\lessapprox_{\delta}\delta^{-\e'})$-regular.
\end{lemma}
\begin{proof}
We argue by induction. If $j=0$ this holds by assumption. Assume true for $j<J$ and consider $j+1$. Since $\mathcal{P}_{j+1} \subset \mathcal{D}_{\Delta_{j+1}}(\mathcal{P}_{0})$, the scale-invariant upper bound in Definition \ref{def:deltaTRegularSet}(2) is simply inherited from $\mathcal{P}_{0}$. We therefore only need to check that $\mathcal{P}_{j+1}$ is a $(\Delta_{j+1},t,\lessapprox_{\delta} \delta^{-\epsilon'})$-set. By property \eqref{it-a} of $\overline{\mathcal{P}}_{j}$ and the induction assumption, $\overline{\mathcal{P}}_{j}$ is a $(\Delta_{j},t,\lessapprox_{\delta} \delta^{-\epsilon'})$-set. Using property \eqref{it-b} of $\overline{\mathcal{P}}_{j}$ and Claim \eqref{it-induction-i-body}, we deduce from Corollary \ref{cor1} that $\widehat{\mathcal{P}}_j$ is a $(\Delta_{j+1},t,\lessapprox_{\delta} \delta^{-\epsilon'})$-set. Since $\mathcal{P}_{j+1}=\mathcal{D}_{\Delta_{j+1}}(\widehat{\mathcal{P}}_j)$ by definition, so is $\mathcal{P}_{j+1}$, as claimed.
\end{proof}

As a corollary, we next check that also the "blow-ups" $S_{\mathbf{q}}(\widehat{\mathcal{P}}_{j} \cap \mathbf{q})$ introduced in part \eqref{it-induction-iv-body} are $t$-regular:

\begin{cor}\label{cor5} Let $0 \leq j \leq J - 1$ and $\mathbf{q} \in \mathcal{P}_{j + 1}$. Then, $S_{\mathbf{q}}(\widehat{\mathcal{P}}_{j} \cap \mathbf{q})$ is $(\Delta,t,\lessapprox_{\delta} \delta^{-4\epsilon'})$-regular. In particular, $S_{\mathbf{q}}(\widehat{\mathcal{P}}_{j} \cap \mathbf{q})$ is $(\Delta,t,\Delta^{-\epsilon})$-regular, assuming that $\delta > 0$ is small enough. \end{cor}

\begin{proof} In Remark \ref{rem:regular}, we generally observed that if $\mathcal{P} \subset \mathcal{D}_{\delta_{1}}$ is $(\delta_{1},t,C)$-regular, and $\mathbf{q} \in \mathcal{D}_{\delta_{2}}(\mathcal{P})$ with $\delta_{1} \leq \delta_{2} \leq 1$, then the renormalisation $S_{\mathbf{q}}(\mathcal{P} \cap \mathbf{q})$ is $(\delta_{1}/\delta_{2},t,\bar{C})$-regular with constant
\begin{equation}\label{form149} \bar{C} = C \cdot \frac{\delta_{2}^{t}|\mathcal{P}|}{|\mathcal{P} \cap \mathbf{q}|}. \end{equation}
We apply this to the set $\widehat{\mathcal{P}}_{j}$, at scales $\{\delta_{1},\delta_{2}\} = \{\Delta_{j},\Delta_{j + 1}\}$, and to the given square $\mathbf{q} \in \mathcal{P}_{j + 1} = \mathcal{D}_{\Delta_{j + 1}}(\widehat{\mathcal{P}}_{j})$. It follows from the previous lemma and Property \eqref{it-induction-i-body} that $\widehat{\mathcal{P}}_{j}$ is $(\Delta_{j},t,\lessapprox_{\delta} \delta^{-\epsilon'})$-regular. Therefore, the main task is to find a lower bound for $|\widehat{\mathcal{P}}_{j} \cap \mathbf{q}|$. First of all, since $\widehat{\mathcal{P}}_{j} \subset \mathcal{P}_{\Delta_{j}}(\mathcal{P}_{0})$, we have
\begin{equation}\label{form148} |\widehat{\mathcal{P}}_{j}|_{\Delta_{j}} \leq \delta^{-\epsilon'}\Delta_{j}^{-t} \quad \text{and} \quad |\widehat{\mathcal{P}}_{j}|_{\Delta_{j + 1}} \leq \delta^{-\epsilon'}\Delta_{j +1}^{-t} \end{equation}
by the $(\delta,t,\delta^{-\epsilon'})$-regularity of $\mathcal{P}_{0}$. Second, again by \eqref{it-induction-i-body}, we have
\begin{displaymath} |\widehat{\mathcal{P}}_{j} \cap \mathbf{q}'| \approx_{\delta} |\overline{\mathcal{P}}_{j} \cap \mathbf{q}'| = A_{j}\quad\text{for all } \mathbf{q}' \in \mathcal{D}_{\Delta_{j + 1}}(\widehat{\mathcal{P}}_{j}). \end{displaymath}
Finally, from \eqref{form148} and the $(\Delta_{j},t,\lessapprox_{\delta} \delta^{-\epsilon'})$-set property of $\widehat{\mathcal{P}}_{j}$, we deduce
\begin{displaymath} \delta^{\epsilon'}\Delta_{j}^{-t} \lessapprox_{\delta} |\widehat{\mathcal{P}}_{j}|_{\Delta_{j + 1}} \cdot A_{j} \leq \delta^{-\epsilon'}\Delta_{j + 1}^{-t} \cdot A_{j},  \end{displaymath}
and rearranging this gives
\begin{displaymath}
|\widehat{\mathcal{P}}_{j} \cap \mathbf{q}| \approx_{\delta} A_{j} \gtrapprox_{\delta} \delta^{2\epsilon'}\Delta^{-t}.
\end{displaymath}
According to \eqref{form149}, we may now conclude that $S_{\mathbf{q}}(\widehat{\mathcal{P}}_{j} \cap \mathbf{q})$ is $(\Delta,t,\bar{C})$-regular with constant
\begin{displaymath}
\bar{C} \approx_{\delta} \delta^{-\epsilon'} \frac{\Delta_{j + 1}^{t}|\widehat{\mathcal{P}}_{j}|}{|\widehat{\mathcal{P}}_{j} \cap \mathbf{q}|} \stackrel{\eqref{form148}}{\lessapprox_{\delta}} \delta^{-\epsilon'} \frac{\Delta_{j + 1}^{t} \cdot \delta^{-\epsilon'} \cdot \Delta_{j}^{-t}}{\delta^{2\epsilon'}\Delta^{-t}} = \delta^{-4\epsilon'}.
\end{displaymath}

Finally, since we assumed in \eqref{eq:eps_0} that $\epsilon' \leq \tfrac{1}{10}\epsilon^4$, this shows that $\bar{C} \leq \delta^{-\epsilon^4} =\Delta^{-\epsilon}$, assuming that $\delta > 0$ is small enough.
\end{proof}

\subsubsection*{Reduction to the quasi-product case} 

Our next goal is to reduce the proof of \eqref{eq:FurstenbergSTU} to the simpler estimate \eqref{eq:claim-to-finish-proof} below.

Let $\mathbf{q}_j\in \mathcal{D}_{\Delta_j}$, $0 \leq j \leq J$, be a nested sequence of squares such that
\begin{equation}\label{form99} \mathbf{q}_{j}\in \widehat{\mathcal{P}}_{j}\quad\text{for all } 0 \leq j \leq J.  \end{equation}
(This can be achieved by property \eqref{it-induction-iii-body}.) For $1 \leq j \leq J$, abbreviate $N_j = M_{\mathbf{q}_j}$ (the number introduced in \eqref{it-induction-iv-body}). For future reference, we record the following consequence of \eqref{form140} applied to $\mathbf{p} = \mathbf{q}_{j - 1}$:
\begin{equation}\label{form141} N_{j} = M_{\mathbf{q}_{j}} \stackrel{\eqref{form140}}{\approx_{\delta}} |\mathcal{T}_{j - 1}(\mathbf{q}_{j - 1})|_{\Delta}\quad\text{for all } 1 \leq j \leq J. \end{equation}
Iterating property \eqref{it-induction-v-body} (and using the trivial estimate $|\mathcal{T}_{J}| \geq M_{J}$ at the last step of the iteration), we get
\begin{equation} \label{eq:product-incidences}
\frac{|\mathcal{T}_{0}|}{M_{0}} \gtrapprox_{\delta} \prod_{j = 1}^{J} \frac{|\mathcal{T}_{\mathbf{q}_{j}}|}{N_{j}}.
\end{equation}
From the inclusion \eqref{form98}, and since $\Delta_{j + 1} \leq \Delta$ for all $0 \leq j \leq J - 1$, we may deduce that
\begin{equation*}
  |\widehat{\mathcal{T}}_{j}(\mathbf{q}_j)|_{\Delta} \le |\mathcal{T}_{j + 1}(\mathbf{q}_{j+1})|_{\Delta}\quad\text{for all } 0 \leq j \leq J - 1. 
\end{equation*}
Consequently, for $1 \leq j \leq J - 1$, we have
\begin{equation}\label{eq-Njincreasing} N_{j} \stackrel{\eqref{form141}}{\approx_{\delta}} |\mathcal{T}_{j - 1}(\mathbf{q}_{j - 1})|_{\Delta} \stackrel{\eqref{form138}}{\approx_{\delta}} |\widehat{\mathcal{T}}_{j - 1}(\mathbf{q}_{j - 1})|_{\Delta} \leq |\mathcal{T}_{j}(\mathbf{q}_{j})|_{\Delta} \stackrel{\eqref{form141}}{\approx_{\delta}} N_{j + 1}. \end{equation}
Thus, the sequence $\{N_{j}\}_{j = 1}^{J}$ is increasing, at least in the sense $N_{j} \lessapprox_{\delta} N_{j + 1}$. In particular, if $\delta$ chosen small enough, it follows (as we will next check) that the "exceptional set" $E = \big\{ j\in \{1,\ldots,J - 1\}: N_{j+1}\ge \Delta^{-\e} N_j \big\}$ has 
\begin{equation}\label{eq:small-exceptional-indices}
|E| \le 2\e^{-1}. \end{equation}
To see this, recall that $J \leq \epsilon^{-3}$, $\Delta = \delta^{\epsilon^{3}}$, $N_{J} = M_{{\bf q}_{J}} \leq C\Delta^{-1}$ for an absolute constant $C > 0$, and that $N_{j+1} \gtrapprox_\delta N_j$ for every $j \in \{0,\ldots,J - 1\}$ by \eqref{eq-Njincreasing}. Using these facts,
\begin{align*}
C\Delta^{-1} &\geq N_{J} \geq \prod_{j=1}^{J-1} \frac{N_{j+1}}{N_j} = \prod_{j\in E} \frac{N_{j+1}}{N_j} \prod_{j\not\in E}\frac
{N_{j+1}}{N_j} \\
&\geq (\Delta^{-\e})^{|E|}(C_{s} \log (1/\delta)^{-C})^{\epsilon^{-3}}\geq C\Delta^{1 -\e|E|},
\end{align*}
provided $\delta>0$ is small enough in terms of $\epsilon,s$. The bound $|E|\leq 2\e^{-1}$ follows.

We will show that
  \begin{equation}\label{eq:claim-to-finish-proof}
\frac{|\mathcal{T}_{\mathbf{q}_{j}}|}{N_{j}} \geq \Delta^{-v + 2\epsilon}\quad\text{for all } j\in\{1,\ldots,J - 1\} \, \setminus \, E.
\end{equation}
\begin{remark} Why is \eqref{eq:claim-to-finish-proof} easier to prove than our main claim \eqref{eq:FurstenbergSTU} in full generality? Recall from \eqref{it-induction-iv-body} that $\mathcal{T}_{\mathbf{q}_{j}}$ is a family of $\Delta$-tubes with the property that $(S_{\mathbf{q}_{j}}(\widehat{\mathcal{P}}_{j - 1} \cap \mathbf{q}_{j}),\mathcal{T}_{\mathbf{q}_{j}})$ is a $(\Delta,s,\lessapprox_{\delta} \delta^{-\epsilon'},M_{\mathbf{q}_{j}})$-nice configuration, and we have abbreviated $N_{j} := M_{\mathbf{q}_{j}}$. Thus, \eqref{eq:claim-to-finish-proof} looks \emph{a priori} very similar to \eqref{eq:FurstenbergSTU}. 

The extra information we have available is that $N_{j} \approx_{\delta} N_{j + 1}$ for $j \notin E$. This translates to the fact that the slope sets (viewed at scale $\Delta$) of the families $\mathcal{T}_{\mathbf{q}_{j}}(S_{\mathbf{q}_{j}}(\mathbf{p}))$, $\mathbf{p} \in \widehat{\mathcal{P}}_{j - 1} \cap \mathbf{q}_{j}$, are all contained in a single, common, slope set with $\Delta$-covering number comparable to each of the individual numbers $|\mathcal{T}_{\mathbf{q}_{j}}(S_{\mathbf{q}_{j}}(\mathbf{p}))|_{\Delta}$. See \eqref{Theta-p-subset-Theta}-\eqref{form143} for a rigorous justification of these claims.

More informally still, the proof of \eqref{eq:claim-to-finish-proof} corresponds to the special case of \eqref{eq:FurstenbergSTU} where all the slope sets of the families $\mathcal{T}(p)$ (in the notation of Theorem \ref{thm:Furstenberg-Ahlfors}) are "roughly the same". But this special case is equivalent to the projection theorem we have already obtained in Corollary \ref{cor2}. So, in brief, Corollary \ref{cor2} implies \eqref{eq:claim-to-finish-proof}. 

\end{remark}
In light of \eqref{eq:product-incidences} and \eqref{eq:small-exceptional-indices}, and recalling that $\Delta = \delta^{\epsilon^{3}}$, \eqref{eq:claim-to-finish-proof} will show that
\begin{displaymath} \frac{|\mathcal{T}_{0}|}{M_{0}} \gtrapprox_{\delta} (\Delta^{-v + 2\epsilon})^{\epsilon^{-3} - 2 - 2\epsilon^{-1}} \geq \delta^{2(\epsilon + \epsilon^{2} + \epsilon^{3}) -  v} \geq \delta^{6\epsilon - v}. \end{displaymath}
Recalling from \eqref{form90} that $v - 6\epsilon > u$, this will finish the proof of \eqref{eq:FurstenbergSTU} (taking $\delta'_0$ smaller if needed).

\subsubsection*{Proof of \eqref{eq:claim-to-finish-proof}}

Fix $j\in\{1,\ldots,J - 1\} \, \setminus \, E$ for the rest of the proof, so $N_{j+1}\le \Delta^{-\e}N_j$. Abbreviate $\mathbf{q} := \mathbf{q}_{j}$, and $\mathcal{Q} := \widehat{\mathcal{P}}_{j - 1} \cap \mathbf{q}$. Let
\begin{align*}
\Theta&:=\mathcal{D}_{\Delta}\big(\sigma(\mathcal{T}_{j}(\mathbf{q}))\big) ,\\
\Theta_{\mathbf{p}}&:= \mathcal{D}_{\Delta}\left[ \sigma(\mathcal{T}_{\mathbf{q}}(S_{\mathbf{q}}(\mathbf{p})))\right]  \stackrel{\eqref{eq:slopes-rescaled-tubes}}{=} \mathcal{D}_{\Delta}\left[\sigma(\widehat{\mathcal{T}}_{j - 1}(\mathbf{p})) \right] ,\quad \mathbf{p}\in \mathcal{Q}.
\end{align*}
Thus, $\Theta$ and $\Theta_{\mathbf{p}}$ are collections of dyadic intervals, but we identify them with the left end-points of the respective intervals (and so as subsets of $\Delta\cdot \Z$) and view them as sets of slopes. We make two observations. First,
\begin{equation} \label{Theta-p-subset-Theta}
\Theta_{\mathbf{p}} = \mathcal{D}_{\Delta}\left[ \sigma(\widehat{\mathcal{T}}_{j - 1}(\mathbf{p})) \right]  \stackrel{\eqref{form98}}{\subset} \Theta\quad\text{for all } \mathbf{p} \in \mathcal{Q}.
\end{equation}
Second,
\begin{align} |\Theta| & \stackrel{\eqref{form141}}{\approx_{\delta}} N_{j + 1} \leq \Delta^{-\epsilon}N_{j} \stackrel{\eqref{form141}}{\approx_{\delta}} \Delta^{-\epsilon}|\mathcal{T}_{j - 1}(\mathbf{p})|_{\Delta} \notag\\
&\label{form143} \stackrel{\eqref{form138}}{\approx_{\delta}} \Delta^{-\epsilon}|\widehat{\mathcal{T}}_{j - 1}(\mathbf{p})|_{\Delta} \sim \Delta^{-\epsilon}|\Theta_{\mathbf{p}}|\quad\text{for all } \mathbf{p} \in \mathcal{Q}. \end{align}
In other words, $\Theta_{\mathbf{p}}$ is a "dense" subset of $\Theta$ for all $\mathbf{p} \in \mathcal{Q}$. As a consequence, the sets
\begin{displaymath} \mathcal{Q}_{\theta} := \{\mathbf{p} \in \mathcal{Q} : \theta \in \Theta_{\mathbf{p}}\}\quad\text{for all } \theta \in \Theta, \end{displaymath}
are dense subsets of $\mathcal{Q}$ for a dense subset of values $\theta \in \Theta$, see Figure \ref{fig1}. More precisely, noting that $\sum_{\theta \in \Theta} |\mathcal{Q}_{\theta}| = \sum_{\mathbf{p} \in \mathcal{Q}} |\Theta_{\mathbf{p}}|$, and applying \eqref{Theta-p-subset-Theta}--\eqref{form143}, there exists a subset $\Theta' \subset \Theta$ of cardinality $|\Theta'| \gtrapprox_{\delta} \Delta^{\epsilon}|\Theta|$ such that $|\mathcal{Q}_{\theta}| \gtrapprox_{\delta} \Delta^{\epsilon}|\mathcal{Q}|$ for all $\theta \in \Theta'$. In particular,
\begin{equation}\label{form145}
|\mathcal{Q}_{\theta}| \geq \Delta^{2\epsilon}|\mathcal{Q}|\quad\text{for all } \theta \in \Theta',
\end{equation}
assuming that $\delta$ is small enough.
\begin{figure}[h!]
\begin{center}
\begin{overpic}[scale = 0.9]{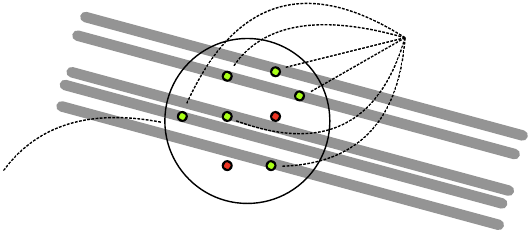}
\put(78,37){$\mathbf{p} \in \mathcal{Q}_{\theta}$}
\put(-1,9){$\mathbf{q}$}
\end{overpic}
\caption{For each $\mathbf{p} \in \mathcal{Q}_{\theta}$, there exists a $\Delta_{j - 1}$-tube in the collection $\widehat{\mathcal{T}}_{j - 1}(\mathbf{p})$ pointing in the direction $\theta$. For $\theta \in \Theta'$, the set $\mathcal{Q}_{\theta}$ (the small green discs) is a dense subset of $\mathcal{Q}$ (all the small discs).}\label{fig1}
\end{center}
\end{figure}

The next key observation is that $\Theta$ is a $(\Delta,s,\lessapprox_{\delta} \delta^{-\epsilon'})$-set. We already know by \eqref{it-induction-iii-body} that $\mathcal{T}_{j}(\mathbf{q})$ is a $(\Delta_{j},s,\lessapprox_{\delta} \delta^{-\epsilon'})$-set, and the same is true for the slope set $\sigma(\mathcal{T}_{j}(\mathbf{q}))$ by Remark \ref{rem:slopes}. On the other hand, recall from \eqref{form133} that
\begin{displaymath} |\mathcal{T}_{j}(\mathbf{q}) \cap \mathbf{T}| = H_{j}\quad\text{for all } \mathbf{T} \in \mathcal{T}^{\Delta}(\mathcal{T}_{j}(\mathbf{q})). \end{displaymath}
This implies by Lemma \ref{lemma4} that $\Theta$ is a $(\Delta,s,\lessapprox_{\delta} \delta^{-\epsilon'})$-set, as desired. Therefore, $\Theta'$ is a $(\Delta,s,\lessapprox_{\delta} \Delta^{-\epsilon}\delta^{-\epsilon'})$-set. Since $\epsilon'  \leq \epsilon^{4}$ according to \eqref{eq:eps_0}, and $\delta = \Delta^{\epsilon^{-3}}$, we see that $\Delta^{-\epsilon}\delta^{-\epsilon'} \leq \Delta^{-2\epsilon}$. Thus, assuming $\delta$ small enough, we conclude that $\Theta'$ is a $(\Delta,s,\Delta^{-3\epsilon})$-set.

We plan to apply Corollary \ref{cor2} to the rescaled set $\overline{\mathcal{Q}} := S_{\mathbf{q}}(\mathcal{Q}) \subset \mathcal{D}_{\Delta}$ and its dense subsets $\overline{\mathcal{Q}}_{\theta} := S_{\mathbf{q}}(\mathcal{Q}_{\theta}) \subset \mathcal{D}_{\Delta}$, for $\theta \in \Theta'$. We restate the corollary (combining the statement with Remark \ref{rem:cor2}):
\begin{cor}\label{cor6} Let $s \in (0,1]$ and $t \in [s,2]$. For every $0 \leq u < \min\{(s + t)/2,1\}$, there exist $\epsilon,\Delta_{0} > 0$ such that the following holds for all $\Delta \in (0,\Delta_{0}]$. Let $\overline{\mathcal{Q}} \subset \mathcal{D}_{\Delta}$ be a $(\Delta,t,\Delta^{-\epsilon})$-regular set, and let $E \subset S^{1}$ be a $(\Delta,s,\Delta^{-\epsilon})$-set. Then, there exists a subset $E' \subset E$ with $|E'|_{\Delta} \geq \tfrac{1}{2}|E|_{\Delta}$ such that for all $\theta \in E'$,
\begin{equation}\label{form144} |\pi_{\theta}(\overline{\mathcal{Q}}')|_{\Delta} \geq \Delta^{-u}\quad\text{for all } \overline{\mathcal{Q}}' \subset \overline{\mathcal{Q}}, \, |\overline{\mathcal{Q}}'| \geq \Delta^{\epsilon}|\overline{\mathcal{Q}}|. \end{equation}
\end{cor}
Recall from around \eqref{form146} that we started the current proof by choosing $\epsilon > 0$ so small that the conclusion of this corollary holds with parameter $v = \tfrac{1}{2}(u + \min\{[s + t]/2,1\})$ in place of $u$, and constant "$10\epsilon$".

The slopes $\Theta' \subset \Delta \cdot \Z$ play the role of the set $E$, and we are keen to apply \eqref{form144} to the subsets $\mathcal{Q}_{\theta}$, $\theta \in \Theta' \subset \Delta \cdot \Z$. For $\theta \in \Theta'$, let $\pi_{\theta} \colon \R^{2} \to \R$ (here) denote the orthogonal projection along the slope $\theta$. Formally $\pi_{\theta}(x,y) := -\theta x + y$, but the following informal description is more useful: if $T \in \mathcal{T}^{\Delta}$ with $\sigma(T) = \theta$, then $\pi_{\theta}(T \cap B(1))$ is contained in an interval of length $\sim \Delta$.

Now, recall from Corollary \ref{cor5} that $\overline{\mathcal{Q}} = S_{\mathbf{q}}(\mathcal{Q})$ is $(\Delta,t,\Delta^{-\epsilon})$-regular. Since moreover \eqref{form145} holds, and $\Theta'$ is a $(\Delta,s,\Delta^{-3\epsilon})$-set we may deduce from Corollary \ref{cor6} that there exists a further subset $\Theta'' \subset \Theta'$ of cardinality $|\Theta''| \geq \tfrac{1}{2}|\Theta'|$ such that
\begin{equation}\label{form147} |\pi_{\theta}(\overline{\mathcal{Q}}_{\theta})|_{\Delta} \geq \Delta^{-v}\quad\text{for all } \theta \in \Theta''. \end{equation}
To complete the proof, let us deduce \eqref{eq:claim-to-finish-proof} from \eqref{form147}. First, notice that
\begin{displaymath} |\Theta''| \sim |\Theta'| \gtrapprox_{\delta} \Delta^{\epsilon}|\Theta| \stackrel{\eqref{form138}}{\approx_{\delta}} \Delta^{\epsilon} N_{j + 1} \gtrapprox_{\delta} \Delta^{\epsilon} N_{j}. \end{displaymath}
Therefore, in order to prove \eqref{eq:claim-to-finish-proof}, it suffices to show that $|\mathcal{T}_{\mathbf{q}}^{\theta}| \gtrsim \Delta^{-v}$ for all $\theta \in \Theta''$, where $\mathcal{T}_{\mathbf{q}}^{\theta} := \{T \in \mathcal{T}_{\mathbf{q}} : \sigma(T) = \theta\}$. This is virtually a restatement of \eqref{form147}, but let us unwrap the notation carefully.

Fix $\theta \in \Theta''$ and $\bar{\mathbf{p}} = S_{\mathbf{q}}(\mathbf{p}) \in \overline{\mathcal{Q}}_{\theta}$, where $\mathbf{p} \in \mathcal{Q}_{\theta}$. This means by definitions that
\begin{displaymath} \theta \in \Theta_{\mathbf{p}} = \mathcal{D}_{\Delta}(\sigma(\mathcal{T}_{\mathbf{q}}(\bar{\mathbf{p}}))) . \end{displaymath}
In other words, the family
\begin{displaymath} \mathcal{T}_{\mathbf{q}}^{\theta}(\bar{\mathbf{p}}) := \{T \in \mathcal{T}_{\mathbf{q}}(\bar{\mathbf{p}}) : \sigma(T) = \theta\} \subset \mathcal{T}_{\mathbf{q}}^{\theta} \end{displaymath}
is non-empty for every $\bar{\mathbf{p}} \in \overline{\mathcal{Q}}_{\theta}$. Additionally, observe the following: if $\bar{\mathbf{p}}_{1},\bar{\mathbf{p}}_{2} \in \overline{\mathcal{Q}}$, and $\pi_{\theta}(\bar{\mathbf{p}}_{1}),\pi_{\theta}(\bar{\mathbf{p}}_{2})$ are $C\Delta$-separated for a sufficiently large absolute constant $C > 0$, then
\begin{displaymath} \mathcal{T}_{\mathbf{q}}^{\theta}(\bar{\mathbf{p}}_{1}) \cap \mathcal{T}_{\mathbf{q}}^{\theta}(\bar{\mathbf{p}}_{2}) = \emptyset. \end{displaymath}
Finally, deduce from \eqref{form147} that $\overline{\mathcal{Q}}_{\theta}$ contains a subset $\{\bar{\mathbf{p}}_{1},\ldots,\bar{\mathbf{p}}_{k}\}$ with $k \gtrsim \Delta^{-v}$ whose $\pi_{\theta}$-projections are $C\Delta$-separated. Consequently, $|\mathcal{T}^{\theta}_{\mathbf{q}}| \geq k \gtrsim \Delta^{-v}$, as claimed. The proof of Theorem \ref{thm:Furstenberg-Ahlfors} is complete.

\section{From \texorpdfstring{$\delta$}{delta}-discretised to continuous results}\label{s:continuousResults}

We have already established the $\delta$-discretised analogues of  Theorems  \ref{thm:projection} and \ref{thm:FurstIntro}. In this section, we discuss the "continuous" versions involving Hausdorff and packing dimension. The next lemma will be needed. It roughly says that uniform $(\delta,t)$-sets with box dimension $\leq t$ are automatically $(\delta,t)$-regular.

\begin{lemma}\label{lemma9} Let $\delta = 2^{-mT}$ for some $m,T \in \N$, and let $\mathcal{P} \subset \mathcal{D}_{\delta}$ be a $\{2^{-jT}\}_{j = 1}^{m}$-uniform set satisfying
\begin{equation}\label{form118} c\Delta^{-t} \leq |\mathcal{P}|_{\Delta} \leq C\Delta^{-t}\quad\text{for all } \delta \leq \Delta \leq \Delta_{0}, \end{equation}
where $t \in [0,d]$, $\epsilon > 0$, and $\Delta_{0} \in (\delta,1]$. Then, for all dyadic $\delta \leq r \leq R \leq 1$, we have
\begin{displaymath} |\mathcal{P} \cap Q|_{r} \lesssim_{d,T,\Delta_{0}} (C/c)(R/r)^{t}\quad\text{for all } Q \in \mathcal{D}_{R}. \end{displaymath}
In particular, if $\mathcal{P}$ is a $\{2^{-jT}\}_{j = 1}^{m}$-uniform $(\delta,t,\delta^{-\epsilon})$-set satisfying  \eqref{form118} with $C = \delta^{-\epsilon}$ and $c=\delta^{\epsilon}$, then $\mathcal{P}$ is automatically $(\delta,t,O_{d,T,\Delta_{0}}(1)\delta^{-2\epsilon})$-regular.
\end{lemma}

\begin{proof} First, let $\delta \leq r \leq R \leq \Delta_{0}$ be elements of the sequence $\{2^{-jT}\}_{j = 0}^{m}$. Let $Q \in \mathcal{D}_{R}$. Then, by the uniformity hypothesis, and \eqref{form118}, we have
\begin{equation}\label{form119} |\mathcal{P} \cap Q|_{r} = \frac{|\mathcal{P}|_{r}}{|\mathcal{P}|_{R}} \leq \frac{Cr^{-t}}{cR^{-t}} \leq (C/c)(R/r)^{t}.
\end{equation}
Next, if $\delta \leq r \leq R \leq \Delta_{0}$ are arbitrary dyadic numbers, we may replace $r,R$ by the closest elements in the sequence $\{2^{-jT}\}_{j = 1}^{m}$ to deduce that $|\mathcal{P} \cap Q|_{r} \leq O_{d,T}(1)(C/c)(R/r)^{t}$ for every $Q \in \mathcal{D}_{R}$.

If $r \leq \Delta_{0} < R$ and $Q \in \mathcal{D}_{R}$, we first decompose $Q$ into $\leq \Delta_{0}^{-d}$ cubes of side-length $\Delta_{0}$, and then use the previous case to obtain \eqref{form119} with constant $O_{d,T,\Delta_{0}}(1)$. Finally, if $\Delta_{0} \leq r \leq R \leq 1$, then \eqref{form119} trivially holds with constant $\Delta_{0}^{-d}$. \end{proof}

\subsection{Proof of Theorem \ref{thm:FurstIntro}} In the following, we will say that a line set $\mathcal{L}$ is a $(\delta,t,C)$-set, or a $(\delta,t,C)$-regular set, if $\mathcal{L} = \mathbf{D}(P)$, where $P$ is a $(\delta,t,C)$-set, or a $(\delta,t,C)$-regular set, respectively. (Here $\mathbf{D}$ is the point-line duality map from Definition \ref{def:dyadicTubes}.)

\begin{proposition}\label{prop5} Let $0 < s \leq 1$ and $t \in (0,2]$, and assume that there exist $\delta_{0},\epsilon,\chi > 0$ such that the following holds for all $\delta \in (0,\delta_{0}]$.

Let $\mathcal{L}$ be a $(\delta,t,\delta^{-\epsilon})$-regular line set. Let $\mathcal{P} \subset \mathcal{D}_{\delta}$ be a set with the property that $\mathcal{P} \cap \ell$ contains a non-empty $(\delta,s,\delta^{-\epsilon})$-set for every $T \in \mathcal{T}$. Then, $|\mathcal{P}| \geq \delta^{-\chi}$.

Then, the following also holds. Let $F \subset \R^{2}$ be an $(s,t)$-Furstenberg set associated to a line family $\mathcal{L}$ with packing dimension $t$. Then, $\Hd F \geq \chi$. \end{proposition}

\begin{remark}\label{rem3} The $\delta$-discretised hypothesis in Proposition \ref{prop5} is "dual" to the conclusion we obtained in Theorem \ref{thm:Furstenberg-Ahlfors}. More precisely, if $s \in (0,1]$, $t \in [s,2]$, and $0 \leq u < \min\{(s + t)/2,1\}$, and $\epsilon = \epsilon(s,t,u) > 0$ is sufficiently small, then Theorem \ref{thm:Furstenberg-Ahlfors} implies that the hypothesis of Proposition \ref{prop5} holds with $\chi = s + u$. (We omit the straightforward details about "dualising" Theorem \ref{thm:Furstenberg-Ahlfors}, because -- up to numerology -- they have been recorded in the proof of \cite[Theorem 3.2]{OS23}.) Consequently, every $(s,t)$-Furstenberg set associated to a line family $\mathcal{L}$ with packing dimension $t$ has
\begin{displaymath} \Hd F \geq s + u. \end{displaymath}
Letting $u \nearrow \min\{(s + t)/2,1\}$ proves Theorem \ref{thm:FurstIntro}. 
\end{remark}

\begin{proof}[Proof of Proposition \ref{prop5}] 
  We denote upper box (or Minkowski) dimension by $\overline{\dim}_{\mathrm{B}}$. Thus, for bounded sets $A \subset \R^{d}$, as well as bounded line families $\mathcal{L}$, we write
  \begin{displaymath} \overline{\dim}_{\mathrm{B}} A := \limsup_{r \to 0} \frac{\log N(A,r)}{\log (1/r)}, \end{displaymath}
  where $N(A,r)$ is the least number of closed balls of radius $r$ needed to cover $A$. We also recall from \cite[Section 5.9]{Mattila95} that packing dimension may be defined via upper box dimension as
\begin{displaymath} 
  \dim_{\mathrm{P}} \mathcal{L} = \inf_{\{\mathcal{L}_{j}\}} \left\{\sup_{j} \overline{\dim}_{\mathrm{B}} \mathcal{L}_{j} : \mathcal{L} = \bigcup \mathcal{L}_{j}\right\}.
\end{displaymath}
Here the "$\inf$" runs over countable families $\{\mathcal{L}_{j}\}$ of bounded subsets of $\mathcal{L}$ whose union agrees with $\mathcal{L}$. Since, by our assumption, $t = \Hd \mathcal{L}=\dim_{\mathrm{P}} \mathcal{L}$, for any $\underline{t} < t < \bar{t}$ there exists a bounded subset $\mathcal{L}_{j} \subset \mathcal{L}$ with
\begin{displaymath}
   \underline{t} < \Hd \mathcal{L}_{j} \leq \overline{\dim}_{B} \mathcal{L}_{j} < \bar{t}. 
\end{displaymath}
We choose $\bar{t} - \underline{t} \leq \epsilon/8$, where "$\epsilon$" is the parameter from our hypothesis. We replace $\mathcal{L}$ by $\mathcal{L}_{j}$ without changing notation (precisely: we will only use the information $\Hd (F \cap \ell) \geq s$ for $\ell \in \mathcal{L}_{j}$ in the sequel). Let
\[
  \mathcal{H}^{\underline{t}}_{\infty}(\mathcal{L}) = \inf\left\{\sum_{i} \diam(U_{i})^{\underline{t}} : \mathcal{L} \subset \bigcup_{i} U_{i}\right\}
\]
denote Hausdorff content at dimension $\underline{t}$. Since  Then, for some $c,\Delta_{0} > 0$, we have
\begin{equation}\label{form112} \mathcal{H}^{\underline{t}}_{\infty}(\mathcal{L}) \geq c > 0 \quad \text{and} \quad |\mathcal{L}|_{\Delta} \leq \Delta^{-\bar{t}} \text{ for } \Delta \in (0,\Delta_{0}]. \end{equation}
Next, choose $s - \epsilon/2 < \underline{s} < s$. Then, up to replacing $\mathcal{L}$ by a further subset, and taking "$c$" smaller if needed, we may assume that
\begin{equation}\label{form113} \mathcal{H}^{\underline{s}}_{\infty}(F \cap \ell) \geq c > 0\quad\text{for all } \ell \in \mathcal{L}. \end{equation}
Let $\mathcal{F} \subset \mathcal{D}$ be an arbitrary cover of $F$ by dyadic squares of side-length $\leq \Delta_{1} \ll \Delta_{0}$, to be determined later. We will need to require that $\Delta_{1}$ is small in terms of $c,\epsilon,\Delta_{0}$. Whenever such requirements are met, we will show that
\begin{equation}\label{form114a} \sum_{Q \in \mathcal{F}} \ell(Q)^{\chi} \geq 1. \end{equation}
This will prove that $\Hd F \geq \chi$, as claimed.

We claim that there exists $\delta \in 2^{-\N} \cap (0,\Delta_{1}]$ with the following property. Let $\mathcal{F}_{\delta} := \{Q \in \mathcal{F} : \ell(Q) = \delta\}$. There exists a subset $\mathcal{L}_{\delta} \subset \mathcal{L}$ with $\mathcal{H}^{\underline{t}}_{\infty}(\mathcal{L}_{\delta}) \geq \delta^{\epsilon/4}$ such that
\begin{equation}\label{form114b} 
  \mathcal{H}^{\underline{s}}_{\infty}(\cup \mathcal{F}_{\delta} \cap \ell) \geq \delta^{\epsilon/4}\quad\text{for all } \ell \in \mathcal{L}_{\delta}. 
\end{equation}
To see this, write $\mathcal{F}_{2^{-j}} := \{Q \in \mathcal{F} : \ell(Q) = 2^{-j}\}$, for $2^{-j} \leq \Delta_{1}$, and $\mathcal{L}_{2^{-j}} := \{\ell \in \mathcal{L} : \mathcal{H}_{\infty}^{\underline{s}}(\cup \mathcal{F}_{2^{-j}} \cap \ell) \geq 2^{-\epsilon j/4}\}$. Note that if $\ell$ is a line which is not contained in $\bigcup_{j} \mathcal{L}_{2^{-j}}$, then by the sub-additivity of Hausdorff content,
\begin{displaymath} \mathcal{H}^{\underline{s}}_{\infty}(F \cap \ell) \leq \sum_{2^{-j} \leq \Delta_{1}} \mathcal{H}^{\underline{s}}_{\infty}(\cup \mathcal{F}_{2^{-j}} \cap \ell) \lesssim_{\epsilon} \Delta_{1}^{\epsilon/4}. \end{displaymath}
In particular, $\mathcal{H}^{\underline{s}}_{\infty}(F \cap \ell) < c$, provided $\Delta_{1} > 0$ is sufficiently small in terms of $\epsilon$. Comparing this with \eqref{form113}, we see that $\mathcal{L} \subset \bigcup_{j} \mathcal{L}_{2^{-j}}$. Applying \eqref{form112}, and again the sub-additivity of Hausdorff content,
\begin{displaymath} c \leq \mathcal{H}^{\underline{t}}_{\infty}(\mathcal{L}) \leq \sum_{2^{-j} \leq \Delta_{1}} \mathcal{H}^{\underline{t}}_{\infty}(\mathcal{L}_{2^{-j}}). \end{displaymath}
Assuming again $\Delta_{1}$ sufficiently small in terms of $\epsilon$, this shows that $\mathcal{H}^{\underline{t}}_{\infty}(\mathcal{L}_{2^{-j}}) \geq 2^{-\epsilon j/4}$ for at least one index "$j$" with $\delta := 2^{-j} \leq \Delta_{1}$. This is what we claimed in \eqref{form114b}.

Using \cite[Proposition 3.13]{FaO}, we may find a non-empty $(\delta,\underline{t},\delta^{-\epsilon/2})$-subset of $\mathcal{L}_{\delta}$, which we keep denoting $\mathcal{L}_{\delta}$. Further, by applying Lemma \ref{l:uniformization} with a suitable $T \sim_{\epsilon} 1$, and passing to another subset, we may assume that $\mathcal{L}_{\delta}$ is $\{2^{-jT}\}_{j = 1}^{m}$ uniform, with $\delta = 2^{-mT}$.

Next, we observe that $\mathcal{L}_{\delta}$ is a $(\delta,t,\delta^{-\epsilon})$-regular (line) set, assuming that the upper bound "$\Delta_{1}$" for $\delta$ is sufficiently small in terms of $\epsilon,\Delta_{0}$. This follows with a little effort from Lemma \ref{lemma9} using the upper bound \eqref{form112}, the fact that $\mathcal{L}$ is a $\{2^{-jT}\}_{j = 1}^{m}$-uniform $(\delta,\underline{t},\delta^{-\epsilon/2})$-set, and that $\bar{t} - \underline{t} < \epsilon/8$.

Finally, it follows from \eqref{form114b}, the inequality $s - \underline{s} < \epsilon/2$, and \cite[Proposition A.1]{FaO}, that $\mathcal{F}_{\delta} \cap \ell$ contains a $(\delta,s,\delta^{-\epsilon})$-set for every $\ell \in \mathcal{L}_{\delta}$. Consequently, the hypothesis of the proposition is applicable with $\mathcal{L} = \mathcal{L}_{\delta}$ and $\mathcal{P} = \mathcal{F}_{\delta}$. The conclusion is that $|\mathcal{F}_{\delta}| \geq \delta^{-\chi}$, and \eqref{form114a} is an immediate consequence. \end{proof}

\subsection{Proof of Theorem \ref{thm:projection}} \label{s:proof-of-proj-thm} In this section, we prove Theorem \ref{thm:projection}. We will need a variant of the "uniform subset lemma", Lemma \ref{l:uniformization}, where the set $P$ (or $\mathcal{P}$) is nearly exhausted by uniform subsets:
\begin{cor}\label{cor4} For every $\epsilon > 0$, there exists $T_{0} = T_{0}(\epsilon) \geq 1$ such that the following holds for all $\delta = 2^{-mT}$ with $m \geq 1$ and $T \geq T_{0}$. Let $\mathcal{P} \subset \mathcal{D}_{\delta}$. Then, there exist disjoint $\{2^{-jT}\}_{j = 1}^{m}$-uniform subsets $\mathcal{P}_{1},\ldots,\mathcal{P}_{N} \subset \mathcal{P}$ with the properties
\begin{itemize}
\item $|\mathcal{P}_{j}| \geq \delta^{2\epsilon}|\mathcal{P}|$ for all $1 \leq j \leq N$,
\item $|\mathcal{P} \, \setminus \, (\mathcal{P}_{1} \cup \ldots \cup \mathcal{P})| \leq \delta^{\epsilon}|\mathcal{P}|$.
\end{itemize}
\end{cor}

\begin{proof} We pick $T_{0} \in \N$ satisfying $T_{0}^{-1}\log(2T_{0}) \leq \epsilon$, and let $T \geq T_{0}$. We then apply Lemma \ref{l:uniformization} once to find a $\{2^{-jT}\}_{j = 1}^{m}$-uniform subset $\mathcal{P}_{1} \subset \mathcal{P}$ with $|\mathcal{P}_{1}| \geq \delta^{\epsilon}|\mathcal{P}|$.

Assume that the sets $\mathcal{P}_{1},\ldots,\mathcal{P}_{k}$ have already been selected for some $k \geq 1$. If $|\mathcal{P} \, \setminus \, (\mathcal{P}_{1} \cup \ldots \cup \mathcal{P}_{k})| \leq \delta^{\epsilon}|\mathcal{P}|$, we set $N := k$ and the construction terminates. In the opposite case, we apply Lemma \ref{l:uniformization} to the set $\mathcal{P}' := \mathcal{P} \, \setminus \, (\mathcal{P}_{1} \cup \ldots \cup \mathcal{P}_{k})$ to find another $\{2^{-jT}\}_{j = 1}^{m}$-uniform subset $\mathcal{P}_{k + 1} \subset \mathcal{P}'$ with $|\mathcal{P}_{k + 1}| \geq \delta^{\epsilon}|\mathcal{P}'| \geq \delta^{2\epsilon}|\mathcal{P}|$. \end{proof}

We also need to borrow the following lemma from \cite[Lemma 2.3]{MR4324956}:

\begin{lemma}\label{lemma10} Let $0 \leq s \leq d$, $\delta > 0$, $C \geq 1$, and let $K \subset \R^{d}$ be a bounded set with
\begin{equation*} 
  |K|_{\delta} \leq C\delta^{-t}. 
\end{equation*}
Then, for any $L \geq 1$, there exists a disjoint decomposition $K = K_{\mathrm{good}} \cup K_{\mathrm{bad}}$ such that
\begin{itemize}
\item[\textup{(1)}] $\calH^{t}_{\infty}(K_{\mathrm{bad}}) \lesssim_{d} L^{-1}$, and
\item[\textup{(2)}] $K_{\mathrm{good}}$ satisfies
\begin{displaymath} |K_{\mathrm{good}} \cap B(x,r)|_{\delta} \lesssim_{d} CL\left(\frac{r}{\delta}\right)^{t}\quad\text{for all } x \in \R^{d}, \, r \geq \delta. \end{displaymath}
\end{itemize}
\end{lemma}

We can then state a proposition which implies Theorem \ref{thm:projection}:

\begin{proposition}\label{prop6} Let $s,u \in (0,1]$, $t \in (0,2]$, and assume that there exist $\delta_{0},\epsilon > 0$ such that the following holds for all $\delta \in (0,\delta_{0}]$.

Let $\mathcal{P}$ be a non-empty $(\delta,t,\delta^{-\epsilon})$-regular set, and let $\mathcal{E} \subset S^{1}$ be a non-empty $(\delta,s,\delta^{-\epsilon})$-set. Then, there exists $\theta \in \mathcal{E}$ such that
\begin{displaymath} |\pi_{\theta}(\mathcal{P}')|_{\delta} \geq \delta^{-u}\quad\text{for all } \mathcal{P}' \subset \mathcal{P}, \, |\mathcal{P}'| \geq \delta^{\epsilon}|\mathcal{P}|. \end{displaymath}
Then, the following also holds. Let $K \subset \R^{2}$ be a set with $\Hd K = \dim_{\mathrm{P}} K = t$, and let $E \subset S^{1}$ be a set with $\Hd E \geq s$. Then, there exists $\theta \in E$ such that $\Hd \pi_{\theta}(K) \geq u$.
\end{proposition}

\begin{remark} According to Corollary \ref{cor2}, the $\delta$-discretised hypothesis of Proposition \ref{prop6} holds for all $s \in (0,1]$, $t \in [s,2]$, and for all $0 \leq u < \min\{(s + t)/2,1\}$. Consequently, Proposition \ref{prop6} implies the following: Let
\begin{displaymath} s \in (0,1], \quad t \in [s,2], \quad \text{and} \quad 0 \leq u < \min\{(s + t)/2,1\}. \end{displaymath}
If $E \subset S^{1}$ has $\Hd E \geq s$, and $K \subset \R^{2}$ is an arbitrary set with $\Hd K = \dim_{\mathrm{P}} K = t$, then there exists $\theta \in E$ such that $\Hd \pi_{\theta}(K) \geq u$. This statement formally implies that
\begin{displaymath} \dim\{\theta \in S^{1} : \Hd \pi_{\theta}(K) < v\} \leq \min\{2v - t,0\}, \quad 0 \leq v \leq \min\{t,1\}, \end{displaymath}
as claimed in Theorem \ref{thm:projection}. \end{remark}

\begin{proof}[Proof of Proposition \ref{prop6}] Let $\underline{t} < t < \bar{t}$ with $\bar{t} - \underline{t} < \epsilon/100$. Arguing as in the proof of Proposition \ref{prop5}, we may assume that $K \subset [0,1)^{2}$, and
\begin{equation}\label{form121} \mathcal{H}_{\infty}^{\underline{t}}(K) \geq c > 0 \quad \text{and} \quad |K|_{\Delta} \leq \Delta^{-\bar{t}} \text{ for } 0 < \Delta \leq \Delta_{0}, \end{equation}
where $c > 0$ and $\Delta_{0} \in (0,\delta_{0}]$. Similarly, may assume that $\mathcal{H}^{\underline{s}}_{\infty}(E) \geq c > 0$ for some $s - \epsilon/100 \leq \bar{s} < s$.

Next, we fix $\Delta_{1} \ll \Delta_{0}$ (depending eventually on $\epsilon,c,\Delta_{0}$) and $\underline{u} < u$, and make the counter assumption that $\Hd \pi_{\theta}(K) < \underline{u}$ for all $\theta \in E$. The following objects can be located by pigeonholing (using the subadditivity of Hausdorff content):
\begin{enumerate}
\item A dyadic scale $\delta \in (0,\Delta_{1}]$.
\item A subset $E_{\delta} \subset E$ with $\mathcal{H}^{\underline{s}}_{\infty}(E_{\delta}) \geq (\log (1/\delta))^{-3}$.
\item For $\theta \in E_{\delta}$ a set $K_{\theta} \subset K$ with $\mathcal{H}^{\underline{t}}_{\infty}(K_{\theta}) \geq (\log(1/\delta))^{-3}$, and the property
\begin{equation}\label{form125} |\pi_{\theta}(K_{\theta})|_{\delta} \leq \delta^{-\underline{u}}\quad\text{for all } \theta \in E_{\delta}. \end{equation}
\end{enumerate}
Note from \eqref{form121} that $|K|_{\delta} \leq C\delta^{-\underline{t}}$ with $C := \delta^{-\epsilon/100}$. Write $L = \delta^{-\epsilon/8}$, and apply Lemma \ref{lemma10} with "$\underline{t}$" and this "$L$" to decompose $K = K_{\mathrm{good}} \cup K_{\mathrm{bad}}$, where $\mathcal{H}^{\underline{t}}_{\infty}(K_{\mathrm{bad}}) \lesssim_{d} \delta^{\epsilon/8}$, and $K_{\mathrm{good}}$ satisfies
\begin{equation}\label{form122} |K_{\mathrm{good}} \cap B(x,r)|_{\delta} \lesssim_{d} CL\left(\frac{r}{\delta}\right)^{\underline{t}} \leq \delta^{-\epsilon/7}\left(\frac{r}{\delta} \right)^{t}\quad\text{for all } x \in \R^{2}, \, r \geq \delta. \end{equation}
Let $\mathcal{K} := \{p \in \mathcal{D}_{\delta} : p \cap K_{\mathrm{good}} \neq \emptyset\}$, which is a cover of $K_{\mathrm{good}}$. Therefore,
\begin{displaymath} |\mathcal{K}| \gtrsim \delta^{-\underline{t}} \cdot \mathcal{H}^{\underline{t}}_{\infty}(K_{\mathrm{good}}) \geq \delta^{-\underline{t}} \cdot (\mathcal{H}^{\underline{t}}_{\infty}(K) - \mathcal{H}^{\underline{t}}_{\infty}(K_{\mathrm{bad}})) \gtrsim \delta^{\epsilon/100} \cdot \delta^{-t}. \end{displaymath}
Since \eqref{form122} continues to hold with "$\mathcal{K}$" in place of "$K_{\mathrm{good}}$", this shows that $\mathcal{K}$ is a $(\delta,t,\delta^{-\epsilon/6})$-set, satisfying $|\mathcal{K}|_{\Delta} \leq \delta^{-\epsilon/100} \cdot \Delta^{-t}$ for all $\delta \leq \Delta \leq \Delta_{0}$ by \eqref{form121}.

We now apply Corollary \ref{cor4} to $\mathcal{K}$, and with parameter "$\epsilon/100$", to obtain disjoint $\{2^{-jT}\}_{j = 1}^{m}$-uniform subsets $\mathcal{K}_{1},\ldots,\mathcal{K}_{N} \subset \mathcal{K}$ (with $T \sim_{\epsilon} 1$) satisfying $|\mathcal{K}_{j}| \geq \delta^{\epsilon/50}|\mathcal{K}|$, and
\begin{equation}\label{form123} |\mathcal{K} \, \setminus \, (\mathcal{K}_{1} \cup \ldots \cup \mathcal{K}_{N})| \leq \delta^{\epsilon/100}|\mathcal{K}|. \end{equation}
We record that $N \leq \delta^{-\epsilon/50}$. Furthermore, let us note that each $\mathcal{K}_{j}$ is a $\{2^{-jT}\}_{j = 1}^{m}$-uniform $(\delta,t,\delta^{-\epsilon/4})$-set, and therefore a $(\delta,t,\delta^{-\epsilon})$-regular set by Lemma \ref{lemma9}, assuming that $\delta > 0$ is small enough in terms of $\epsilon,\Delta_{0}$.

Next, recall from (3) above that $\mathcal{H}^{\underline{t}}_{\infty}(K_{\theta}) \geq (\log (1/\delta))^{-3}$ for all $\theta \in E_{\delta}$. Since on the other hand $K_{\mathrm{good}} \subset \cup \mathcal{K}$, and $\mathcal{H}^{\underline{t}}_{\infty}(K_{\mathrm{bad}}) \lesssim_{d} \delta^{\epsilon/8}$, we have
\begin{displaymath} \mathcal{H}_{\infty}^{\underline{t}}((\cup \mathcal{K}) \cap K_{\theta}) \geq \mathcal{H}_{\infty}^{\underline{t}}(K_{\theta}) - \mathcal{H}^{\underline{t}}_{\infty}(K_{\mathrm{bad}}) \gtrsim (\log (1/\delta))^{-3}\quad\text{for all } \theta \in E_{\delta}. \end{displaymath}
This implies that $\mathcal{K}_{\theta} := \{p \in \mathcal{K} : p \cap K_{\theta} \neq \emptyset\} \subset \mathcal{K}$ satisfies $|\mathcal{K}_{\theta}| \geq \delta^{\epsilon/200}|\mathcal{K}|$, because otherwise,
\begin{displaymath}  \mathcal{H}_{\infty}^{\underline{t}}((\cup \mathcal{K}) \cap K_{\theta}) \lesssim |\mathcal{K}_{\theta}| \cdot \delta^{\underline{t}} < \delta^{\epsilon/200}|\mathcal{K}| \cdot \delta^{\underline{t}} \stackrel{\eqref{form121}}{\leq} \delta^{\epsilon/200}. \end{displaymath}
Since $|\mathcal{K}_{\theta}| \geq \delta^{\epsilon/200}|\mathcal{K}|$, we may infer from \eqref{form123} that $\mathcal{K}_{\theta}$ must have large intersection with one of the $(\delta,t,\delta^{-\epsilon})$-regular families $\mathcal{K}_{1},\ldots,\mathcal{K}_{N} \subset \mathcal{K}$. More precisely, for each $\theta \in E_{\delta}$, there exists an index $j = j(\theta) \in \{1,\ldots,N\}$ such that
\begin{equation}\label{form124} |\mathcal{K}_{\theta} \cap \mathcal{K}_{j}| \gtrsim \tfrac{1}{N}|\mathcal{K}_{\theta}| \geq \delta^{\epsilon/50 + \epsilon/200}|\mathcal{K}| \geq \delta^{\epsilon}|\mathcal{K}_{j}|. \end{equation}
By a final application of the pigeonhole principle, and since $\mathcal{H}^{\underline{s}}_{\infty}(E_{\delta}) \geq (\log(1/\delta))^{-3}$ by (2) above, we may choose a fixed index $j \in \{1,\ldots,N\}$ such that the set
\begin{displaymath} E_{\delta}^{j} := \{\theta \in E_{\delta} : j(\theta) = j\} \end{displaymath}
satisfies $\mathcal{H}_{\infty}^{\underline{s}}(E^{j}_{\delta}) \geq \delta^{\epsilon/2}$. Recalling that $\underline{s} \geq s - \epsilon/4$, we may finally select a $(\delta,s,\delta^{-\epsilon})$-subset $\mathcal{E} \subset E_{\delta}^{j}$, for this specific index "$j$".

Let $\mathcal{P} := \mathcal{K}_{j}$, and let $\mathcal{P}_{\theta} := \mathcal{K}_{\theta} \cap \mathcal{K}_{j}$ for $\theta \in \mathcal{E}$. From \eqref{form124}, we infer that $|\mathcal{P}_{\theta}| \geq \delta^{\epsilon}|\mathcal{P}|$ for all $\theta \in \mathcal{E}$, and from \eqref{form125} we infer that
\begin{displaymath} |\pi_{\theta}(\mathcal{P}_{\theta})|_{\delta} < \delta^{-u}\quad\text{for all } \theta \in \mathcal{E}. \end{displaymath}
Since $\mathcal{P}$ is $(\delta,t,\delta^{-\epsilon})$-regular, and $\mathcal{E}$ is a $(\delta,s,\delta^{-\epsilon})$-set, this contradicts the hypothesis of the proposition. The proof is complete. \end{proof}

\def\cprime{$'$}

\bibliographystyle{plain}
\bibliography{references}

\end{document}